\documentclass{elsart}

\usepackage{amssymb}
\usepackage{color}
\usepackage{graphicx}\usepackage{subfigure}
\usepackage{amsmath}
\usepackage{amsfonts}

\usepackage[title]{appendix}
\usepackage{latexsym}
\usepackage{psfrag}
\usepackage{multirow}

\allowdisplaybreaks

\usepackage{paralist}

\def\dfr#1#2{\displaystyle{\frac{#1}{#2}}}

\renewcommand{\vec}[1]{\mbox{\boldmath \small $#1$}}

 \def\bga{\begin{array}} \def\eda{\end{array}}

\def\dfr#1#2{\displaystyle{\frac{#1}{#2}}} 

\def\Ronu#1{\uppercase\expandafter{\romannumeral#1}}

\newtheorem{Def}{Definition}[section]
\newtheorem{example}{Example}[section]
\newtheorem{lemma}{Lemma}[section]
\newtheorem{remark}{Remark}[section]
\numberwithin{equation}{section}
\numberwithin{figure}{section}
\numberwithin{table}{section}
\numberwithin{thm}{section}

\newenvironment{proof}[1][Proof]{\begin{trivlist}
\item[\hskip \labelsep {\bfseries #1}]}{\end{trivlist}}
\renewcommand{\qed}{\hfill \nobreak \ifvmode \relax \else
      \ifdim\lastskip<1.5em \hskip-\lastskip
      \hskip1.5em plus0em minus0.5em \fi \nobreak
      \vrule height0.75em width0.5em depth0.25em\fi}

\begin{document}
\begin{frontmatter}
\title{High-order accurate physical-constraints-preserving finite difference WENO schemes for special relativistic hydrodynamics}
\author[WKL]{Kailiang Wu},
\ead{wukl@pku.edu.cn}
\address[WKL]{LMAM, School of Mathematical Sciences, Peking University,
	Beijing 100871, China}
\author[THZ]{Huazhong Tang\corauthref{cor}}
\corauth[cor]{Corresponding author.}
\ead{hztang@math.pku.edu.cn}
\address[THZ]{HEDPS, CAPT  \& LMAM, School of Mathematical Sciences, Peking University,
	Beijing 100871, China}

\date{}
\maketitle

\begin{abstract}
The paper develops high-order accurate physical-constraints-preserving
finite difference WENO schemes
for special relativistic hydrodynamical (RHD) equations,
built on the local Lax-Friedrich splitting,
the WENO reconstruction, the physical-constraints-preserving
flux limiter, and the high-order strong stability preserving time discretization.
They are extensions of the positivity-preserving finite difference WENO schemes for the non-relativistic Euler equations \cite{Hu2013}.
However, developing physical-constraints-preserving methods  for the RHD system becomes much more difficult than the non-relativistic case
because of the strongly coupling between the RHD equations,
 no explicit formulas of the primitive variables and the flux vectors with respect to the conservative vector,
and  one more physical constraint for the fluid velocity
in addition to the positivity of the rest-mass density and the pressure.
The key is to prove the convexity and other properties
of the admissible state set and discover a concave function with respect to
the conservative vector instead of the pressure which
is  an important ingredient to enforce the positivity-preserving property for
the non-relativistic case.

Several one- and two-dimensional numerical examples  are used to demonstrate  accuracy, robustness, and effectiveness of the proposed physical-constraints-preserving schemes in solving  RHD  problems with large Lorentz factor, or strong discontinuities, or low
rest-mass density or pressure etc.

\end{abstract}

\end{frontmatter}


\section{Introduction}
\label{sec:intro}

The paper is concerned with developing high-order accurate numerical methods for special relativistic hydrodynamical (RHD) equations.
In the laboratory frame, the $(d+1)$-dimensional space-time RHD equations
may be   written into a system of conservation laws as follows
\begin{equation}\label{eqn:coneqn3d}
\displaystyle\frac{\partial \vec{U}}{\partial t} +
\sum^d_{i=1}\frac{\partial \vec{F}_i(\vec{U})}{\partial x_i}=0,
\end{equation}
where   $\vec{U}$ and $\vec{F}_i$ are the conservative  vector and the flux
in the
$x_i$-direction, respectively,  defined by
\begin{align*}
\vec{U} =& \big(D, m_1, \cdots, m_d, E\big)^T,\\
\vec{F}_i =& \big(Dv_i, m_1 v_i + p \delta_{1,i},  \cdots,  m_d v_i  + p \delta_{d,i}, m_i\big)^T, ~~i=1,\cdots, d,
\end{align*}
with the mass density $D=\rho W$,
the momentum density vector  $\vec m=Dh W \vec v$, and the energy density $E=DhW-p$, respectively.
Here $\rho$, $p$, and $\vec v= (v_1,  \cdots,  v_d)^T$ denote the rest-mass density, the kinetic pressure, and
the fluid velocity respectively,
 $W=1/\sqrt{1-v^2}$ is the Lorentz factor with $v=(v_1^2+\cdots +v_d^2)^{1/2}$,
and
$h$ denotes the specific enthalpy
 defined by
 \begin{equation}
 \label{eq:h}
 h = 1 + e + \dfr{p}{\rho},
 \end{equation}
with units in which the speed of light is equal to one,
and $e$ is the specific internal energy.

 The  system \eqref{eqn:coneqn3d} has taken into account the relativistic description of fluid dynamics
 where the fluid flow is at nearly speed of light in vacuum and appears  in investigating numerous astrophysical phenomena
from stellar to galactic scales, e.g.   formation of black holes, coalescing
neutron stars, core collapse super-novae, X-ray binaries, active galactic nuclei, super-luminal jets
and gamma-ray bursts, etc.
However, it still involves highly nonlinear equations due to the Lorentz factor
so that its analytic treatment is extremely difficult.
A powerful and primary approach to improve our understanding of the
physical mechanisms in RHDs is through numerical simulations.
Comparing to the non-relativistic case, the numerical difficulties
are coming from strongly nonlinear coupling between the RHD equations, which leads to
no explicit expression of  the primitive variables
$\vec V=(\rho,\vec v,p)^T$ and the flux vector $\vec F_i$ in terms of $\vec U$,
and some physical constraints such as $\rho>0$, $p>0$,  $E\geq D$,
and $1> v$ etc.
Its numerical study  did not attract considerable attention  until 1990s.

The first attempt to numerically solve the RHD equations was made by  using a finite difference method
with the artificial viscosity technique in Lagrangian or Eulerian coordinates, see  \cite{May-White1966,May-White1967,Wilson:1972,Wilson:1979}.
After that, various modern shock-capturing methods
were gradually developed for the special RHDs since 1990s, e.g.
 the HLL (Harten-Lax-van Leer-Einfeldt) method \cite{Schneider:1993},
 the two-shock approximation solvers \cite{Balsara:1994,DaiWood:1997}, the flux corrected transport method \cite{Duncan:1994}, the Roe solver \cite{EulderinkMel:1995}, the HLLC (Harten-Lax-van Leer-Contact) approximate Riemann solver \cite{MignoneBodo:2005},
the  flux-splitting method based on  the spectral decomposition \cite{DonatFont:1998},
 Steger-Warming flux vector splitting method \cite{ZhaoHeTang2014},
and  the kinetic schemes  \cite{YangBeam1997,Kunik2004,Qamar2004}.
Besides those, high-order accurate schemes for the RHD system were also studied,
e.g. ENO (essentially non-oscillatory) and weighted ENO methods \cite{DolezalWong:1995,ZannaBucciantini:2002,Tchekhovskoy2007},
the piecewise parabolic methods  \cite{Marti3,Mignoneetal:2005},
the space-time conservation element and solution element method \cite{Qamar2012},
and the discontinuous Galerkin (DG) method \cite{Radice2011},
the Runge-Kutta DG methods with WENO (weighted ENO) limiter  \cite{ZhaoTang2013},
the direct Eulerian GRP schemes  \cite{YangHeTang2011,YangTang2012,WuYangTang2014},
the adaptive moving mesh methods \cite{HePeng2011,HeTang2011,HeTang2012},
and genuinely multi-dimensional finite volume local evolution Galerkin method \cite{WuTang2014}.
The readers are also referred to the early review articles
 \cite{Ibanez-Marti1999,MME:2003,Font2008}.

The above existing works do not preserve the positivity of the rest-mass density and the pressure and  the bounds of the fluid velocity. Although they have  been  used to simulate some RHD flows successfully, there exists  the big risk of failure when they are applied to RHD problems with large Lorentz factor, or low density or pressure, or strong discontinuity,
because the negative density or pressure, or the larger velocity  than the speed of light
 may be obtained so that the calculated eigenvalues of the Jacobian matrix   become imaginary, in other words,
 the discrete problem becomes ill-posed.
In practice,  the nonphysical numerical solutions are usually simply replaced with
a ``close'' and ``physical'' one
by performing recalculation with more diffusive schemes and smaller CFL number until the numerical solutions become physical, see e.g. \cite{ZhangMacfadyen2006,Hughes2002}.
Obviously,  such approch is not  scientifically reasonable to a certain extent,
and it is of great significance to develop  high-order accurate numerical schemes, whose solutions
 satisfy the intrinsic physical constraints.
 Recently,  there exist some works on
   the maximum-principle-satisfying schemes for scalar hyperbolic conservation law
  \cite{zhang2010,zhang2012a},
 the positivity-preserving schemes
 for the non-relativistic Euler equations
 with or without source terms \cite{zhang2012,zhang2010b,zhang2011},
the positivity-preserving  well-balanced schemes  for the shallow water equations \cite{Xing2010},
the positivity preserving semi-Lagrangian DG method for the Vlasov-Poisson system \cite{Qiug2011},
and Lagrangian method with positivity-preserving limiter for multi-material compressible flow \cite{Cheng2014}.
 A class of the parametrized maximum principle preserving and positivity-preserving flux limiters
were also  well-developed  via decoupling some linear or nonlinear constraints
for the high order accurate schemes for
scalar hyperbolic conservation laws \cite{Xu_MC2013,Liang2014},
nonlinear convection-dominated diffusion equations \cite{JiangXu2013},
compressible Euler equations \cite{XiongQiuXu2014} and
ideal magnetohydrodynamical equations \cite{Christlieb} as well as simulating incompressible flows \cite{XiongQiuXu2013}.
A survey of the maximum-principle-satisfying or positivity-preserving
 high-order schemes is presented in \cite{zhang2011b}.

 The aim of the paper is to do the first attempt in the aspect of
  developing the high-order accurate
 physical-constraints-preserving finite difference schemes
 for special RHD equations \eqref{eqn:coneqn3d}.
Such attempt is nontrivial in comparison of the non-relativistic case,
because of the strongly nonlinear coupling between the RHD equations due to the Lorentz factor,
 no explicit formulas of $\vec V$ and   $\vec F_i$ with respect to $\vec U$,
 and  one more physical constraint for the fluid velocity
 in addition to the positivity of the rest-mass density and the pressure.
 The key will be to prove the convexity and other properties
 of the admissible state set and discover a concave function with respect to
 the conservative vector $\vec U$ instead of the pressure, which
 is  an important ingredient to enforce the positivity-preserving property for
 the non-relativistic case.
The paper is organized as follows.
Section \ref{sec:GoEq} discusses the admissible state set and its properties
of the special RHD equations.
 They play a pivotal role in studying the physical-constraints-preserving property
 of numerical schemes.
Section \ref{sec:scheme} presents the
high-order accurate  physical-constraints-preserving  finite difference WENO
schemes for the RHD equations.
Section \ref{sec:1Dscheme} considers detailedly one-dimensional case
 with  spatial discretization in Section \ref{sec:reviweWENO}
and time discretization in Section \ref{sec:Time}.
Section \ref{sec:2Dscheme}  gives the 2D extension of the above scheme
and apply it to the 2D  axisymmetric case.
Section \ref{sec:experiments}
gives several 1D and 2D numerical experiments to demonstrate   accuracy,
robustness and effectiveness of the proposed schemes for relativistic problems with large Lorentz factor,
 or strong discontinuities, or low
rest-mass density or pressure etc.
Section \ref{sec:conclude} concludes the paper with several remarks.


\section{Properties of RHD equations}
\label{sec:GoEq}

This section discusses the admissible state set and its properties for
the RHD equations \eqref{eqn:coneqn3d}.
Throughout the paper, the equation of state (EOS) will be restricted  to the  $\Gamma$--law
\begin{equation}\label{gamma-law}
p = (\Gamma-1)\rho e,
\end{equation}
with the adiabatic index $\Gamma \in (1,2]$.
Such restriction on $\Gamma$  is reasonable
 under the compressibility assumptions, see
  \cite{Cissoko1992}.

The RHD system \eqref{eqn:coneqn3d} with \eqref{gamma-law} is identical  to the $d$-dimensional non-relativistic Euler equations in the formal structure, and also satisfies  the
rotational invariance and the homogeneity as well as the hyperbolicity in time, see \cite{ZhaoHeTang2014}.
The momentum equations in \eqref{eqn:coneqn3d} are only with a Lorentz-contracted momentum
density replacing $\rho v_i$ in the non-relativistic Euler equations.
When the fluid velocity is much smaller than the speed of light (i.e.
$ v\ll 1$) and
the velocity of the internal (microscopic) motion of the fluid particles is small, the system \eqref{eqn:coneqn3d}
reduces to the non-relativistic Euler equations.
The $(d+2)$ real eigenvalues of the Jacobian matrix  $\vec A_i (\vec U) = {\partial \vec F_i (\vec U)}/{\partial \vec U}$
for  \eqref{eqn:coneqn3d} with \eqref{gamma-law} are
\begin{align*}
 \lambda _i^{(1)}  &= \frac{{v_i (1 - c_s^2 ) - c_s W^{ - 1} \sqrt {1 - v_i^2  - (v^2  - v_i^2 )c_s^2 } }}{{1 - v^2 c_s^2 }}, \\
 \lambda _i^{(2)}  &= \cdots  = \lambda _i^{(d+1)}  = v_i , \\
 \lambda _i^{(d+2)}  &= \frac{{v_i (1 - c_s^2 ) + c_s W^{ - 1} \sqrt {1 - v_i^2  - (v^2  - v_i^2 )c_s^2 } }}{{1 - v^2 c_s^2 }},
\end{align*}
for $i=1,\cdots, d$, where the local sound speed $c_s = \sqrt{\frac{\Gamma p}{\rho h}}$.
However, relations between the laboratory quantities ($D$,
$m_i$, and  $E$)
 and the quantities in the local rest frame ($\rho$,
  $v_i$, and  $e$) introduce a strong coupling
 between the hydrodynamic equations and pose more additional
 numerical difficulties than the non-relativistic case.
 For example, the flux vectors $\vec F_i$ and the primitive
 variable vector $\vec V:= (\rho, \vec v, p)^T$
 can not be formulated in  explicit forms of the conservative vector $\vec U$, and
 some  constraints (e.g. $\rho> 0$, $p> 0$, and $ 1>v$, etc.) should be fulfilled by
 the physical solution $\vec U$.

\begin{Def}
 The
{\em set of (physically) admissible states} of the RHD equations
\eqref{eqn:coneqn3d} with \eqref{gamma-law} is defined by
\begin{equation}\label{EQ-adm-set01}
{\mathcal G} = \left\{ { \left. \vec U = (D,\vec m,E)^T \right| {\rho (\vec U) > 0,~p(\vec U) > 0,~ 1>v(\vec U)  } } \right\}.
\end{equation}
\end{Def}

Such definition is very natural and intuitive  but
unpractical when giving the value of the conservative vector $\vec U$,
because there is no explicit expression of the primitive
variable $\vec V= (\rho, \vec v, p)^T$
in terms to  $\vec U$ for the system \eqref{eqn:coneqn3d}.
  It is this feature that makes the discussions on the admissible state and the
physical-constraints-preserving schemes presented later nontrivial or even challenging for RHD equations
\eqref{eqn:coneqn3d}.
In practical computations, for given $\vec U=(D,
\vec m, E)^T$,
one has to (iteratively) solve
a nonlinear algebraic equation such as
 \begin{equation}\label{eq:solveP}
  E + p = DW + \displaystyle\frac{\Gamma}{\Gamma-1} pW^2,
 \end{equation}
by any standard root-finding algorithm  to get the pressure $p(\vec U)$.  Then $v_i$ and $\rho$ are sequentially calculated by
\begin{equation}\label{eq:solveVRHO}
v_i(\vec U) = \frac{m_i}{{E + p(\vec U)}},\quad \rho (\vec U) = D\sqrt {1 -  v^2(\vec U)}.
\end{equation}
Note that the Lorentz factor $W$ in \eqref{eq:solveP}
has been rewritten into $\left(1- m^2/(E+p)^2\right)^{-1/2} $ with $ m:= (m_1^2+\cdots+m_d^2)^{1/2}$.
Existence of the unique positive solution to the pressure equation \eqref{eq:solveP} is given in the proof of Lemma \ref{lam:equDef}.
It is worth mentioning that different from the non-relativistic case,
such positive solution $p(\vec U)$
is not a concave function of $\vec U$ generally, see Fig. \ref{fig:phi_lambda}.
\begin{figure}[htbp]
  \centering
  {\includegraphics[width=0.56\textwidth]{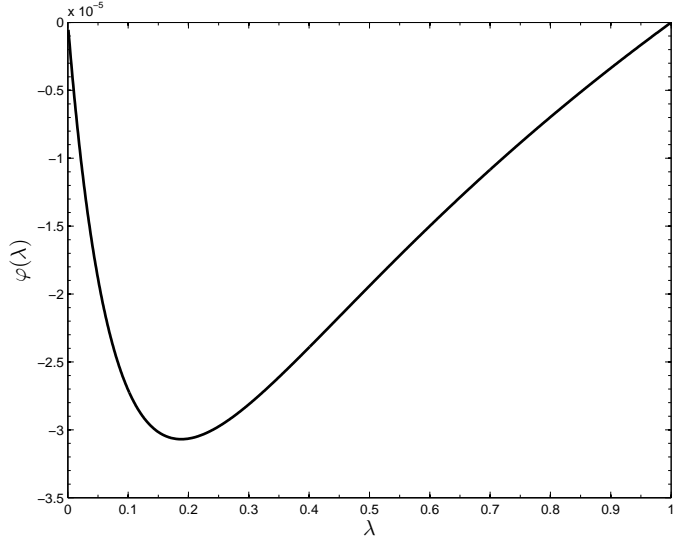}}
  \caption{\small The function
  $\varphi (\lambda):=p(\lambda \vec U^{(1)}  + (1 - \lambda )\vec U^{(0)} ) - \lambda p( \vec U^{(1)}) -(1 - \lambda )p(\vec U^{(0)} )$ with $\vec U^{(0)}=(2,1.2,8)^{\rm T} \in {\mathcal G}$ and $\vec U^{(1)}=(2,5,35)^{\rm T} \in {\mathcal G}$.
  The value of $\varphi (\lambda)$ is always
  less than zero when $\lambda \in (0,1)$.
 }
  \label{fig:phi_lambda}
\end{figure}

A practical and equivalent definition of  ${\mathcal G}$ is given as follows.

\begin{lemma}\label{lam:equDef}
The admissible set ${\mathcal G}$ defined in \eqref{EQ-adm-set01}
is equivalent to the following set
\begin{equation}\label{eq:newG}
{\mathcal G}_1 = \left\{ {\vec U = (D,\vec m,E)^T \left|
{D > 0,~q(\vec U):=E-\sqrt {D^2  +   m^2 } >0} \right.} \right\}.
\end{equation}
\end{lemma}
\begin{proof}
{\tt (i):  Prove that $\vec U\in {\mathcal G}_1$ when $\vec U\in {\mathcal G}$.
}
When $\vec U = (D,\vec m,E)^T$ satisfy the constraints $\rho (\vec U) > 0$, $p(\vec U) > 0$, and $v(\vec U) < 1$, it is not difficult to
show 
\[
D = \frac{\rho }{{\sqrt {1 -   v^2(\vec U) } }} > 0, \quad E = \frac{{\rho h}}{{1 -  v^2(\vec U) }} - p > \rho h - p \overset{\eqref{eq:h}}{=} \rho (1 + e) > 0,
\]
and
\begin{align*}
 E^2  -& \left( {D^2  +  m^2 } \right) = \left( {\frac{{\rho h}}{{1 - v^2 }} - p} \right)^2  - \frac{{\rho ^2 }}{{1 - v^2 }} - \left( {\frac{{\rho hv}}{{1 - v^2 }}} \right)^2  \\
 &= \left( {\frac{{\rho h}}{{1 - v^2 }}} \right)^2  +p^2 -2p{\frac{{\rho h}}{{1 - v^2 }}}
  - \frac{{\rho ^2 }}{{1 - v^2 }} - \left( {\frac{{\rho hv}}{{1 - v^2 }}} \right)^2
 \\
  &= \frac{1}{{1 - v^2 }}\left[ {\left( {\rho h - p} \right)^2  - \rho ^2  - p^2 v^2 } \right]
  \overset{\eqref{eq:h}}{=}
  \frac{1}{{1 - v^2 }}\left[ {\rho ^2 \left( {1 + e} \right)^2  - \rho ^2  - p^2 v^2} \right] \\
    &\overset{v< 1}{>}
    \frac{1}{{1 - v^2 }}\left[ {\rho ^2 \left( {1 + e} \right)^2  - \rho ^2  - p^2 } \right]
  \overset{\eqref{gamma-law}}{=}
  \frac{{\rho^2 e}}{{1 - v^2 }}
     \big( 2+ e\Gamma (2 - \Gamma ) \big)  \overset{\Gamma \in (1,2]}{>} 0.
\end{align*}
Thus $q(\vec U)=E- \sqrt{D^2  +  m^2 }>0$ so that $\vec U\in {\mathcal G}_1$.

{\tt (ii): Prove that $\vec U\in {\mathcal G}$ when $\vec U\in {\mathcal G}_1$.
}
 Consider the function of $p$ defined by
\[
\Phi (p) := \frac{{ m^2 }}{{E + p}} + D\sqrt {1 - \frac{{ m^2 }}{{(E + p)^2 }}}  + \frac{p}{{\Gamma  - 1}} - E, \quad  p\in [0,+\infty),
\]
with $\vec U$  satisfying that $D>0$  and $E>\sqrt{D^2+ m^2}$.
Obviously, $ \Phi (p) \in C^1[0,+\infty)$, and
\[
\Phi '(p) = \frac{1}{{\Gamma  - 1}} - \frac{{ m^2 }}{{\left( {E + p} \right)^2 }}\left( {1 - \frac{D}{{\sqrt {(E + p)^2  -  m^2 } }}} \right) \ge 1 - \frac{{ m^2 }}{{\left( {E + p} \right)^2 }} > 0,\quad \forall p \in \left[ {0, + \infty } \right),
\]
when
$E>\sqrt{D^2+ m^2}$ and $\Gamma\in (1,2]$.
Thus  $\Phi (p)$ is a strictly monotonically increasing function of $p$ in the interval $[0,+\infty)$.
On the other hand, one has
\[
\Phi (0) = \frac{{ m^2 }}{E} + D\sqrt {1 - \frac{{ m^2 }}{{E^2 }}}  - E = \left( {D - \sqrt {E^2  -  m^2 } } \right)\sqrt {1 - \frac{{  m  ^2 }}{{E^2 }}}  < 0,
\]
and $\mathop {\lim }\limits_{p \to  + \infty } \Phi (p) =  + \infty$
because
\[
\mathop {\lim }\limits_{p \to  + \infty } \frac{{\Phi (p)}}{p} = \frac{1}{{\Gamma  - 1}}>0.
\]
Thanks to the {\em intermediate value theorem} and the monotonicity
of $\Phi (p)$, there
exists  a unique positive solution to the equation $\Phi (p)=0$,
which is equivalent to the equation \eqref{eq:solveP}.
Denote this positive solution by $p(\vec U)$.
Substituting  this 
into \eqref{eq:solveVRHO} may give
\[
  {v(\vec U)}  = \frac{{ m}}{{E + p(\vec U)}} <  \frac{{ m}}{E} < 1,\quad
\rho (\vec U) = D\sqrt {1 -  { v^2(\vec U)}   }  > 0,
\]
by using the conditions that $D>0$ and $E>\sqrt{D^2+ m^2}$.
Thus $\vec U \in {\mathcal G}$ and the proof is completed. \qed  %
\end{proof}

\begin{remark}
Comparing to ${\mathcal G}$ defined in  \eqref{EQ-adm-set01},
the constraints on conservative variables in the set ${\mathcal G}_1$
is much easier to be verified when the value of $\vec U$ is given.
\end{remark}

With the help of equivalence of the admissible state set ${\mathcal G} $ in Lemma \ref{lam:equDef}, we can further prove that it is a convex set. 

\begin{lemma}
\label{lam:convex}
The admissible set ${\mathcal G}_1$ is a convex set.
\end{lemma}

\begin{proof}
To show that the set ${\mathcal G}_1$ is convex,
 one has to prove that for  all $\lambda$  in the interval $[0,1]$,
and all  $\vec U^{(0)}  = (D^{(0)} ,\vec m^{(0)} ,E^{(0)} )^T$ and
 $\vec U^{(1)}  = (D^{(1)} ,\vec m^{(1)} ,E^{(1)} )^T$ in the
 set ${\mathcal G}_1$, the point
$\lambda \vec U^{(1)}  + (1 - \lambda )\vec U^{(0)}  = :(D^{(\lambda )} ,\vec m^{(\lambda )} ,E^{(\lambda )} )^T  \in {\mathcal G}_1$ also belongs to ${\mathcal G}_1$.

Because $\vec U^{(0)}, \vec U^{(1)}\in  {\mathcal G}_1$,
one has
\[
D^{(\lambda )}  = \lambda D^{(1)}  + (1 - \lambda )D^{(0)}  > 0,
\]
and
\begin{align*}
 E^{(\lambda )}  &= \lambda E^{(1)}  + (1 - \lambda )E^{(0)}
  \\
  &> \lambda \sqrt {\left( {D^{(1)} } \right)^2  +  \sum\limits_{i = 1}^d \left(m_i^{(1)} \right)^2 }
  + (1 - \lambda )\sqrt {\left( {D^{(0)} } \right)^2  + \sum\limits_{i = 1}^d
  	 {\left( {m_i^{(0)} } \right)^2 } }  \\
  &\ge \sqrt {\left[ {\lambda D^{(1)}  + (1 - \lambda )D^{(0)} } \right]^2  + \sum\limits_{i = 1}^d {\left[ {\lambda \left| {m_i^{(1)} } \right| + (1 - \lambda )\left| {m_i^{(0)} } \right|} \right]^2 } }  \\
  &\ge \sqrt {\left( {D^{(\lambda )} } \right)^2  + \sum\limits_{i = 1}^d {\left( {m_i^{(\lambda )} } \right)^2 } }.
\end{align*}
Here  the Minkowski inequality
 for both vectors $(\lambda D^{(1)}, \lambda m_1^{(1)}, \cdots,\lambda m_d^{(1)})$
and $((1-\lambda) D^{(0)}, (1-\lambda) m_1^{(0)},\cdots,(1-\lambda) m_d^{(0)})$
and the triangle inequality
$|\lambda m_i^{(1)}+(1-\lambda) m_i^{(0)}|\leq
\lambda| m_i^{(1)}|+(1-\lambda) | m_i^{(0)}| $
have been used respectively.
Thus $\lambda \vec U^{(1)}  + (1 - \lambda )\vec U^{(0)}  \in {\mathcal G}_1$ for all $\lambda  \in [0,1]$. The proof is completed.
\qed
\end{proof}

\begin{remark} \label{rem:q}
The proof of Lemma \ref{lam:convex} implies that the function $q(\vec U)$ defined in \eqref{eq:newG}
is  concave. Moreover,  the function $q(\vec U)$ is also
Lipschitz continuous with respect to  $\vec U$ and satisfies
\begin{align}\nonumber
\left| {q(\vec U^{(1)} ) - q(\vec U^{(0)} )} \right|
&
 \le \left| {E^{(1)}  - E^{(0)} } \right| +  \left| {\sqrt {\left( {D^{(1)} } \right)^2  +
 \left( m^{(1)}\right) ^2 }
 - \sqrt {\left( {D^{(0)} } \right)^2  + \left( m^{(0)} \right)^2 } } \right|  \\
 \nonumber
&
 \le \left| {E^{(1)}  - E^{(0)} } \right| +   {\sqrt {\left( {D^{(1)}  - D^{(0)} } \right)^2  + \sum_{i=1}^d \left({ m^{(1)}_i  -  m^{(0)}_i } \right) ^2 } }
\\ \nonumber
&
 \le \sqrt {2\left[ {\left( {D^{(1)}  - D^{(0)} } \right)^2  + \sum_{i=1}^d
 \left({ m^{(1)}_i  -  m^{(0)}_i } \right) ^2  + \left( {E^{(1)}  - E^{(0)} } \right)^2 } \right]}
 \\ \label{EQ:Lip}
&= \sqrt 2 \left\| {\vec U^{(1)}  - \vec U^{(0)} } \right\|,
\end{align}
 for any $\vec U^{(0)}  = (D^{(0)} ,\vec m^{(0)} ,E^{(0)} )^T  \in \mathbb{R}^{d+2}$
and  $\vec U^{(1)}  = (D^{(1)} ,\vec m^{(1)} ,E^{(1)} )^T  \in \mathbb{R}^{d+2}$,
where 
the inequality $a+b \le \sqrt{2(a^2+b^2)}$ has been used.
The concavity and Lipschitz continuity of $q(\vec U)$ will play a pivotal role in designing our physical-constraints-preserving schemes for
the RHD equations \eqref{eqn:coneqn3d}.
\end{remark}

By means of  the convexity of ${\mathcal G}$,
the following properties of ${\mathcal G}$ can further be verified.

\begin{lemma}
\label{lam:propertyG}
Assume  $\vec U \in {\mathcal G}_1$, then
\begin{enumerate}
  \item[(i)] $\lambda \vec U \in {\mathcal G}_1$, for all $\lambda  > 0$.
  \item[(ii)]  $\vec T \vec U \in {\mathcal G}_1$,
  where $\vec T=\mbox{diag}\{1,\vec T_{d\times d},1\}$ and
  $\vec T_{d\times d}$ is
  the $d\times d$ rotational matrix.
 \item[(iii)] $\vec U \pm {\alpha}^{-1}{{\vec F_i (\vec U)}} \in {\mathcal G}_1$
 for all real number $\alpha\geq \varrho _i$,  $i=1, \cdots, d$,
 where $\varrho _i$ is  the spectral radius of the
  Jacobian matrix  $\vec A_i(\vec U)$,  i.e.
\[
\varrho _i  := \frac{{\left| {v_i } \right|(1 - c_s^{ 2} ) + c_s W^{ - 1} \sqrt {1 - v_i^2  - (v^2  - v_i^2 )c_s^2 } }}{{1 - v^2 c_s^2 }}.
\]
\end{enumerate}
\end{lemma}

\begin{proof}
The verification of   first two properties are omitted here because
they can be  directly and easily verified via
the definition  of ${\mathcal G}_1$.

The following  will prove the third conclusion ({\em iii})
that   if $\vec  U\in {\mathcal G}_1$, then
$\left( {D^ \pm,\vec m^ \pm, E^ \pm  } \right)^T
=\vec  U^\pm:=\vec U \pm {\alpha }^{-1}{{\vec F_1 (\vec U)}} \in {\mathcal G}_1$.
It is nontrivial and requires several techniques thanks to
	 no explicit formulas of the flux $\vec F_1(\vec U)$
in terms of  $\vec U$.

For all $\vec U \in {\mathcal G}_1$ and the perfact gas \eqref{gamma-law} with $\Gamma\in (1,2]$,
one has 
\begin{align}\label{EQ-thz02}
 c_s^2  =  \frac{{\Gamma p}}{{\rho h}} &= \frac{{\Gamma p}}
 {{\rho  + \frac{\Gamma }{{\Gamma  - 1}}p}} < \Gamma  - 1 \le 1,
 \end{align}
 so that the inequalities
 $$
W^{ - 1} \sqrt {1 - v_i^2  - (v^2  - v_i^2 )c_s^2 }\geq W^{ - 1} \sqrt {1 - v^2}=1-v^2,
 $$
 and
 $$
 \varrho_1\geq \frac
    {{|v_1|} (1 - c_s^{ 2} ) + c_s (1 - v^2 )}{  1 - v^2 c_s^2  },
 $$
 hold. They  imply
 \begin{align*}
 1 - \frac{{\left| {v_1 } \right|}}{ \varrho _1 }
 &\ge 1 - \frac{ \left| v_1  \right| (1 - v^2 c_s^2 ) }
   {{|v_1|} (1 - c_s^{ 2} ) + c_s (1 - v^2 )}
  = \frac{{W^{ - 2} c_s \left( {1 - \left| {v_1 } \right|c_s } \right)}}{{\left| {v_1 } \right|\left( {1 - c_s^2 } \right) + c_s (1 - v^2 )}} \\
 & \ge \frac{{W^{ - 2} c_s \left( {1 - \left| {v_1 } \right|c_s } \right)}}{{\left| {v_1 } \right|\left( {1 - c_s^2 } \right) + c_s (1 - v_1^2 )}} = \frac{{W^{ - 2} c_s }}{{\left| {v_1 } \right| + c_s }} > \frac{{W^{ - 2} c_s }}{{1 + c_s }},
\end{align*}
so that  one has
\begin{align}
{1 \pm \frac{{v_1 }}{\alpha }}  \ge  {1 - \frac{{\left| {v_1 } \right|}}{\alpha }}
 \ge   {1 - \frac{{\left| {v_1 } \right|}}{{\varrho _1 }}}
 > \frac{{W^{ - 2} c_s }}{{1 + c_s }}> 0.\label{EQ-thz01}
\end{align}
Thus one gets
\begin{align*}
D^ \pm  &= D\left( {1 \pm \frac{{v_1 }}{\alpha }} \right) > 0,
\\
 E^ \pm  & = E \pm \frac{{m_1 }}{\alpha }
 =(\rho hW^2-p)\pm \rho h W^2\frac{{v_1}}{\alpha }
  \overset{\eqref{EQ-thz01}}{>} \frac{{\rho hc_s }}{{1 + c_s }} - p
  \overset{\eqref{EQ-thz02}}{=} p\left( {\frac{\Gamma }{{c_s (1 + c_s )}} - 1} \right)
 \\ & \overset{\eqref{EQ-thz02}}{>} p\left( {\frac{\Gamma }{{\Gamma  - 1 + \sqrt {\Gamma  - 1} }} - 1} \right) = p\left( {\frac{{1 - \sqrt {\Gamma  - 1} }}{{\Gamma  - 1 + \sqrt {\Gamma  - 1} }}} \right)
  \overset{\Gamma \in (1,2]}{\ge}
   0,
\end{align*}
and
\begin{align}\nonumber
& \left( {D^ \pm  } \right)^2  + (\vec m^ \pm  )^2  - \left( {E^ \pm  } \right)^2
\\ \nonumber
&  = \left( {1 \pm \frac{{v_1 }}{\alpha }} \right)^2 \left( {D^2  +  m^2  - E^2 }
 \right) \pm \frac{{2p}}{\alpha }\left( {m_1  - Ev_1 } \right)\left( {1 \pm \frac{{v_1 }}{\alpha }} \right) + \frac{{p^2 }}{{\alpha ^2 }}(1 - v_1^2 )
 \\ \nonumber
&  = \left( {1 \pm \frac{{v_1 }}{\alpha }} \right)^2 W^2 \left[ {\rho ^2  + p^2 v^2  - \left( {\rho h - p} \right)^2 } \right] + p^2 \left( {1 \pm \frac{{v_1 }}{\alpha }} \right)^2  - p^2  + \frac{{p^2 }}{{\alpha ^2 }}
\\ \nonumber
&  = \left( {1 \pm \frac{{v_1 }}{\alpha }} \right)^2 W^2 \left[ {\rho ^2  + p^2  - \left( {\rho  + \frac{p}{{\Gamma  - 1}}} \right)^2 } \right] + p^2 \left( {\frac{1}{{\alpha ^2 }} - 1} \right) \\
&
 \le
 \left( {1 - \frac{{\left| {v_1 } \right|}}{{\varrho _1 }}} \right)^2 W^2 \left[ {\rho ^2  + p^2  - \left( {\rho  + \frac{p}{{\Gamma  - 1}}} \right)^2 } \right] + p^2 \left( {\frac{1}{{\varrho _1^2 }} - 1} \right),
\label{EQ-thz03}
\end{align}
here  \eqref{EQ-thz01} and $\rho ^2  + p^2  - \left( {\rho  + \frac{p}{{\Gamma  - 1}}} \right)^2 \le 0$ have been  used.
Note that $\varrho_1$ is a positive solution to the following quadratic equation
\[
(1 - v^2 c_s^2 )\varrho _1^2  - 2\left| {v_1 } \right|(1 - c_s^{ 2} )\varrho _1
+ { v_1^2(1-c_s^2) - c_s^2 (1-v^2)}  = 0,
\]
which is equivalent to
\begin{align}\label{EQ-thz04}
\left( {1 - \varrho _1^2 } \right)c_s^2  = W^2 \left( {\varrho _1  - \left| {v_1 } \right|} \right)^2 (1 - c_s^2 ).
\end{align}
It implies that $\varrho _1 < 1$.
With the help of \eqref{EQ-thz03} and \eqref{EQ-thz04},
one  has \begin{align*}
&
 \left( {D^ \pm  } \right)^2  + (\vec m^ \pm)^2  - \left( {E^ \pm  } \right)^2  \\
&  \le
 \left( {\frac{1}{{\varrho _1^2 }} - 1} \right)\frac{{c_s^2 }}{{1 - c_s^2 }}\left[ {p^2  - \frac{{2\rho p}}{{\Gamma  - 1}} - \frac{{p^2 }}{{\left( {\Gamma  - 1} \right)^2 }}} \right] + p^2 \left( {\frac{1}{{\varrho _1^2 }} - 1} \right) \\
&  = \left( {\frac{1}{{\varrho _1^2 }} - 1} \right)\frac{{p^2 }}{{(1 - c_s^2 )\left( {\Gamma  - 1} \right)}}\left[ {\Gamma  - 1 - c_s^2 \left( {\frac{1}{{\Gamma  - 1}} + \frac{{2\rho }}{p}} \right)} \right] \\
&  \overset{\Gamma\in(1,2]}{\le} \left( {\frac{1}{{\varrho _1^2 }} - 1} \right)\frac{{p^2 }}{{(1 - c_s^2 )\left( {\Gamma  - 1} \right)}}\left[ {1 - c_s^2 \left( {\frac{1}{{\Gamma  - 1}} + \frac{{2\rho }}{p}} \right)} \right] \\
&  = \left( {\frac{1}{{\varrho _1^2 }} - 1} \right)
\frac{{p^2 }}{{(1 - c_s^2 )\left( {\Gamma  - 1} \right)}}\cdot \frac{1 - 2\Gamma}{h}
\overset{\Gamma\in(1,2]}{<} 0.
\end{align*}
Thus %
$\vec U \pm {\alpha }^{-1}{{\vec F_1 (\vec U)}}
\in {\mathcal G}_1 $.
Combining the above deduction and the property ({\em ii}) may verify
$\vec U \pm {\alpha }^{-1}{{\vec F_i (\vec U)}}
\in {\mathcal G}_1$ for $i=2,\cdots, d$. For example,
  in the case of $d=3$, $T$ may be taken as
$$
\vec T_{\theta ,\phi } : = \left[ {\begin{array}{*{20}c}
   1 & 0 & 0 & 0 & 0  \\
   0 & { \cos \theta \sin \phi} & { \sin \theta\sin\phi  } & {\cos \phi } & 0  \\
   0 & { - \sin \theta } & {\cos \theta } & 0 & 0  \\
   0 & { - \cos \theta \cos \phi } & { -  \sin \theta \cos \phi} & {\sin \phi } & 0  \\
   0 & 0 & 0 & 0 & 1
\end{array}} \right],
$$
 for all $\theta\in[0,2\pi) ,\phi  \in [0,\pi]$.
Using the property ({\em ii}) that
$\vec T_{\theta ,\phi } \vec U \in {\mathcal G}_1$ and the above proof
of the property ({\em iii}) gives
\[
\vec T_{\theta ,\phi } \vec U \pm {\alpha }^{-1}{{\vec F_1 (\vec T_{\theta ,\phi } \vec U)}}
\in {\mathcal G}_1 ,
\]
where $\alpha$ is not less than $ \varrho _1$. 
Since $\vec T_{\theta ,\phi }^{ - 1}$  is also a rotational matrix, it holds that
\[
\vec T_{\theta ,\phi }^{ - 1} \left( {\vec T_{\theta ,\phi } \vec U \pm
{\alpha }^{-1}{{\vec F_1 (\vec T_{\theta ,\phi } \vec U)}}} \right)
= \vec U \pm {\alpha }^{-1}{{\vec T_{\theta ,\phi }^{ - 1} \vec F_1 (\vec T_{\theta ,\phi } \vec U)}}
 \in {\mathcal G}_1.
\]
With the help of the rotational invariance of the system \eqref{eqn:coneqn3d}, see \cite{ZhaoHeTang2014},
one has $\vec U \pm {\alpha }^{-1}{{\vec F_i (\vec U)}} \in {\mathcal G}_1 $
for $\forall \alpha  \ge \varrho _i,~i=2,3$.   The proof is completed. \qed
\end{proof}

 It is worth emphasizing that Lemma \ref{lam:propertyG} also
 plays a pivotal role in seeking the physical-constraints-preserving schemes.

\section{Numerical schemes}
\label{sec:scheme}

This section gives the physical-constraints-preserving
finite difference WENO schemes  for the RHD
system \eqref{eqn:coneqn3d} with the EOS \eqref{gamma-law}. 

\subsection{One-dimensional case}
\label{sec:1Dscheme}

This subsection first discusses numerical discretization of  the 1D RHD equations in the laboratory frame
\begin{equation}
	\label{eq:1D}
	\displaystyle\frac{\partial \vec{U}}{\partial t} +
	\frac{\partial \vec{F}_1(\vec{U})}{\partial x}=0,
\end{equation}
where
\begin{equation*}
	\vec U = (D,m_1,E)^T,\ \ \vec F_1 = (D v_1,m_1 v_1+p,m_1)^T.
\end{equation*}

\subsubsection{Spatial discretization}
\label{sec:reviweWENO}
Let us divide the space into cells of size
$\Delta x$, and denote the $j$th cell by $I_j = \left(x_{j-\frac{1}{2}},x_{j+\frac{1}{2}} \right)$,
where $x_{j+\frac{1}{2}} = \frac{1}{2}(x_j+x_{j+1})$ and $x_j=j\Delta x$,
 $j\in \mathbb Z$.

A  semi-discrete, $(2r-1)$th-order accurate,  conservative finite difference scheme
of the 1D RHD equations \eqref{eq:1D}
may be written as
\begin{equation}\label{eq:semi-ds}
\frac{{d \vec  U_j (t)}}{{dt}} =  - \frac{1}{{\Delta x}}\left( {\widehat {\vec F}_{j + \frac{1}{2}}
	- \widehat {\vec F}_{j - \frac{1}{2}} } \right)=:{\cal L}(\vec U(t);j),
\end{equation}
where $\vec U_j (t)\approx \vec U(x_j,t)$ and the numerical flux $\widehat {\vec F}_{j + \frac{1}{2}}$
is consistent with the flux vector $\vec F_1(\vec U)$ and satisfies
\[
 \frac{1}{{\Delta x}}\Big( {\widehat {\vec F}_{j + \frac{1}{2}}  - \widehat {\vec F}_{j - \frac{1}{2}} } \Big)
 =\partial _x \vec F_1 (\vec U)|_{x_j} +{\cal O}(\Delta x^{2r-1}).
\]

\begin{Def}
The scheme \eqref{eq:semi-ds} is  physical-constraints-preserving if
$\vec U_j (t) + \Delta t  {\cal L}\left( {\vec U (t)};j \right) \in {\mathcal G}$ for all
$j\in \mathbb Z$,
under a CFL-type condition for $\Delta t$ when $\vec U_j (t)\in {\mathcal G}$
for all $j$.
\end{Def}

The third property of Lemma \ref{lam:propertyG} has implied
that there at least exists a physical-constraints-preserving scheme for the 1D RHD system \eqref{eq:1D}. An example is the first-order (i.e. $r=1$) accurate local Lax-Friedrichs scheme
with the numerical flux
	$$\widehat {\vec F}_{j + \frac{1}{2}}=\widehat {\vec F}_{j + \frac{1}{2}}^{\rm LLF}:=
	\frac{1}{2}\big( {\vec F_1 (\vec U_j ) + \vec F_1 (\vec U_{j + 1} ) -  \alpha_{j + \frac{1}{2}} \left( {\vec U_{j + 1}  - \vec U_j } \right)} \big),
	$$
	where the viscosity coefficient $ \alpha _{j + \frac{1}{2}}:= \max \left\{ {\varrho_1 \left( {\vec U_{j-r+1} } \right),\cdots,
		\varrho_1 \left( {\vec U_{j + r} } \right)} \right\}$.
In practical computations, the above $\alpha _{j + \frac{1}{2}}$
may be replaced with
$$\alpha _{j + \frac{1}{2}} =\vartheta \max \left\{ {\varrho _{j + \frac{1}{2}}^{ROE} ,\varrho_1 \left( {\vec U_{j - r + 1} } \right), \cdots ,\varrho_1 \left( {\vec U_{j + r} } \right)} \right\}.
$$
see \cite{balsara2000},
where the parameter  $\vartheta$ is typically in the range of 1.1 to 1.3 and controls the amount of dissipation in the numerical schemes  while $\varrho _{j + \frac{1}{2}}^{ROE}$
is the spectral radius of the Roe matrix $\widehat{\vec A}_1(\vec U_j,\vec U_{j+1})$ approximating the Jacobian matrix
$ \vec A_1(\vec U)$, see \cite{EulderinkMel:1995}.

The physical-constraints-preserving property of
the above local Lax-Friedrichs scheme is shown below.

\begin{lemma}
	\label{lam:LLF}
	If $\vec U_j \in {\mathcal G}$ for all $j$,	
 then under the CFL-type condition
 	\begin{equation}\label{eq:CFL-LLF}
 		\Delta t \le \frac{{ \Delta x}}{{ 2 \mathop {\max }\limits_{j} \alpha_{j+\frac{1}{2}}}},
 	\end{equation}
 	one has
	\begin{equation*}
		\vec U_j^{\pm,\rm LLF}  := \vec U_j  \mp \frac{{2\Delta t}}{{\Delta x}}   \widehat {\vec F}_{j \pm \frac{1}{2}}^{\rm LLF} \in {\mathcal G},
	\end{equation*}
	and $\vec U_j (t) + \Delta t  {\cal L}\left( {\vec U (t)};j \right)
	=\frac12\left(\vec U_j^{+,\rm LLF}+\vec U_j^{-,\rm LLF}\right)	
	 \in {\mathcal G}$
	for all $j$.
\end{lemma}

\begin{proof}
	Note that $\vec U_j^{ \pm,\rm LLF }$ can be rewritten as
	\begin{align*}
		\vec U_j^{ \pm ,\rm LLF}  &= \vec U_j  \mp \frac{{\Delta t}}{{\Delta x}}\big( {\vec F_1 (\vec U_j ) + \vec F_1 (\vec U_{j \pm 1} ) \mp \alpha _{j \pm \frac{1}{2}} \left( {\vec U_{j \pm 1}  -\vec  U_j } \right)} \big) \\
		&= \left( {1 - 2\alpha _{j \pm \frac{1}{2}} \frac{{\Delta t}}{{\Delta x}}} \right) \vec U_j
		+ \alpha _{j \pm \frac{1}{2}} \frac{{\Delta t}}{{\Delta x}}
		\left[
		\left( {\vec U_j  \mp \frac{{\vec F_1 (\vec U_j )}}{{\alpha _{j \pm \frac{1}{2}} }}} \right)
		+ 
		\left( {\vec U_{j \pm 1}  \mp \frac{{\vec F_1 (\vec U_{j \pm 1} )}}{{\alpha _{j \pm \frac{1}{2}} }}} \right) \right],
	\end{align*}
	which is a  convex combination under the CFL-type condition \eqref{eq:CFL-LLF}. Utilizing the property (iii) of Lemma \ref{lam:propertyG} implying that
	$$\vec U_j  \mp \frac{{\vec F_1 (\vec U_j )}}{{\alpha _{j \pm \frac{1}{2}} }}, ~  \vec U_{j \pm 1}  \mp \frac{{\vec F_1 (\vec U_{j \pm 1} )}}{{\alpha _{j \pm \frac{1}{2}} }} \in {\mathcal G},$$
	and the convexity of the set $\mathcal G$ in Lemma \ref{lam:convex}
	may complete the proof. \qed
\end{proof}

The remaining task is to develop higher-order (i.e. $r>1$) accurate
 physical-constraints-preserving finite difference scheme
 for the 1D RHD equations \eqref{eq:1D}. To finish such task,
the idea of the positivity-preserving finite difference WENO schemes for compressible Euler equations in \cite{zhang2012} and the flux-limiter
in \cite{Hu2013} is borrowed here. For the sake of convenience, the
independent variable $t$ will be temporarily omitted.


Given point values $\{   \vec U_j  \}$,  for each $j$, calculate
$$\vec {\overline H}^\pm_k:=
\frac{1}{2}\left( {\vec U_k \pm   {{\alpha _{j + \frac{1}{2}}^{-1} }}
{{\vec F_1(\vec U_k)}} } \right), ~~ j-r+1\leq k\leq  j+r,
$$
which may be considered as the point values
of both local Lax-Friedrichs type splitting functions
\[
\frac{1}{2}\left( {\vec U(x) \pm {{\alpha _{j + \frac{1}{2}}^{-1} }} {{\vec F_1(\vec U(x))}}} \right).
\]

If define the  functions $\vec H_{j + \frac{1}{2}}^\pm(x)$ by
$$
\frac{1}{2}\left( {\vec U(x) \pm {{\alpha _{j + \frac{1}{2}}^{-1} }} {{\vec F_1(\vec U(x))}}} \right)
=\frac{1}{\Delta x}\int_{x-\Delta x/2}^{x+\Delta x/2}
\vec H_{j + \frac{1}{2}}^\pm(s)~ds,$$
then $\vec {\overline H}^\pm_k$ become
 the cell average values  of $\vec H_{j + \frac{1}{2}}^\pm(x)$
 over the cell $I_k$ because
\[
\overline{\vec  H}^\pm_k\equiv \frac{1}{{\Delta x}}\int_{x_k - \frac{{\Delta x}}{2}}^{x_k + \frac{{\Delta x}}{2}} {\vec H_{j+\frac{1}{2}}^{\pm}  (\xi )d\xi }, \quad j - r + 1 \le k \le j + r.
\]
Based on these cell-average values,
using the WENO reconstruction \cite{jiang1996,ShuSIReV2009} may
get high-order accurate left- and right-limited approximations of $\vec H_{j + \frac{1}{2}}^\pm(x)$
at the cell boundary $x_{j+\frac12}$, denoted by
$\vec H_{j + \frac{1}{2},L}^{+,\rm WENO}$ and
$\vec H_{j + \frac{1}{2},R}^{-,\rm WENO}$ respectively.
After then, the numerical flux of a $(2r-1)$th-order accurate  finite difference WENO scheme
of the 1D RHD equations \eqref{eq:1D} is
\begin{equation}\label{eq:WENOflux}
\widehat {\vec F}_{j + \frac{1}{2}}=
\widehat {\vec F}_{j + \frac{1}{2}}^{\rm WENO} : = \alpha _{j + \frac{1}{2}} \left(\vec H_{j + \frac{1}{2},L}^{+,\rm WENO}
 -  \vec H_{j + \frac{1}{2},R}^{-,\rm WENO} \right).
\end{equation}
Practically, for the system of conservation laws, a better and stable way to derive
 those left- and right-limited approximations is to
impose the WENO reconstruction on the characteristic variables
by means of their cell-average values
\begin{align*}
&
\overline {\vec W}^+_k  := \widetilde{\vec R}_{j + \frac{1}{2}} ^{-1}  \vec {\overline H}^+_k ,\quad j - r + 1 \le k \le j + r - 1,\\[2mm]
&
\overline {\vec W} ^ - _k : =\widetilde{\vec R}_{j + \frac{1}{2}} ^{-1} \vec {\overline H}^-_k ,\quad j - r + 2 \le k \le j + r,
\end{align*}
where $\widetilde{\vec R}_{j + \frac{1}{2} }$ is the right eigenvector matrix of the
Roe matrix $\widetilde{\vec A}_1(\vec U_j,\vec U_{j+1}) $.
After  having the  left- (resp. right-) limited WENO values of
the characteristic variables $\vec W^\pm$ at the cell boundary $x_{j+\frac{1}{2}}$,
denoted by $ \vec W_{j + \frac{1}{2},L}^ {+,\rm WENO}$ (resp. $ \vec W_{j + \frac{1}{2},R}^ {-,\rm WENO}$), then calculate
\[
\vec H_{j + \frac{1}{2},L}^{+,\rm WENO}  = \widetilde{\vec R}_{j + \frac{1}{2}}
 \vec W_{j + \frac{1}{2},L}^ {+,\rm WENO} ,\quad
\vec H_{j + \frac{1}{2},R}^{-,\rm WENO}  = \widetilde{\vec R}_{j + \frac{1}{2}}    \vec W_{j + \frac{1}{2},R}^ {-,\rm WENO}.
\]

Generally,  the high-order accurate finite
		 difference WENO schemes
		 \eqref{eq:semi-ds}
		 with the numerical flux
		  $\widehat {\vec F}_{j + \frac{1}{2}}=\widehat {\vec F}_{j + \frac{1}{2}}^{\rm WENO}$ given in \eqref{eq:WENOflux}
           is not physical-constraints-preserving, that is to say, it is possible to meet
           $\vec U_j (t) + \Delta t  {\cal L}\left( {\vec U (t)};j \right)
           		=\frac12\left(\vec U_j^{+,\rm WENO}+\vec U_j^{-,\rm WENO}\right)	
           		 \notin {\mathcal G}$, where
           		 $\vec U_j^{\pm,\rm WENO}  := \vec U_j  \mp \frac{{2\Delta t}}{{\Delta x}}   \widehat {\vec F}_{j \pm \frac{1}{2}}^{\rm WENO}$.
Thus for some demanding extreme
  problems, such as their solutions involving low density or pressure, or very large
  velocity or the ultra-relativistic flow, 
  these high-order schemes always easily break down after some time steps
  due to the nonphysical numerical solutions ($\vec U_j \notin {\mathcal G}$).
To cure such difficulties,
the positivity-preserving flux limiter \cite{Hu2013} for non-relativistic Euler equations
may be borrowed  and extended to our RHD case.
Because of the definition of ${\mathcal G}_1$ in Lemma \ref{lam:equDef}
and the properties of $q(\vec U)$ shown in Remark \ref{rem:q},
the resulting physical-constraints-preserving flux limiter may be formed into two steps as follows
in order to preserve the positivity of $D(\vec U)$ and $q(\vec U)$.
The flux limiter such as the parametrized flux limiter 
\cite{Xu_MC2013,Liang2014,XiongQiuXu2014}
can also be  extended to our RHD case in a similar way.

Before that,
two sufficiently small positive numbers $\varepsilon _D$ and $\varepsilon _q$
are first introduced (taken as $10^{-13}$ in numerical computations)
such that $ D_j^{ \pm,\rm LLF } \geq  \varepsilon _D>0$
and $q(\vec U_j^{ \pm,\rm LLF }) \geq \varepsilon _q>0$ for all $j$.
It is true because
 Lemma \ref{lam:LLF} tells us that the mass-density $D_j^{ \pm,\rm LLF } >0$
and $q(\vec U_j^{ \pm,\rm LLF }) >0$ for all $j$, where $D_j^{ \pm,\rm LLF }$ denotes the first component of $\vec U_j^{ \pm,\rm LLF }$.

\noindent
{\tt Step I:  Enforce the positivity of  $D(\vec U)$}.
For each $j$, correct the numerical flux $\widehat {\vec F}_{j + \frac{1}{2}}^{\rm WENO}$ as 
  \begin{equation}
    \label{eq:lmiterD}
    \left\{ \widehat {\vec F}_{j + \frac{1}{2}}^{D} \right\}_{\ell}
    :=
    \begin{cases}
      ( 1 - \theta_{D,j+\frac{1}{2}} )\left\{ \widehat {\vec F}_{j + \frac{1}{2}}^{\rm LLF} \right\}_{\ell}
      + \theta_{D,j+\frac{1}{2}}  \left\{  \widehat {\vec F}_{j + \frac{1}{2}}^{\rm WENO} \right\}_{\ell},\ \ & \ell = 1,\\[3mm]
      \left\{  \widehat {\vec F}_{j + \frac{1}{2}}^{\rm WENO} \right\}_{\ell},\ \ & \ell > 1,
    \end{cases}
  \end{equation}
where $\left\{  \widehat {\vec F}_{j + \frac{1}{2}} \right\}_{\ell}$
denotes the $\ell$th component of
$ \widehat {\vec F}_{j + \frac{1}{2}} $ and
$\theta_{D,j+1/2}=\min \{
\theta_{D,j+1/2}^+,  \theta_{D,j+1/2}^-  \}$
with
\begin{align*}
\theta _{D,j + \frac{1}{2}}^ \pm
 = \begin{cases}
 ( D_{j+\frac12\mp\frac12}^{ \pm ,\rm LLF}
 - \varepsilon _D )/
(    D_{j+\frac12\mp\frac12}^{ \pm ,\rm LLF}
 - D_{j+\frac12\mp\frac12}^{ \pm ,\rm WENO}),
                  &  \mbox{if  } D_{j+\frac12\mp\frac12}^{ \pm ,\rm WENO}<\varepsilon _D,  \\
1, & \mbox{otherwise}.\end{cases}
\end{align*}

\noindent
{\tt Step II:
Enforce the positivity of  $q(\vec U)$}.
For each $j$, limit the numerical flux $\widehat {\vec F}_{j + \frac{1}{2}}^{D}$ as
\begin{equation}\label{eq:lmiterq}
\widehat {\vec F}_{j + \frac{1}{2}}^{\rm PCP}
 :=( 1 - \theta_{q,j+\frac{1}{2}} ) \widehat {\vec F}_{j + \frac{1}{2}}^{\rm LLF} + \theta_{q,j+\frac{1}{2}} \widehat {\vec F}_{j + \frac{1}{2}}^{D},
\end{equation}
where $\theta_{q,j+\frac{1}{2}} = \min \left\{ \theta_{q,j+\frac{1}{2}}^+ ,\theta_{q,j+\frac{1}{2}}^- \right\}$, and
\begin{align*}
\theta _{q,j + \frac{1}{2}}^ \pm
 = \begin{cases}
 \left( q(\vec U_{j+\frac12\mp\frac12}^{ \pm ,\rm LLF})
 - \varepsilon _q \right)/
\left(  q(\vec U_{j+\frac12\mp\frac12}^{ \pm ,\rm LLF} )
 - q(\vec U_{j+\frac12\mp\frac12}^{ \pm ,\rm D})\right),
                  &  \mbox{if  }
                  q(\vec U_{j+\frac12\mp\frac12}^{ \pm ,\rm D})
                  <\varepsilon _q,  \\
1, & \mbox{otherwise}.\end{cases}
\end{align*}




It is worth emphasizing that the above limitting procedure is slightly different from that in \cite{Hu2013}, because only the first component of the numerical flux vector is detected and limited in {\tt Step I}. Moreover,
the  finite  difference  schemes
		 \eqref{eq:semi-ds}
		 with the numerical flux
	  $\widehat {\vec F}_{j + \frac{1}{2}}=\widehat {\vec F}_{j + \frac{1}{2}}^{\rm PCP}$ given in \eqref{eq:lmiterq}
	  is obviously consistent with the 1D RHD equations \eqref{eq:1D} and physical-constraints-preserving (see Theorem
	  \ref{thm:PCP}),
            and  maintains ($2r-1$)th-order accuracy in the smooth region without vacuum  (see Theorem
            \ref{Thm23}).

\begin{thm}\label{thm:PCP}
Under the assumption of Lemma \ref{lam:LLF},
if $ D_j^{ \pm,\rm LLF } >0$ and $q(\vec U_j^{ \pm,\rm LLF })>0$ for all $j$, then
$\vec U_j^ {-,\rm PCP}, \vec U_j^ {+,\rm PCP} \in {\mathcal G}$, and
           $\vec U_j (t) + \Delta t  {\cal L}\left( {\vec U (t)};j \right)
           		=\frac12\left(\vec U_j^{+,\rm PCP}+\vec U_j^{-,\rm PCP}\right)	
           		 \in {\mathcal G}$  for all $j$, where
           		 $\vec U_j^{\pm,\rm PCP}  := \vec U_j  \mp \frac{{2\Delta t}}{{\Delta x}}   \widehat {\vec F}_{j \pm \frac{1}{2}}^{\rm PCP}$.
\end{thm}

\begin{proof}
According to Lemma \ref{lam:LLF} and the previous flux limiting procedure, one has
$0\leq \theta_{D,j+\frac12}, \theta_{q,j+\frac12}\leq 1$, and
there exist two sufficiently small positive numbers
$\varepsilon _D $ and $\varepsilon _q$ such that
$ D_j^{ \pm,\rm LLF } \geq \varepsilon _D>0$ and $q(\vec U_j^{ \pm,\rm LLF }) \geq \varepsilon _q>0$ for all $j$.

Substituting \eqref{eq:lmiterD} into \eqref{eq:lmiterq} gives
\begin{equation*}
\left\{ \widehat {\vec F}_{j + \frac{1}{2}}^{\rm PCP} \right\}_{1}
 =( 1 - \theta_{j+\frac{1}{2}} )
 \left\{ \widehat {\vec F}_{j + \frac{1}{2}}^{\rm LLF} \right\}_{1}
 + \theta_{j+\frac{1}{2}} \left\{ \widehat {\vec F}_{j + \frac{1}{2}}^{\rm WENO} \right\}_{1},
\end{equation*}
where $0\leq \theta_{j+\frac{1}{2}}=\theta_{D,j+\frac{1}{2}} \theta_{q,j+\frac{1}{2}} \le \theta_{D,j+\frac{1}{2}}^+ $.
According to
the definition of $\theta^+_{D,j+1/2}$,  the inequality
\begin{align*}
(1- \theta^+_{D,j+1/2})D_j^{+,\rm LLF}+\theta^+_{D,j+1/2} D_j^{+,\rm WENO}
\geq \varepsilon_D,
\end{align*}
always holds.
Thus one has
\begin{align*}
D_{j + \frac{1}{2}}^{+,\rm PCP}
  &=  ( 1 - \theta_{j+\frac{1}{2}} )
 D_{j + \frac{1}{2}}^{+,\rm LLF}
 + \theta_{j+\frac{1}{2}} D_{j + \frac{1}{2}}^{+,\rm WENO} \\
&=
\frac{{\theta _{j + \frac{1}{2}} }}{{\theta _{D,j + \frac{1}{2}}^ +  }}
\left(
(1-\theta _{D,j + \frac{1}{2}}^+) D_j^{ + ,\rm LLF}
+ \theta _{D,j + \frac{1}{2}}^+  D_j^{ + ,\rm WENO}
 \right)
+ \left( {1 - \frac{{\theta _{j + \frac{1}{2}} }}{{\theta _{D,j + \frac{1}{2}}^ +  }}} \right) D_j^{ + ,\rm LLF}
\\
&
\geq
\frac{\theta _{j + \frac{1}{2}} }{\theta _{D,j + \frac{1}{2}}^ +  }  \varepsilon _D
+ \left( {1 - \frac{{\theta _{j + \frac{1}{2}} }}{{\theta _{D,j + \frac{1}{2}}^ +  }}} \right) D_j^{ + ,\rm LLF}
\ge \varepsilon _D > 0.
\end{align*}
On the other hand, similarly, making use of the concavity of $q(\vec U)$
 gives
\begin{align*}
 q\left( \vec  U_{j + \frac{1}{2}}^{ + ,\rm PCP}  \right)
 &
 = q\left( (1 - \theta _{q,j + \frac{1}{2}} ) \vec U_{j + \frac{1}{2}}^{ + ,\rm LLF}  + \theta _{q,j + \frac{1}{2}} \vec U_{j + \frac{1}{2}}^{ + ,D}  \right) \\
 &
  \ge (1 - \theta _{q,j + \frac{1}{2}} )q\left( \vec U_{j + \frac{1}{2}}^{ + ,\rm LLF}  \right) + \theta _{q,j + \frac{1}{2}} q\left( \vec U_{j + \frac{1}{2}}^{ + ,D}  \right) \\
  &
  = \frac{{\theta _{q,j + \frac{1}{2}} }}{{\theta _{q,j + \frac{1}{2}}^ +  }}\left( (1 - \theta _{q,j + \frac{1}{2}}^ +  )q\left( \vec U_{j + \frac{1}{2}}^{ + ,\rm LLF}  \right) + \theta _{q,j + \frac{1}{2}}^ +  q\left( \vec U_{j + \frac{1}{2}}^{ + ,D}  \right) \right) + \left( 1 - \frac{{\theta _{q,j + \frac{1}{2}} }}{{\theta _{q,j + \frac{1}{2}}^ +  }} \right)q\left( \vec  U_{j + \frac{1}{2}}^{ + ,\rm LLF}  \right) \\
  &
\geq
  \frac{{\theta _{q,j + \frac{1}{2}} }}{{\theta _{q,j + \frac{1}{2}}^ +  }}\varepsilon _q  + \left( {1 - \frac{{\theta _{q,j + \frac{1}{2}} }}{{\theta _{q,j + \frac{1}{2}}^ +  }}} \right)q\left( \vec U_{j + \frac{1}{2}}^{ + ,\rm LLF}  \right) \ge \varepsilon _q  > 0.
 \end{align*}
With the equivalent definition of the admissible state set ${\mathcal G}$ in Lemma \ref{lam:equDef}, one knows that $\vec U_j^ {+,\rm PCP} \in {\mathcal G}$.
Similarly, one can also prove that $\vec U_j^ {-,\rm PCP} \in {\mathcal G}$. The proof is completed. \qed
\end{proof}

The following is to check the accuracy of the schemes
		 \eqref{eq:semi-ds}
		 with the numerical flux
	  $\widehat {\vec F}_{j + \frac{1}{2}}
	  =\widehat {\vec F}_{j + \frac{1}{2}}^{\rm PCP}$ given in \eqref{eq:lmiterq}.
 In fact,  the above flux limiting procedure implies that
$$
\left\|\widehat {\vec F}_{j + \frac{1}{2}}^{\rm PCP} - \widehat {\vec F}_{j + \frac{1}{2}}^{\rm WENO} \right\| \le
(1-\theta_{j+\frac{1}{2}}) \left\|\widehat {\vec F}_{j + \frac{1}{2}}^{\rm LLF} - \widehat {\vec F}_{j + \frac{1}{2}}^{\rm WENO} \right\|.
$$
Thus if
\begin{align}
\label{EQ:thz}
1-\theta_{j+\frac{1}{2}} = {\cal O} (\Delta x^{2r-1}),
\end{align}
then  the schemes
		 \eqref{eq:semi-ds}
		 with the numerical flux
	  $\widehat {\vec F}_{j + \frac{1}{2}}
	  =\widehat {\vec F}_{j + \frac{1}{2}}^{\rm PCP}$
	  is  $(2r-1)$th-order accurate  because both ${\vec U}_{j + \frac{1}{2}}^{+,\rm LLF}$ and ${\vec U}_{j + \frac{1}{2}}^{+,\rm WENO}$ are  bounded in smooth regions.

\begin{thm}\label{Thm23}
Assuming that the exact solution $\vec U(x)$ is smooth
and satisfies that $D(x)\ge \varepsilon>0$  and
$q(\vec U(x)) \ge \varepsilon>0$  for all $x$,
where $\varepsilon >\frac{2}{1-\hat w} \max\{\varepsilon_D,\varepsilon_q\}$,
 and
 the approximate solution $\vec U_j \in {\mathcal G}$
and $\vec U_j  = \vec U(x_j ) +{\cal O} (\Delta x^{2r-1})$ for all $j$ and
sufficiently small $\Delta x$, then
$$
\theta _{D,j + \frac{1}{2}}^ \pm = 1 + {\cal O} (\Delta x^{2r-1}),\quad \theta _{q,j + \frac{1}{2}}^ \pm = 1 + {\cal O} (\Delta x^{2r-1}),
$$
and \eqref{EQ:thz} hold for the given $(2r-1)$th-order accurate numerical flux $\widehat {\vec F}_{j + \frac{1}{2}}^{\rm WENO}$ under the CFL-type condition 
\begin{equation}\label{eq:CFL-PCP}
\Delta t \le \frac{{ \hat w \Delta x}}{{ 2 \mathop {\max }\limits_{j} \alpha_{j+\frac{1}{2}}}},
\end{equation}
where $\hat w$ is any positive constant less than one.
\end{thm}

\begin{proof}
Only estimations of  $\theta _{D,j + \frac{1}{2}}^+$ and $\theta _{q,j + \frac{1}{2}}^+$
are given below, while other cases are similar and will be omitted here.

Before that,  some basic conclusions are first listed as follows.
\begin{itemize}
  \item Lemma \ref{lam:LLF} implies that
\begin{equation}\label{eq:ULLF_G_2}
\vec U_j  \mp \frac{{2\Delta t}}{{\hat w \Delta x}}
 \widehat {\vec F}_{j \pm \frac{1}{2}}^{\rm LLF} \in {\mathcal G},
\end{equation}
under the CFL-type condition \eqref{eq:CFL-PCP}.
  \item Since
\begin{equation*}
\left\| \widehat {\vec F}_{j + \frac{1}{2}}^{\rm LLF} - \vec H(x_{j+\frac{1}{2}}) \right\| = {\cal O} (\Delta x),\quad
\left\| \widehat {\vec F}_{j + \frac{1}{2}}^{\rm WENO} - \vec H(x_{j+\frac{1}{2}}) \right\| = {\cal O} (\Delta x^{2r-1}),
\end{equation*}
where the vector function $\vec H(x)$ is implicitly defined by
\[
\vec F_1(\vec U(x)) = \frac{1}{{\Delta x}}\int_{x - \frac{{\Delta x}}{2}}^{x + \frac{{\Delta x}}{2}} {\vec H(\xi )} ~{\rm d}\xi,
\]
one has
\begin{align*} 
 \vec U_j^e :& = \vec U_j^{ + ,\rm LLF}  + \frac{{2\Delta t}}{{\Delta x}}\left( { \widehat {\vec F}_{j + \frac{1}{2}}^{\rm LLF}  - \vec H(x_{j + \frac{1}{2}} )} \right)\\
 &= \vec U_j^{ + ,\rm LLF}  + \frac{{2\Delta t}}{{\Delta x}}\left( {\widehat {\vec F}_{j + \frac{1}{2}}^{\rm LLF}  - \widehat {\vec F}_{j + \frac{1}{2}}^{\rm WENO} } \right)
 + {\cal O} (\Delta x^{2r-1})\\
  &= \vec U_j  - \frac{{2\Delta t}}{{\Delta x}} \widehat {\vec F}_{j + \frac{1}{2}}^{\rm WENO} + {\cal O} (\Delta x^{2r-1})
   = \vec U_j^{ + ,\rm WENO}  + {\cal O} (\Delta x^{2r-1}).
\end{align*}
It further yields
\begin{align}\label{eq:UeULLF_D}
&\left|  D_j^e - D_j^{ + ,\rm LLF} \right| \le \left\|  \vec U_j^e - \vec U_j^{ + ,\rm LLF} \right\| = {\cal O} (\Delta x), \\
\label{eq:UeULLF_q}
&\left|  q(\vec U_j^e) - q(\vec U_j^{ + ,{\rm LLF}}) \right| \le \sqrt{2} \left\|  \vec U_j^e - \vec U_j^{ + ,\rm LLF} \right\| = {\cal O} (\Delta x).\\
\label{eq:UeUWENO_D}
&\left|  D_j^e - D_j^{ + ,{\rm WENO}} \right| \le \left\|  \vec U_j^e - \vec U_j^{ + ,\rm WENO} \right\| = {\cal O} (\Delta x^{2r-1}), \\
\label{eq:UeUWENO_q}
&\left|  q(\vec U_j^e) - q(\vec U_j^{ + ,\rm WENO}) \right| \le \sqrt{2} \left\|  \vec U_j^e - \vec U_j^{ + ,\rm WENO} \right\| = {\cal O} (\Delta x^{2r-1}),
\end{align}
where 
\eqref{EQ:Lip} has been used.
\end{itemize}

{\tt (i): Estimate $\theta _{D,j + \frac{1}{2}}^+$}.
The case of that $\theta _{D,j + \frac{1}{2}}^+ = 1$ is trival.
Assuming $0\leq  \theta _{D,j + \frac{1}{2}}^+ < 1$, which implies $D_j^{ + ,\rm WENO} < \varepsilon _D $, then
$$
1 - \theta _{D,j + \frac{1}{2}}^ + = 1-\frac{   D_j^{ + ,\rm LLF}  - \varepsilon _D }
{   D_j^{ + ,\rm LLF}    - D_j^{ + ,\rm WENO}   }
=\frac{ \varepsilon _D - D_j^{ + ,\rm WENO}  } { D_j^{ + ,\rm LLF}   - D_j^{ + ,\rm WENO}  }
< \frac{ \left| \varepsilon _D - D_j^{ + ,\rm WENO}  \right| }
{ D_j^{ + ,\rm LLF}   - \varepsilon _D  }.
$$
Since $D_j$ is a $(2r-1)$th-order accurate approximation to $D(x_j)$,
one has $D_j=D(x_j)+{\cal O} (\Delta x^{2r-1})\ge \varepsilon+ {\cal O} (\Delta x^{2r-1}) \ge \frac{\varepsilon}{2}$. Thanks to \eqref{eq:ULLF_G_2}, one gets
\begin{align*}
D_j^{ + ,\rm LLF}   - \varepsilon _D = (1-\hat w) D_j + \hat w \left( D_j - \frac{2\Delta t}{\hat w \Delta x}
\left \{ \widehat {\vec F}_{j + \frac{1}{2}}^{\rm LLF} \right\}_1\right) - \varepsilon _D \ge { (1-\hat w)\frac{ \varepsilon}{2}- \varepsilon _D >0},
\end{align*}
which shows that $D_j^{ + ,\rm LLF}   - \varepsilon _D$ is bounded away from zero. Then, one only needs to
show $\left| \varepsilon _D - D_j^{ + ,\rm WENO}  \right| = {\cal O} (\Delta x^{2r-1})$. Note that \eqref{eq:UeULLF_D} implies
\[
D_j^e  - \varepsilon _D  = D_j^{ + ,\rm LLF}  - \varepsilon _D  + {\cal O} (\Delta x) \ge (1 - \hat w)\frac{\varepsilon}{2} - \varepsilon _D  + {\cal O} (\Delta x)
 \ge 
\frac12 \left( (1 - \hat w)\frac{\varepsilon}{2} - \varepsilon _D \right)
 > 0.
\]
 Thus
$$
\left| \varepsilon _D - D_j^{ + ,\rm WENO}  \right|
= \varepsilon _D - D_j^{ + ,\rm WENO} < D_j^e - D_j^{ + ,\rm WENO} = {\cal O} (\Delta x^{2r-1}),
$$
where \eqref{eq:UeUWENO_D} has been used.

{\tt (ii):  Estimate $\theta _{q,j + \frac{1}{2}}^+$}.
Similarly, only consider the nontrival case that $0\leq \theta _{q,j + \frac{1}{2}}^+ < 1$ implying  $q(\vec U_j^{ + ,\rm WENO} ) < \varepsilon _q $. Thus
$$
1 - \theta _{q,j + \frac{1}{2}}^ +
=\frac{ \varepsilon _q - q(\vec U_j^{ + ,\rm WENO})  } { q (\vec U_j^{ + ,\rm LLF})   - q (\vec U_j^{ + ,\rm WENO})  }
< \frac{ \left| \varepsilon _q - q (\vec U_j^{ + ,\rm WENO})  \right| }
{ q (\vec U_j^{ + ,\rm LLF})   - \varepsilon _q  }.
$$
Because $\vec U_j$ is a $(2r-1)$th-order accurate approximation to $\vec U(x_j)$, $q(\vec U_j)$ is also a $(2r-1)$th-order accurate approximation to $q(\vec U(x_j))$ by
the Lipschitz continuity of $q(\vec U)$. Thus it holds that $q(\vec U_j ) \ge \varepsilon- {\cal O} (\Delta x^{2r-1}) \ge \frac{\varepsilon}{2}$.
With the help of the concavity of $q(\vec U)$ and \eqref{eq:ULLF_G_2}, one may know that $q(\vec U_j^{ + ,\rm LLF})   - \varepsilon _{{q}}$ is bounded away from zero because
\begin{align*}
q(\vec U_j^{ + ,\rm LLF})   - \varepsilon _{{q}}
&= q\left(  (1-\hat w) \vec U_j + \hat w \left( \vec U_j - \frac{2\Delta t}{\hat w \Delta x}\widehat {\vec F}_{j + \frac{1}{2}}^{\rm LLF} \right) \right) - \varepsilon _{{q}} \\
&\ge (1-\hat w) q( \vec U_j )  + \hat w q \left( \vec U_j - \frac{2\Delta t}{\hat w \Delta x}\widehat {\vec F}_{j + \frac{1}{2}}^{\rm LLF} \right)  - \varepsilon _{{q}} \\
&\ge (1-\hat w) q( \vec U_j ) - \varepsilon _{{q}} \ge \frac{1-\hat w}{2} \varepsilon - \varepsilon _{{q}} >0.
\end{align*}
 Then, one turns to show $\left| \varepsilon _q - q(\vec U_j^{ + ,\rm WENO})  \right| = {\cal O} (\Delta x^{2r-1})$.
In fact, \eqref{eq:UeULLF_q} implies
\[
q(\vec U_j^e)  - \varepsilon _q  = q(\vec U_j^{ + ,\rm LLF})  - \varepsilon _{{ q}}  + {\cal O} (\Delta x) \ge
(1 - \hat w)
\frac{{\varepsilon}}{2} - \varepsilon _{{ q}}
+ {\cal O} (\Delta x) \ge
\frac12\left(  (1 - \hat w)
\frac{{\varepsilon}}{2} - \varepsilon _{{ q}} \right)
> 0.
\]
 Therefore
$$
\left| \varepsilon _q - q(\vec U_j^{ + ,\rm WENO} ) \right|
= \varepsilon _q - q(\vec U_j^{ + ,\rm WENO}) < q(\vec U_j^e) - q(\vec U_j^{ + ,\rm WENO}) = {\cal O} (\Delta x^{2r-1}),
$$
where \eqref{eq:UeUWENO_q} has been used.

Using the above results and
$$\theta_{j+\frac{1}{2}}=\theta_{D,j+\frac{1}{2}} \theta_{q,j+\frac{1}{2}}
= \min \left\{ \theta_{D,j+\frac{1}{2}}^+ ,\theta_{D,j+\frac{1}{2}}^- \right\}  \min \left\{ \theta_{q,j+\frac{1}{2}}^+ ,\theta_{q,j+\frac{1}{2}}^- \right\},$$
gives \eqref{EQ:thz}.

The proof is completed. \qed
\end{proof}

\subsubsection{Time discretization}
\label{sec:Time}

Time derivatives in the semi-discrete schemes \eqref{eq:semi-ds}
 can be approximated  by using some high-order strong stability preserving
(SSP) method \cite{Gottlieb2009}. A special example considered here is the third order accurate
 SSP explicit Runge-Kutta method
\begin{align} \label{eq:RK1} \begin{aligned}
& \vec U^ *_j   = \vec U^n_j  + \Delta t_n {\cal L}(\vec U^n;j ), \\
& \vec U^{ *  * }_j  = \frac{3}{4}\vec U^n_j  + \frac{1}{4}\Big(\vec U^ *_j
 + \Delta t_n {\cal L}(\vec U^ *;j  )\Big), \\
& \vec U^{n+1}_j  = \frac{1}{3} \vec U^n_j  + \frac{2}{3}\Big(\vec U^{ *  * }_j
+ \Delta t_n {\cal L}(\vec U^{ *  * };j)\Big),
\end{aligned}\end{align}
with $ {\cal L}\left( {\vec U};j \right)
           		=\big(\hat{\vec F}_{j-\frac12}^{\rm PCP}  ({\vec U})
           		-\hat{\vec F}_{j+\frac12}^{\rm PCP} ({\vec U})   \big)/\Delta x$.

In practical computations, the time stepsize selection strategy
 in \cite{wang2012} may be adopted to improve computational efficiency,
  and the  physical-constraints-preserving flux limiter may also be
 	slightly modified and implemented
 	via enforcing directly $ \vec U^{ *  * }_j\in {\mathcal G}_1 $ and $ \vec U^{ n+1}_j\in {\mathcal G}_1$
 	in the second and third stages in \eqref{eq:RK1}.
Taking the second stage as an example, 	$\vec U^{\pm,{WENO}}_j$ in the previous
flux limiting procedure is replaced with
 	$\vec U^{\pm,{\rm WENO}}_j=\frac34 \vec U^n_j +\frac14
 	\big(    \vec U^*_j \mp \frac{2\Delta t_n} {\Delta x}
 	\hat{\vec F}_{j\pm \frac12}^{\rm WENO} (\vec U^*)
 	\big)$.


%

\subsection{Two-dimensional case}
\label{sec:2Dscheme}

The high-order accurate physical-constraints-preserving  finite difference WENO schemes
presented in Subsection \ref{sec:1Dscheme}  
can be easily extended to multidimensional RHD equations \eqref{eqn:coneqn3d}.
This section  only presents  its  extension to  the 2D RHD equations
in the laboratory frame
\begin{equation}\label{eq:2D}
	\displaystyle\frac{\partial \vec{U}}{\partial t} +
	\frac{\partial \vec{F}_1(\vec{U})}{\partial x}
	+ \frac{\partial \vec{F}_2(\vec{U})}{\partial y}=0,
\end{equation}
where
\begin{align*}
	\vec{U} =& (D, m_1, m_2, E)^T,\quad
	\vec{F}_1 = (Dv_1, m_1 v_1 + p,  m_2 v_1, m_1)^T,\\
	\vec{F}_2 =& (Dv_2, m_1 v_2,  m_2 v_2+p,   m_2)^T.
\end{align*}

Let us divide the spatial domain   $\Omega$ into  a rectangular mesh
with the cell
$\{(x,y)|~x_{j-\frac{1}{2}}
< x < x_{j+\frac{1}{2}}, y_{k-\frac{1}{2}}<y<y_{k+\frac{1}{2}}\}$
where $x_{j+\frac{1}{2}}
=(j+\frac{1}{2})\Delta x$ and $y_{k+\frac{1}{2}}=(k+\frac{1}{2})\Delta y$, $j,k \in \mathbb{Z}$,
and both spatial stepsizes $\Delta x$ and $\Delta y$ are given positive constants.

Then a semi-discrete, ($2r-1$)th-order accurate, conservation finite difference scheme
for  the 2D RHD equations \eqref{eq:2D} may be derived as
\begin{equation}\label{eq:2Dtimeds}
\frac{{d\vec U_{j,k} (t)}}{{dt}}
=  \frac{   \widehat {\vec F}_{j - \frac{1}{2},k}^{1,\rm PCP}  - \widehat {\vec F}_{j + \frac{1}{2},k}^{1,\rm PCP}     }{{\Delta x}}
 + \frac{   \widehat {\vec F}_{j,k- \frac{1}{2}}^{2,\rm PCP}  - \widehat {\vec F}_{j,k+ \frac{1}{2}}^{2,\rm PCP}     }{{\Delta y}}
= : {\cal L} (\vec U (t);j,k),
\end{equation}
where the numerical flux $\widehat {\vec F}_{j + \frac{1}{2},k}^{1,\rm PCP}$
(resp. $\widehat {\vec F}_{j,k+ \frac{1}{2}}^{2,\rm PCP}$) is derived by  using the procedure in Subsection \ref{sec:1Dscheme} for each fixed $k$ (resp. $j$) based on
the local Lax-Firedrichs splitting and 1D high-order WENO reconstruction with physical-constraints-preserving flux limiter.
The time derivatives in \eqref{eq:2Dtimeds} may be approximated
by utilizing the high-order accurate SSP Runge-Kutta methods, e.g. \eqref{eq:RK1}.

Because of the convex decomposition
\begin{align*}
 \vec U_{j,k}   &+ \Delta t {\cal L} (\vec U; {j,k} )
 = \frac{{\hat{\tau} _1 }}{2}\left( {\vec U_{j,k}  + \frac{{2\Delta t}}{{\hat{\tau} _1 \Delta x}}   \widehat{\vec F}_{j - \frac{1}{2},k}^{1,\rm PCP}   } \right)
+ \frac{{\hat{\tau} _1 }}{2}\left( {\vec U_{j,k}  - \frac{{2\Delta t}}{{\hat{\tau} _1 \Delta x}}  \widehat{\vec F}_{j + \frac{1}{2},k}^{1,\rm PCP}  } \right) \\
&   +\frac{{\hat{\tau} _2 }}{2}\left( {\vec U_{j,k}  + \frac{{2\Delta t}}{{\hat{\tau} _2 \Delta y}}  \widehat{\vec F}_{j,k- \frac{1}{2}}^{2,\rm PCP}   } \right)
+ \frac{{\hat{\tau} _2 }}{2}\left( {\vec U_{j,k}  - \frac{{2\Delta t}}{{\hat{\tau} _2 \Delta y}}  \widehat{\vec F}_{j,k+ \frac{1}{2}}^{2,\rm PCP}  } \right),
\end{align*}
with $\hat{\tau} _i = {\tau_i}/(\tau_1 + \tau_2)$, $i=1,2$, and
\begin{equation}\label{eq:CFLtau}
\tau _1  = {(\Delta x)^{-1}}{\mathop {\max }\limits_{j,k} \{\alpha _{j + \frac{1}{2},k} \}},\quad
\tau _2  = {(\Delta y)^{-1}} {{\mathop {\max }\limits_{j,k} \{\alpha _{j,k + \frac{1}{2}}\} }},
\end{equation}
it is convenient to verify that the solutions
of such resulting fully-discrete schemes belong to
 the admissible state set ${\mathcal G}$ under the CFL-type condition
\begin{equation}\label{eq:CFL2D}
\Delta t \le \frac{ \hat w }{ 2(\tau_1 + \tau _2) },
\end{equation}
where $\hat w$ is any positive constant less than one. In practical computations,
the viscosity coefficients may be taken as
\begin{align*}
\alpha _{j + \frac{1}{2},k} &=\vartheta \max \left\{ {\varrho _{j + \frac{1}{2},k}^{x,ROE} ,\varrho_1 \left( {\vec U_{j - r + 1,k} } \right), \cdots ,\varrho_1 \left( {\vec U_{j + r,k} } \right)} \right\},\\
\alpha _{j,k+ \frac{1}{2}} &=\vartheta \max \left\{ {\varrho _{j,k+ \frac{1}{2}}^{y,ROE} ,\varrho_2 \left( {\vec U_{j,k- r + 1} } \right), \cdots ,\varrho_2 \left( {\vec U_{j,k + r} } \right)} \right\},
\end{align*}
where 
$\varrho _{j + \frac{1}{2},k}^{x,ROE}$ (resp. $\varrho _{j,k+ \frac{1}{2}}^{y,ROE}$)
is the spectral radius of the Roe matrix $\widehat{\vec A}_1(\vec U_{j,k},\vec U_{j+1,k}) $ (resp. $\widehat{\vec A}_2(\vec U_{j,k},\vec U_{j,k+1}) $),  see \cite{EulderinkMel:1995}, approximating the Jacobian matrix $ \vec A_1(\vec U)$ (resp. $\vec A_2(\vec U)$).

%
%
%

The above high-order accurate 
finite difference schemes 
can also be extended to the axisymmetric RHD equations in cylindrical coordinates $(r,z)$
 \begin{equation}\label{eq:2Daxis}
    \displaystyle\frac{\partial \vec{U}}{\partial t} +
    \frac{\partial \vec{F}_1(\vec{U})}{\partial r} + \frac{\partial \vec{F}_2(\vec{U})}{\partial z}= \vec S (\vec U,r),
 \end{equation}
where the flux $\vec{F}_i$ is the same as one in \eqref{eq:2D}, $i=1,2$, $r\geq 0$,
and the source term
$$
\vec S (\vec U,r) = -\frac{1}{r} (Dv_1, m_1 v_1,  m_2 v_1, m_1)^T.
$$
Similarly, when the computational domain $\Omega$ in cylindrical coordinates $(r,z)$
is divided into a uniform  mesh with the rectangular cell $\{(r,z)|~r_{j-\frac{1}{2}}
< r < r_{j+\frac{1}{2}}, z_{k-\frac{1}{2}}<z<z_{k+\frac{1}{2}}\}$, where
$r_{j-\frac{1}{2}}=(j-\frac{1}{2})\Delta r$ $j \in \mathbb{Z^+}$and $z_{k+\frac{1}{2}}=(k+\frac{1}{2})\Delta z$,
$j,k \in \mathbb{Z}$, and
$\Delta r$ and $\Delta z$ are spatial stepsizes in $r$- and $z$-directions, respectively,
the extension of the scheme \eqref{eq:2Dtimeds} to the system
\eqref{eq:2Daxis} is
\begin{align} \nonumber
\frac{{d\vec U_{j,k} (t)}}{{dt}}
&=  \frac{   \widehat {\vec F}_{j - \frac{1}{2},k}^{1,\rm PCP}  - \widehat {\vec F}_{j + \frac{1}{2},k}^{1,\rm PCP}     }{{\Delta r}}
 + \frac{   \widehat {\vec F}_{j,k- \frac{1}{2}}^{2,\rm PCP}  - \widehat {\vec F}_{j,k+ \frac{1}{2}}^{2,\rm PCP}     }{{\Delta z}}
 + \vec S(\vec U_{j,k} (t), r_j)\\[1mm] \label{eq:2Daxis-timeds}
&= : {\cal L}(\vec U (t);{j,k})  + \vec S(\vec U_{j,k} (t), r_j),
\end{align}
where the numerical fluxes $\widehat {\vec F}_{j + \frac{1}{2},k}^{1,\rm PCP}$
and $\widehat {\vec F}_{j,k + \frac{1}{2}}^{2,\rm PCP}$ are the same as those used in
\eqref{eq:2Dtimeds}.

The question is whether the scheme \eqref{eq:2Daxis-timeds} is
still physical-constraints-preserving? 
Since the term $\vec U_{j,k} + \Delta t ( {\cal L} (\vec U;{j,k} ) + \vec S(\vec U_{j,k} , r_j) )$ may be decomposed into
$$
(1-\beta) \big( \vec U_{j,k} + \frac{\Delta t} {1-\beta}   {\cal L} (\vec U;{j,k} ) \big)
    + \beta  \big( \vec U_{j,k} + \frac{\Delta t}{\beta} \vec  S(\vec U_{j,k} , r_j) \big),
$$
for any $\beta \in (0,1)$,
it is sufficient  to ensure that $\vec U_{j,k} + \frac{\Delta t} {1-\beta}   {\cal L} (\vec U;{j,k} )   \in {\mathcal G} $ and $\vec U_{j,k} + \frac{\Delta t}{\beta} \vec  S(\vec U_{j,k} , r_j) \in {\mathcal G}$ for each $j,k$
 when $\vec U_{j,k} \in {\mathcal G}$ for all $j,k$. The first part is true if
 the condition \eqref{eq:CFL2D} is replaced with
\begin{equation}
\label{Timestepsize2a}
\Delta t \le  \frac{{\left( {1 - \beta } \right) \hat w}}{{2(\tau _1  + \tau _2 )}},
\end{equation}
while the second part may be ensured if
\begin{equation}
\label{Timestepsize2b}
\Delta t \le \beta A_s,\quad
A_s := \mathop {\min }\limits_{\{ j,k\}  \in {\cal P}_v } \left\{\frac{{j\Delta r q(\vec U_{j,k} )}}{{\left( {p(\vec U_{j,k} ) + q(\vec U_{j,k} )} \right)|v_1 (\vec U_{j,k} )|}}\right\},
\end{equation}
where 
${\cal P}_v = \left\{ {(j,k)\left| {j,k \in \mathbb{Z}, v_1 (\vec U_{j,k} )> 0} \right.} \right\}$. The readers are referred to the following lemma or the similar
discussion in \cite{zhang2011}.
Combining \eqref{Timestepsize2a} with \eqref{Timestepsize2b},
 an optimal value of  $\beta$ is chosen as  $\hat w / ( \hat w + 2 A_s (\tau _1  + \tau _2) )$ such that
$$
\frac{{\left( {1 - \beta } \right) \hat w}}{{2(\tau _1  + \tau _2 )}} = \beta A_s.
$$

\begin{lemma}\label{thm:Source}
If $\vec U \in {\mathcal G}$,
then $\vec U + \Delta t \vec S (\vec U,r) \in {\mathcal G}$ under
\[
\xi:=\frac{{v_1 \Delta t}}{r} \le \frac{{q(\vec U)}}{{p + q(\vec U)}}.
\]
\end{lemma}

\begin{proof}
Assumption that $\vec U \in {\mathcal G}$ implies
$q(\vec U)>0$, $D>0$, and $p>0$.
Thus if $\xi<1$, then
$D(1 - \xi )>0$.
Hence, when $\xi<1$, to ensure
$$
\vec U + \Delta t \vec S (\vec U,r) = \left( {D(1 - \xi ),m_1 (1 - \xi ),m_2 (1 - \xi ),E - (E + p)\xi } \right)^T \in {\mathcal G},
$$
 it is sufficient to have
$$
E - (E + p)\xi > \sqrt{ (D(1 - \xi ))^2 + (m_1 (1 - \xi ))^2 + (m_2 (1 - \xi ))^2 } = (1 - \xi ) \sqrt{ D^2 +  m^2 },
$$
that is
$$q(\vec U)>\xi (p+q(\vec U)). $$
Therefore, if
\[
\xi  <  \frac{{q(\vec U)}}{{p + q(\vec U)}}<1,
\]
then
$\vec U + \Delta t \vec S (\vec U,r)\in {\mathcal G}$.
The proof is completed. \qed
\end{proof}






\section{Numerical experiments}
\label{sec:experiments}


This section conducts some numerical experiments on several ultra-relativistic RHD problems with large Lorentz factor, or strong discontinuities, or low
rest-mass density or pressure etc.
 to verify the accuracy, robustness and effectiveness of the proposed high-order accurate
 physical-constraints-preserving finite difference WENO  schemes.
 It is worth stressing that those ultra-relativistic RHD problems  seriously challenge the numerical schemes.
 To limit  the length of the paper, this section only presents the numerical results obtained by our fifth- and
 ninth-order accurate schemes with the third-order accurate Runge-Kutta time discretization \eqref{eq:RK1},
and for convenience, abbreviate them as ``{\tt PCPFDWENO5}'' and ``{\tt PCPFDWENO9}'', respectively.
Unless otherwise stated, all the computations are restricted to the equation of state \eqref{gamma-law} with
the adiabatic index $\Gamma = 5/3$, and the parameter $\hat w$  in \eqref{eq:CFL-PCP} or
\eqref{eq:CFL2D}
is taken as 0.45 for {\tt PCPFDWENO5} and 0.4 for {\tt PCPFDWENO9}.


\begin{example}[1D Smooth problem] \label{example1Dsmooth}\rm
	This test  is used  to check the accuracy of our schemes, and
	similar to but more ultra than the one simulated in \cite{YangHeTang2011}.
	The initial data for the 1D RHD equations  \eqref{eq:1D} are taken as
$$
	\vec V(x,0)=\big( 1+0.99999\sin(x), 0.99, 0.005\big)^T,
	\quad x\in [0,2\pi),
$$
and  thus the exact solution can be given as follows
$$\vec V(x,t)=\big( 1+0.99999\sin(x-0.99t), 0.99, 0.005\big)^T, \quad x\in[0,2\pi),\ \ t\geq 0.
$$
It describes a RHD sine wave propagating periodically and quickly  in the interval $[0,2\pi)$.

The computational domain  is divided into $N_i$ uniform cells, $i=1,2, \cdots,{\hat i}$,
where $\hat i$  is taken as 6 for {\tt PCPFDWENO5} and 7 for {\tt PCPFDWENO9}.
Here the periodic boundary conditions are specified at the end points $x=0$ and $2\pi$.
 The time stepsize is taken as ${\Delta t} = (0.5\Delta x)^{\frac{5}{3}}$  for  {\tt PCPFDWENO5} and  $(0.5\Delta x)^{\frac{9}{3}}$
 for {\tt PCPFDWENO9} in order to realize high-order accuracy in
 time in the present case.

Tables \ref{tab:1DaccuracyWENO5} and \ref{tab:1DaccuracyWENO9}
list  $l^1$- and $l^\infty$-errors at $t=0.01$ and  corresponding orders
obtained  by using {\tt PCPFDWENO5} and {\tt PCPFDWENO9}, respectively,
 where the order is calculated by $  {- \ln (  {\mbox{error}}_{i} /  {\mbox{error}}_{{i+1}} ) } /{ \ln ( N_i / N_{i+1} )  }$, and
$ {\mbox{error}}_{i}$ denotes  the error estimated on  the mesh of $N_i$ uniform cell.
For comparison,  the errors and convergence rates are listed there for
corresponding  finite difference WENO  schemes without physical-constraints-preserving limiter.
The results show that the theoretical order  may be obtained by both {\tt PCPFDWENO5} and {\tt PCPFDWENO9} and
the physical-constraints-preserving limiter does destroy the accuracy.

\begin{table}[htbp]
  \centering 
    \caption{\small Example \ref{example1Dsmooth}: Numerical  errors  and orders in $l^1$- and $l^\infty$-norms   at $t=0.01$ for  {\tt PCPFDWENO5} and corresponding {\tt WENO5}  without physical-constraints-preserving limiter.
  }
\begin{tabular}{|c||c|c|c|c||c|c|c|c|}
  \hline
\multirow{2}{8pt}{$N_i$}
 &\multicolumn{4}{c||}{\tt WENO5}&\multicolumn{4}{c|}  {\tt PCPFDWENO5}\\
 \cline{2-9}
 &$l^1$ $ {\mbox{error}}$ &$l^1$ order &$l^\infty$ error &$l^\infty$ order  &$l^1$ error &$l^1$ order &$l^\infty$ error &$l^\infty$ order \\
 \hline
8 &1.8713e-3& --     & 4.4614e-4   &--    &1.8713e-3& --    &4.4614e-4 &--\\
16&6.7642e-5&4.79  & 1.5495e-5   &4.85   &6.7642e-5& 4.79&1.5495e-5 &4.85\\
32&1.8277e-6&5.21  &5.1420e-7    &4.91  &1.8277e-6& 5.21&5.1420e-7 &4.91\\
64&5.1951e-8&5.14  &1.6019e-8    &5.00 &5.1951e-8& 5.14&1.6019e-8 &5.00\\
128&1.5403e-9&5.08 &4.9554e-10  &5.01 &1.5403e-9& 5.08&4.9554e-10 &5.01\\
256&4.6747e-11&5.04&1.5215e-11  &5.03 &4.6746e-11& 5.04&1.5102e-11 &5.04\\
\hline
\end{tabular}\label{tab:1DaccuracyWENO5}
\end{table}

\begin{table}[htbp]
  \centering
    \caption{\small
  Same as Table \ref{tab:1DaccuracyWENO5}, except for  {\tt PCPFDWENO9}. }
\begin{tabular}{|c||c|c|c|c||c|c|c|c|}
  \hline
\multirow{2}{8pt}{$N_i$}
 &\multicolumn{4}{c||}{\tt WENO9}&\multicolumn{4}{c|}  {\tt PCPFDWENO9}\\
 \cline{2-9}
 &$l^1$ error &$l^1$ order &$l^\infty$ error &$l^\infty$ order  &$l^1$ error &$l^1$ order &$l^\infty$ error &$l^\infty$ order \\
 \hline
8 &1.2614e-4& --     & 3.0905e-5   &--    &1.2614e-4& --     & 3.0905e-5   &--\\
16&2.2845e-7& 9.11  & 8.5647e-8   & 8.50   &2.2845e-7& 9.11  & 8.5647e-8   & 8.50 \\
24&5.0564e-9& 9.40  &2.3436e-9    & 8.88   &5.0564e-9& 9.40  &2.3436e-9    & 8.88\\
32&3.4424e-10& 9.34  &1.7915e-10    & 8.94 &3.4422e-10& 9.34  &1.7915e-10    & 8.94\\
40&4.3114e-11& 9.31 &2.4253e-11  & 8.96    &4.3155e-11& 9.31 &2.4253e-11  & 8.96 \\
48&8.1007e-12& 9.17&5.0622e-12  & 8.59      &7.9810e-12& 9.26&4.7192e-12  & 8.98\\
56&1.9977e-12& 9.08 &1.1805e-12  & 9.44     &1.9005e-12& 9.31 &1.1804e-12  & 8.99\\
\hline
\end{tabular}\label{tab:1DaccuracyWENO9}
\end{table}

\end{example}


\begin{example}[1D Riemann problem] \label{example1DRiemann}\rm
The second test is a Riemann problem (RP) for
the 1D RHD equations  \eqref{eq:1D}  
with initial data
  \begin{equation}
    \label{eq:1DRiemann}
    \vec V(x,0)=
    \begin{cases}
      (1,0,10^4)^T,\ \ & x < 0.5,\\
      (1,0,10^{-8})^T,\ \ & x>0.5.
    \end{cases}
  \end{equation}
The initial discontinuity will evolve as a strong left-moving rarefaction wave,
 a quickly right-moving contact discontinuity,  and a quickly right-moving shock wave.
 The flow pattern  is similar to  Example 4.2 of \cite{YangHeTang2011},
 but more extreme and difficult because of the appearance of the ultra-relativistic region.
In the present case,  the speeds of  the  contact discontinuity and the shock wave (about  0.986956 and 0.9963757 respectively) are very close to the speed of light.

\begin{figure}[htbp]
  \centering
  \subfigure[$\rho$]
  {\includegraphics[width=0.48\textwidth]{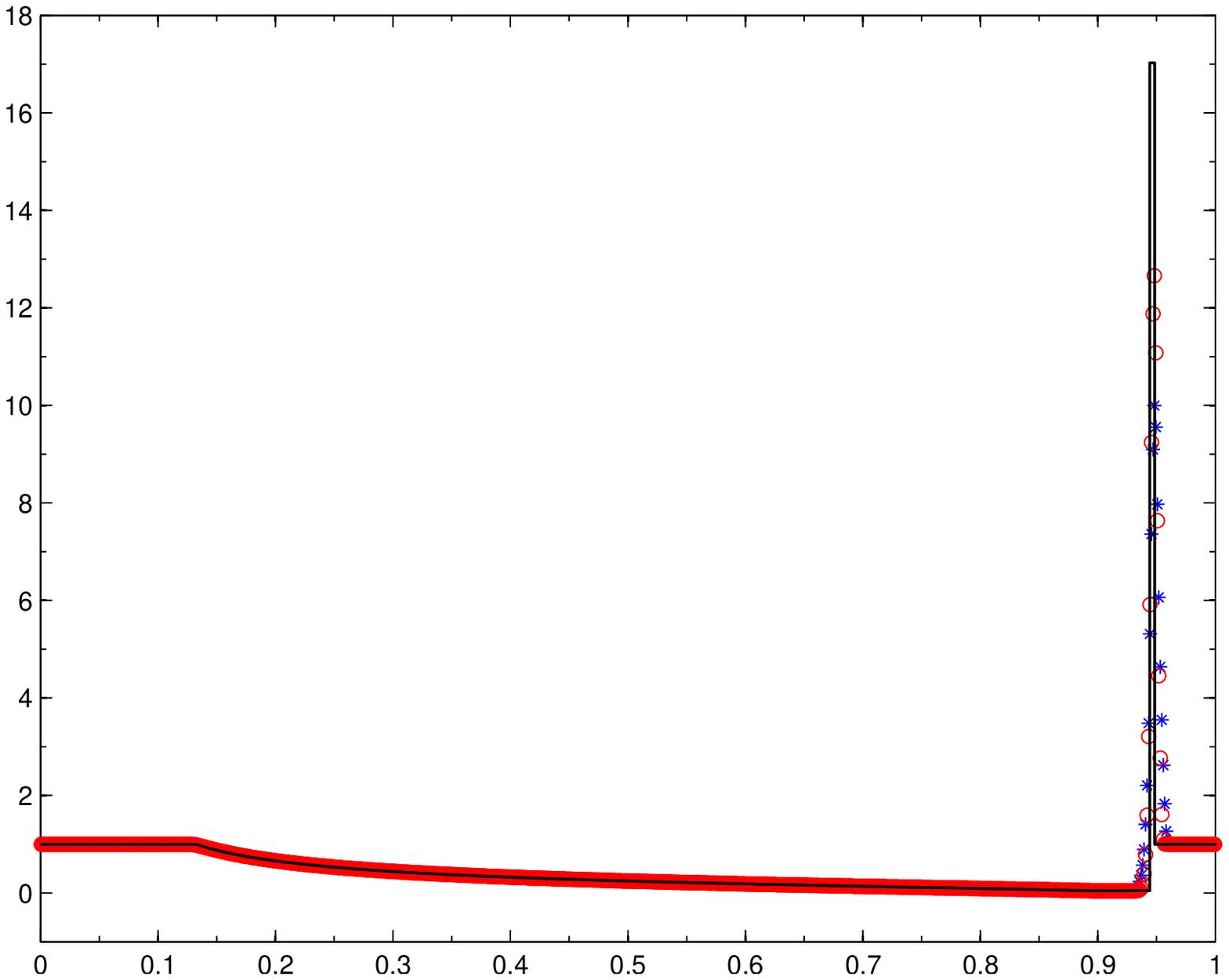}}
  \subfigure[Close-up of $\rho$]
  {\includegraphics[width=0.48\textwidth]{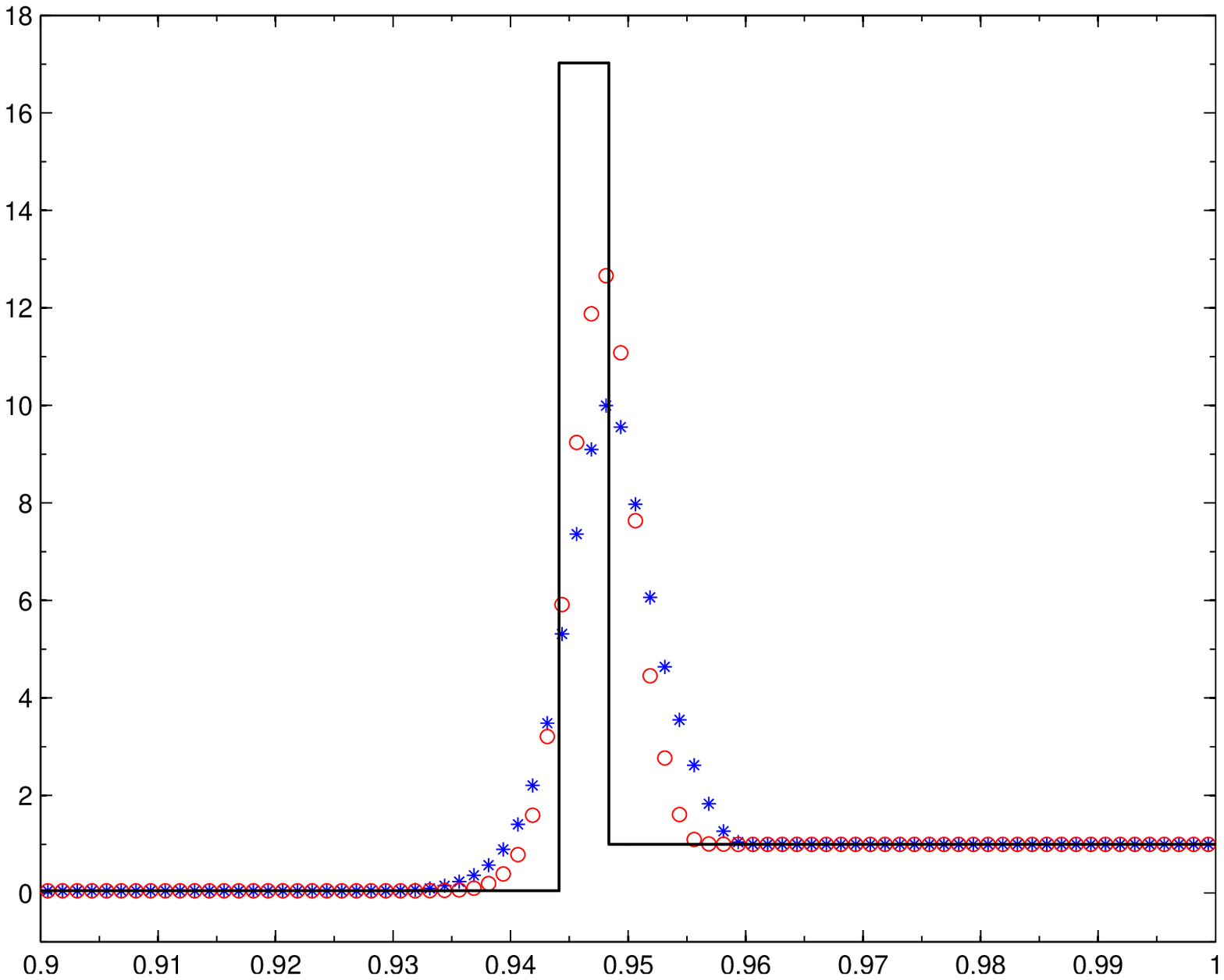}}
  \subfigure[$v_1$]
  {\includegraphics[width=0.48\textwidth]{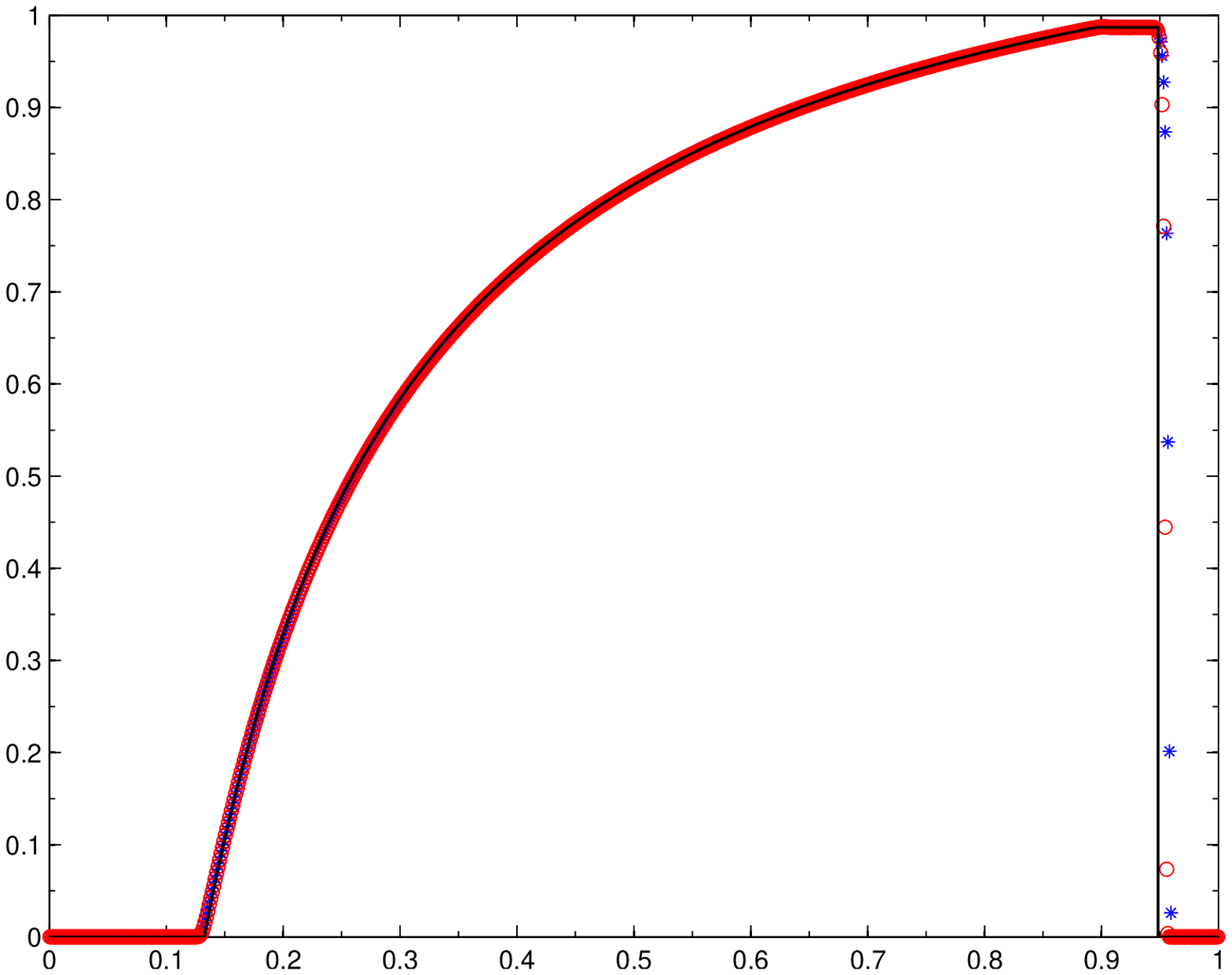}}
  \subfigure[$p$]
  {\includegraphics[width=0.48\textwidth]{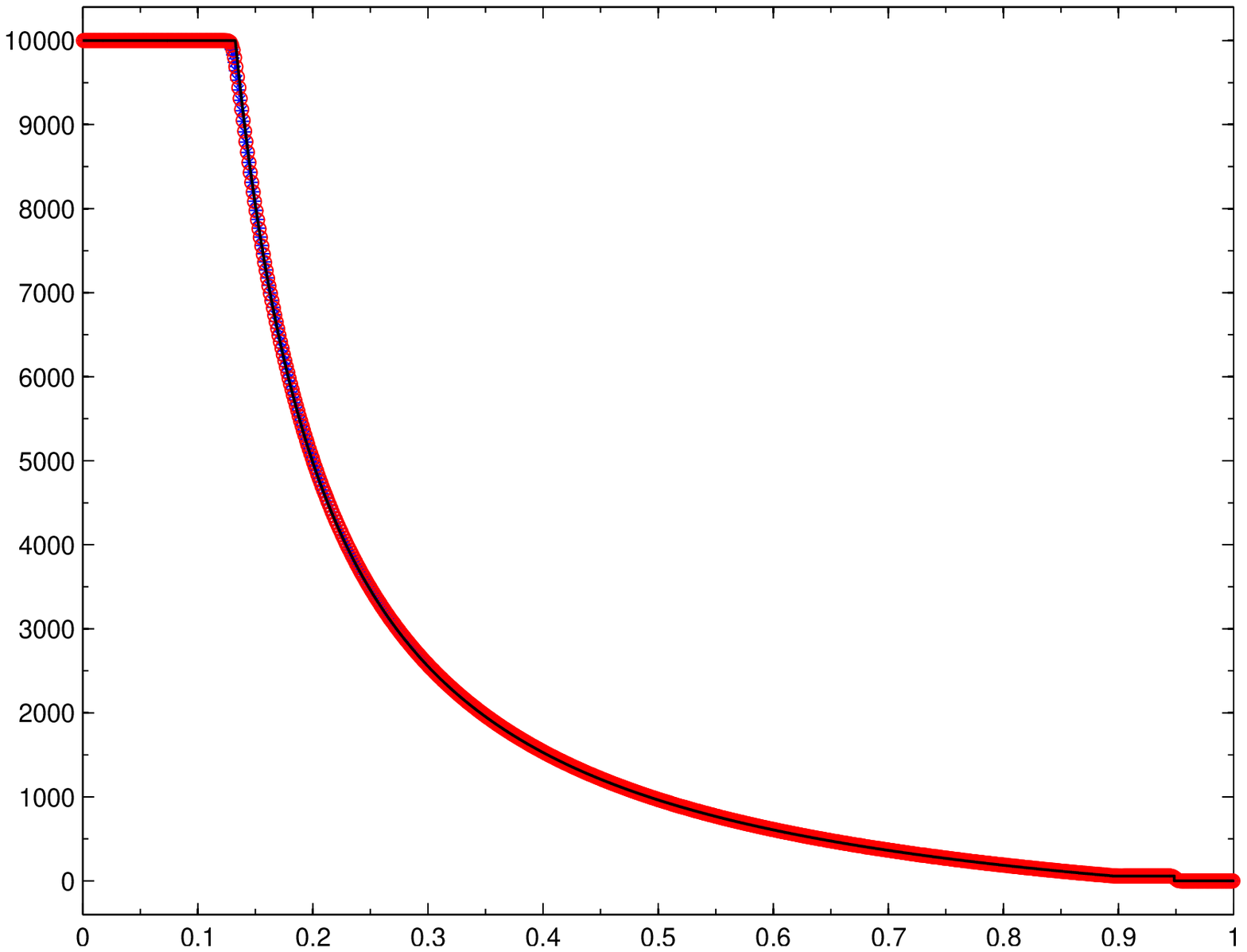}}
  \caption{\small Example \ref{example1DRiemann}: The density $\rho$ and its close-up,
  	the velocity $v_1$, and the pressure $p$ at $t=0.45$ obtained by using  {\tt PCPFDWENO5} (``{\color{blue}$\ast$}'')
  	and {\tt PCPFDWENO9} (``{\color{red}$\circ$}'')
  	with 800 uniform cells.
 }
  \label{fig:1DRiemann}
\end{figure}

Fig. \ref{fig:1DRiemann} displays the numerical results at $t=0.45$ obtained by using {\tt PCPFDWENO5} (``{\color{blue}$\ast$}'') and {\tt PCPFDWENO9} (``{\color{red}$\circ$}'') with  800 uniform cells
within the domain $[0,1]$, where the solid line denotes the exact solution.
It can be seen that  {\tt PCPFDWENO9} exhibits better resolution than  {\tt PCPFDWENO5}, and
they
can well capture the wave configuration except for the extremely narrow region between
the contact discontinuity and the shock wave. The main reason is
that  the region between the shock wave and  the  contact discontinuity is extremely narrow
(its width at $t=0.45$ is about 0.00424) so that
 it can not be well resolved with  800 uniform cells.
At the resolution of 800 uniform cells,  the maximal densities  for {\tt PCPFDWENO5} and {\tt PCPFDWENO9} within the above narrow region  are  about 58.7\% and 74.4\% of the analytic value, respectively.

\end{example}

\begin{example}[Blast wave interaction] \label{exampleBW}\rm
This is an initial-boundary-value problem for the 1D RHD equations  \eqref{eq:1D} and
has been studied in \cite{Marti3,YangHeTang2011}. The same initial setup is considered here.
The adiabatic index $\Gamma$ is taken as $1.4$,
the initial data are taken as follows
\begin{equation}
    \label{eq:BlastInteract}
    \vec V(x,0) =
    \begin{cases}
      (1,0,1000)^T,\ \ & 0<x< 0.1,\\
      (1,0,0.01)^T,\ \ & 0.1<x<0.9,\\
      (1,0,100)^T,\ \ &  0.9<x<1,
    \end{cases}
 \end{equation}
 and outflow boundary conditions are specified at the two ends of the unit interval $[0,1]$.
 This is also a very severe test since it contains the most challenging one-dimensional
relativistic wave configuration, e.g., strong relativistic shock waves, and interaction between blast waves  in a narrow region, etc.

\begin{figure}[htbp]
  \centering
  \subfigure[$\rho$]
  {\includegraphics[width=0.48\textwidth]{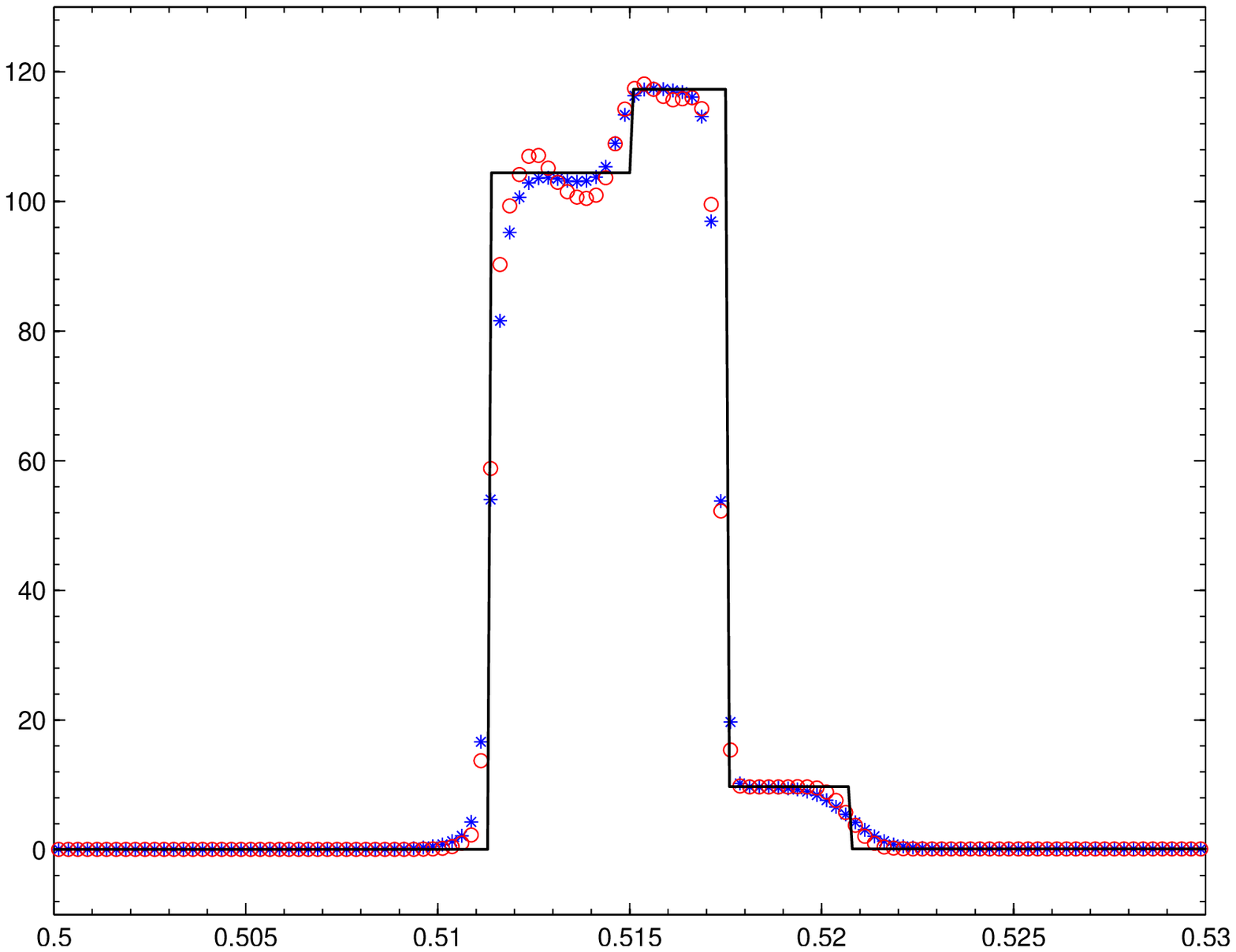}}
  \subfigure[$v_1$]
  {\includegraphics[width=0.48\textwidth]{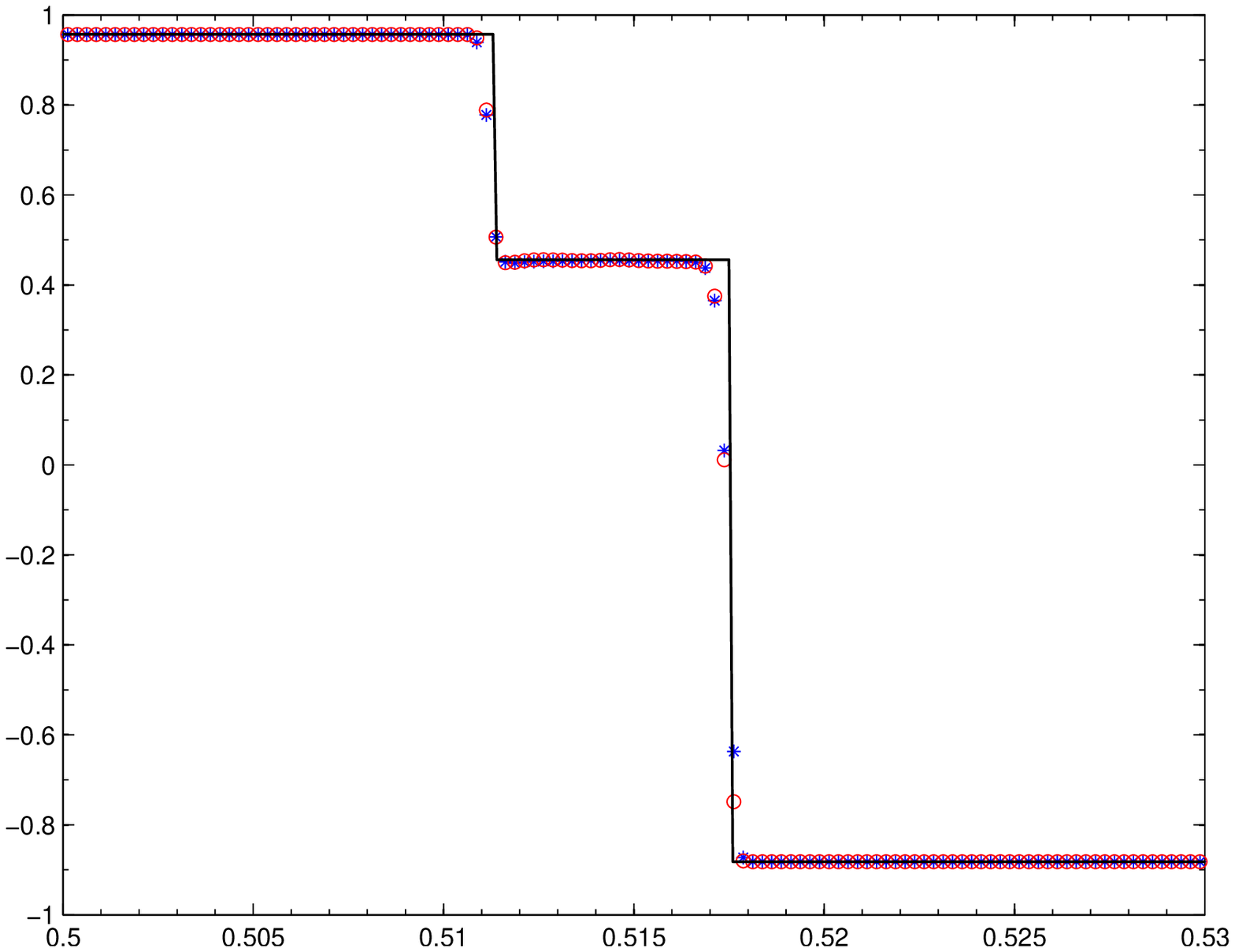}}
  \subfigure[$p$]
  {\includegraphics[width=0.48\textwidth]{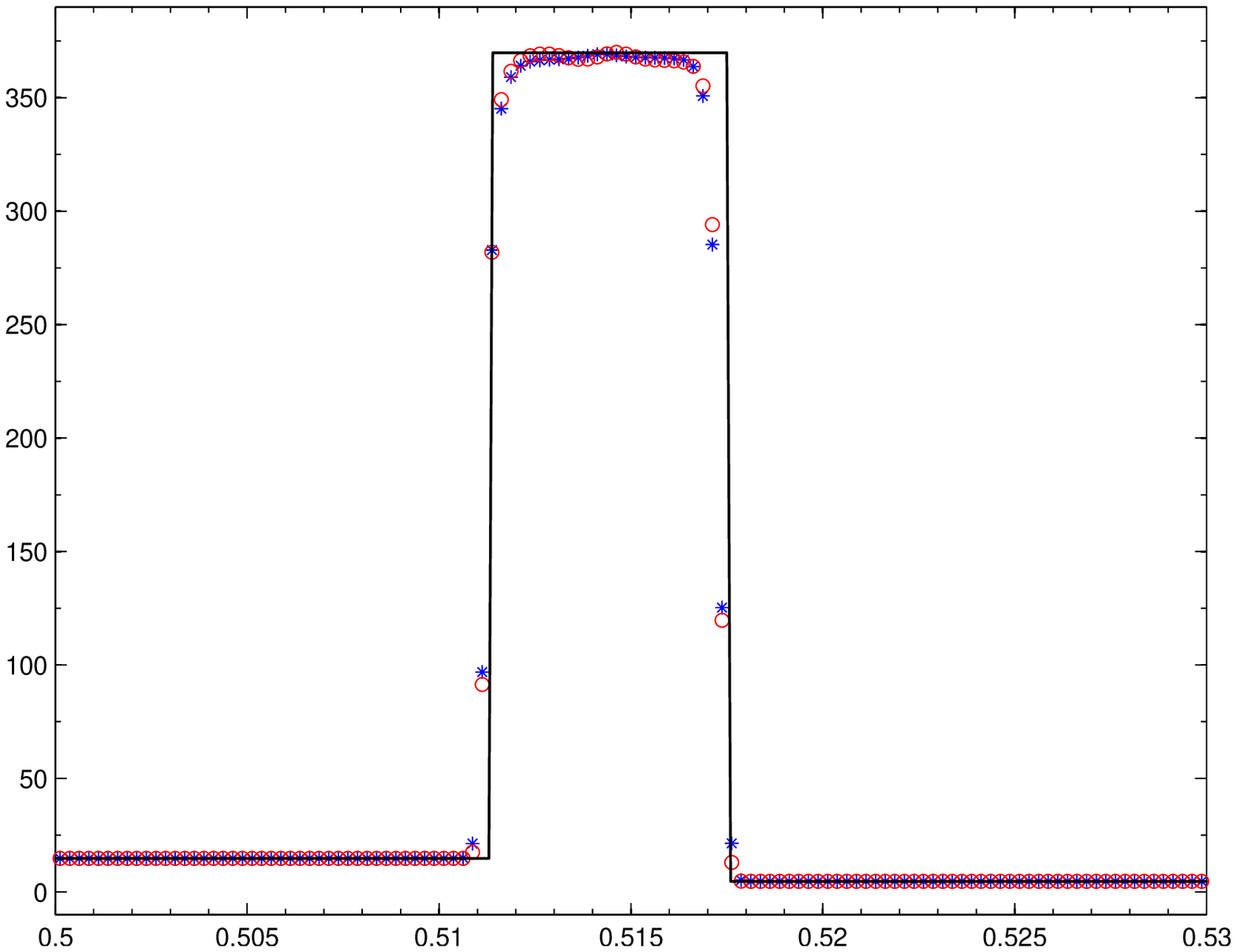}}
  \subfigure[$e$]
  {\includegraphics[width=0.48\textwidth]{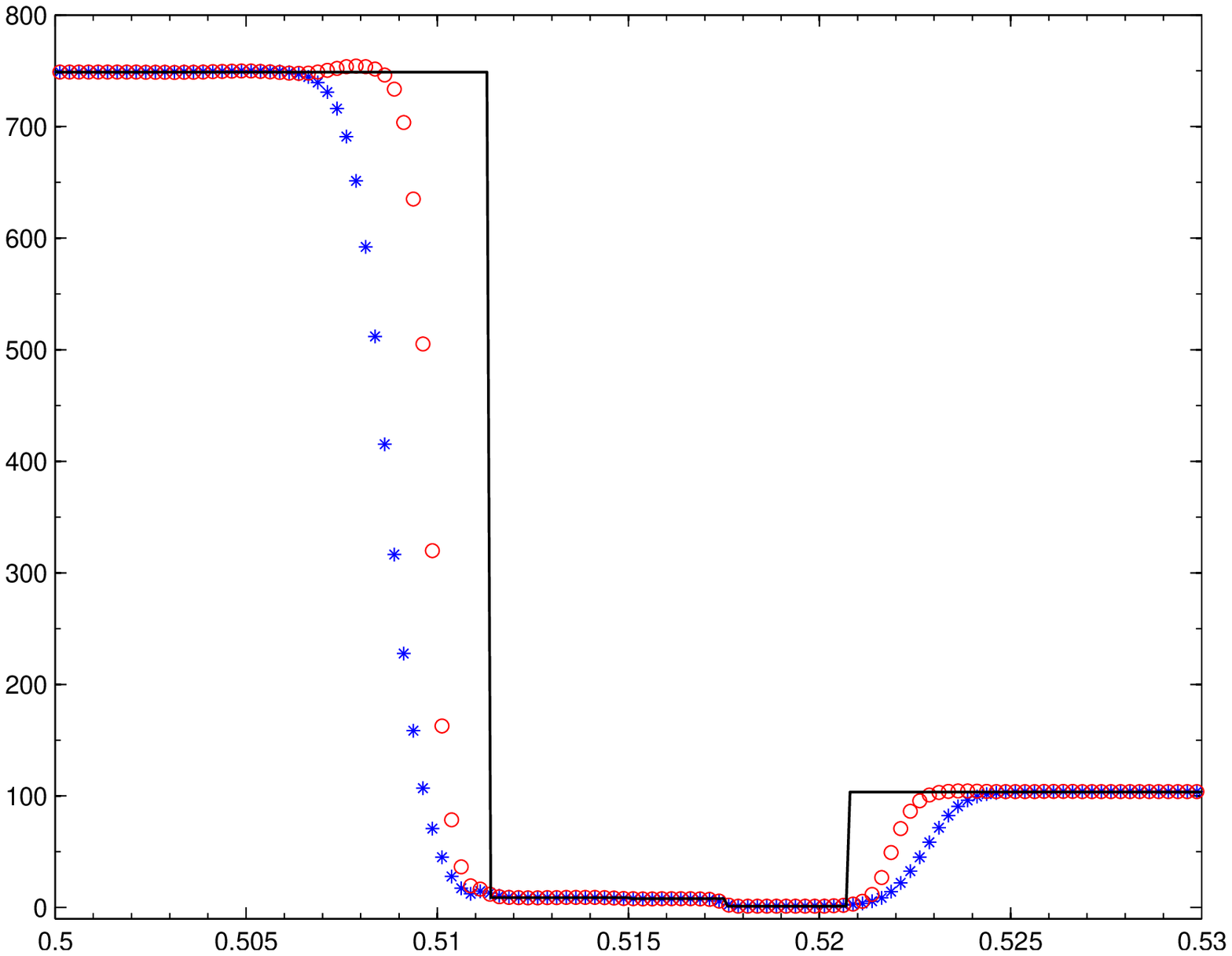}}
  \caption{\small Example \ref{exampleBW}: Close-up of the numerical solutions at $t=0.43$ obtained by using  {\tt PCPFDWENO5} (``{\color{blue}$\ast$}'')
and {\tt PCPFDWENO9} (``{\color{red}$\circ$}'') with 4000 uniform cells. The solid lines denote the exact solutions.
 }
  \label{fig:BWI}
\end{figure}

\begin{figure}[htbp]
  \centering
  \subfigure[$\rho$]
  {\includegraphics[width=0.48\textwidth]{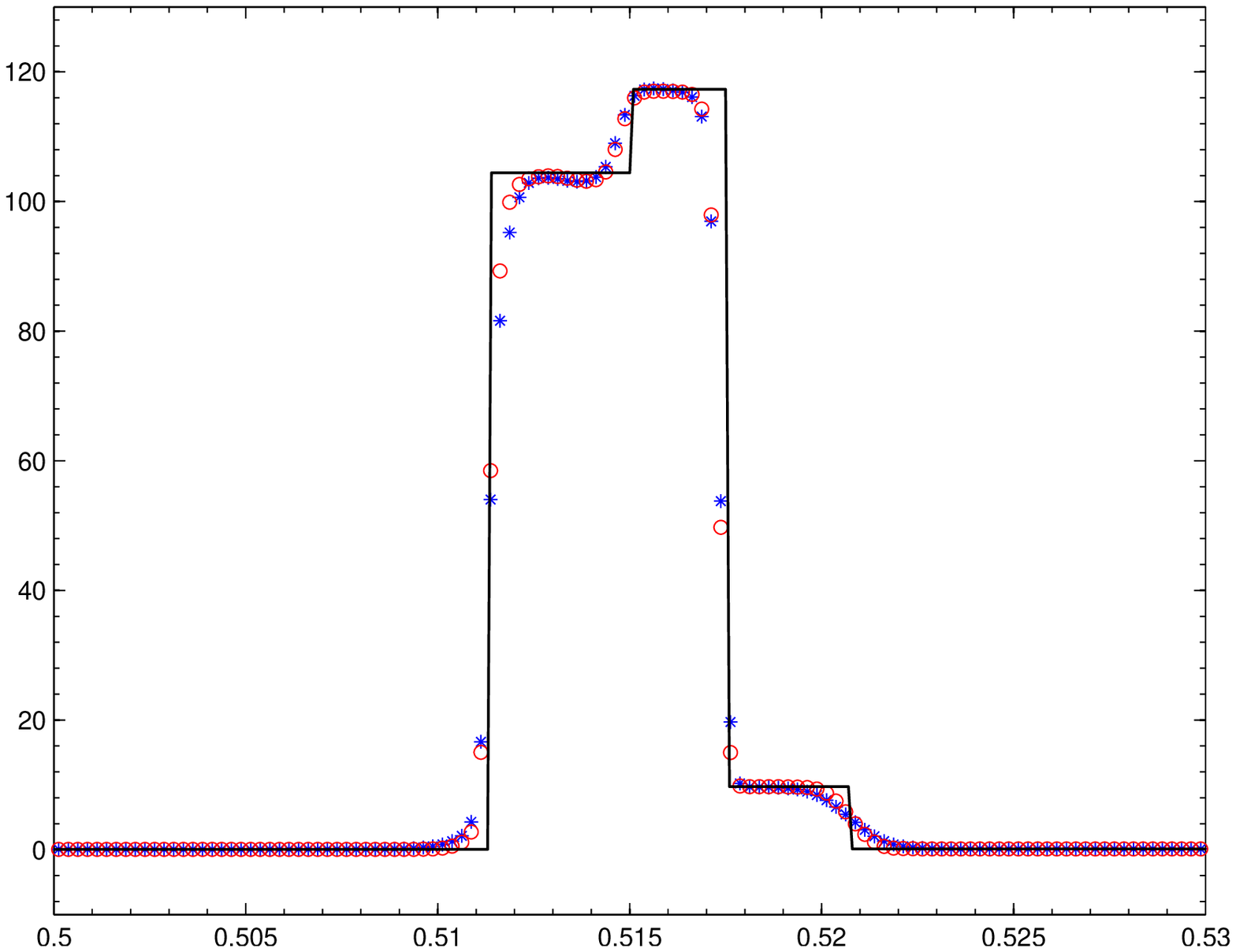}}
  \subfigure[$v_1$]
  {\includegraphics[width=0.48\textwidth]{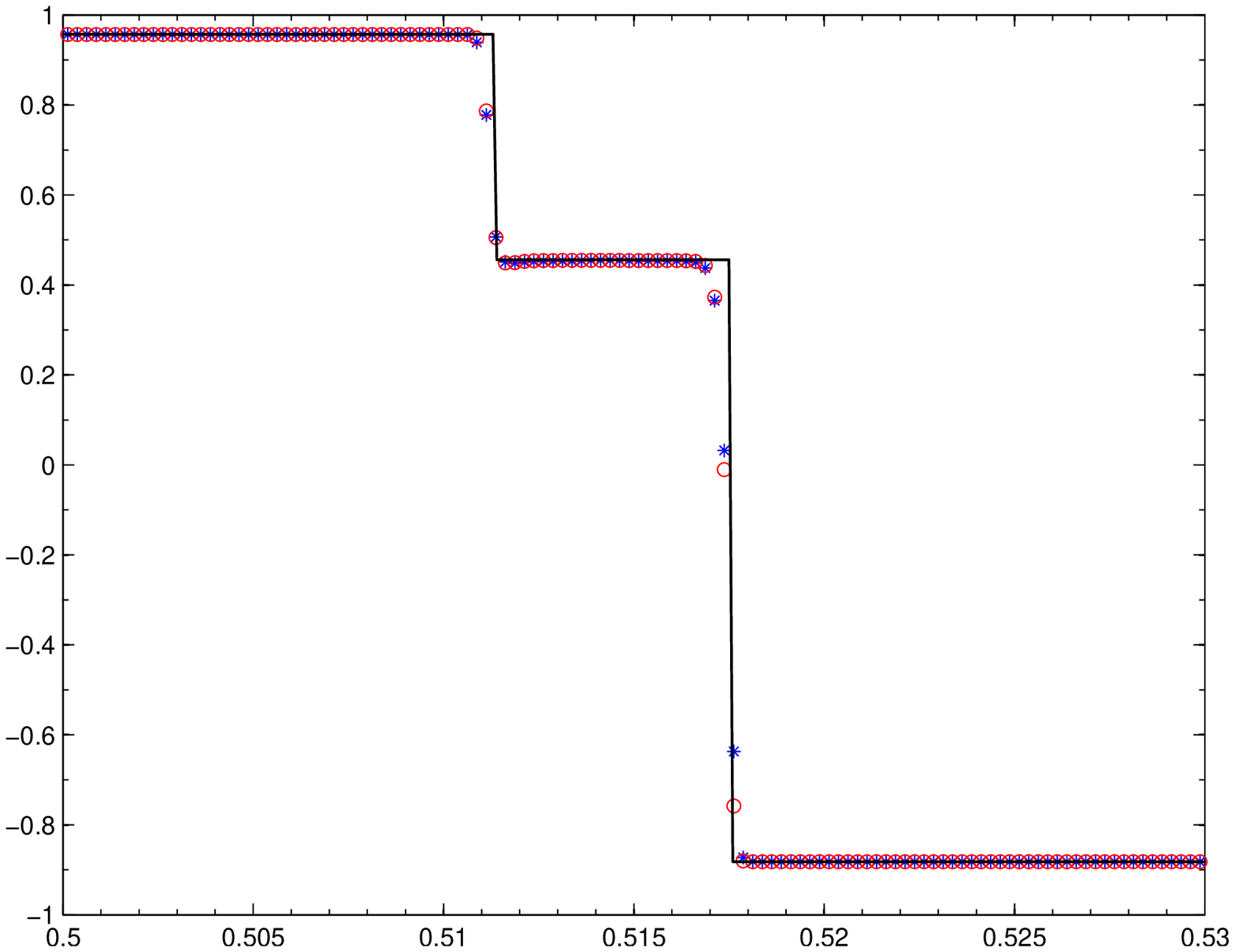}}
  \subfigure[$p$]
  {\includegraphics[width=0.48\textwidth]{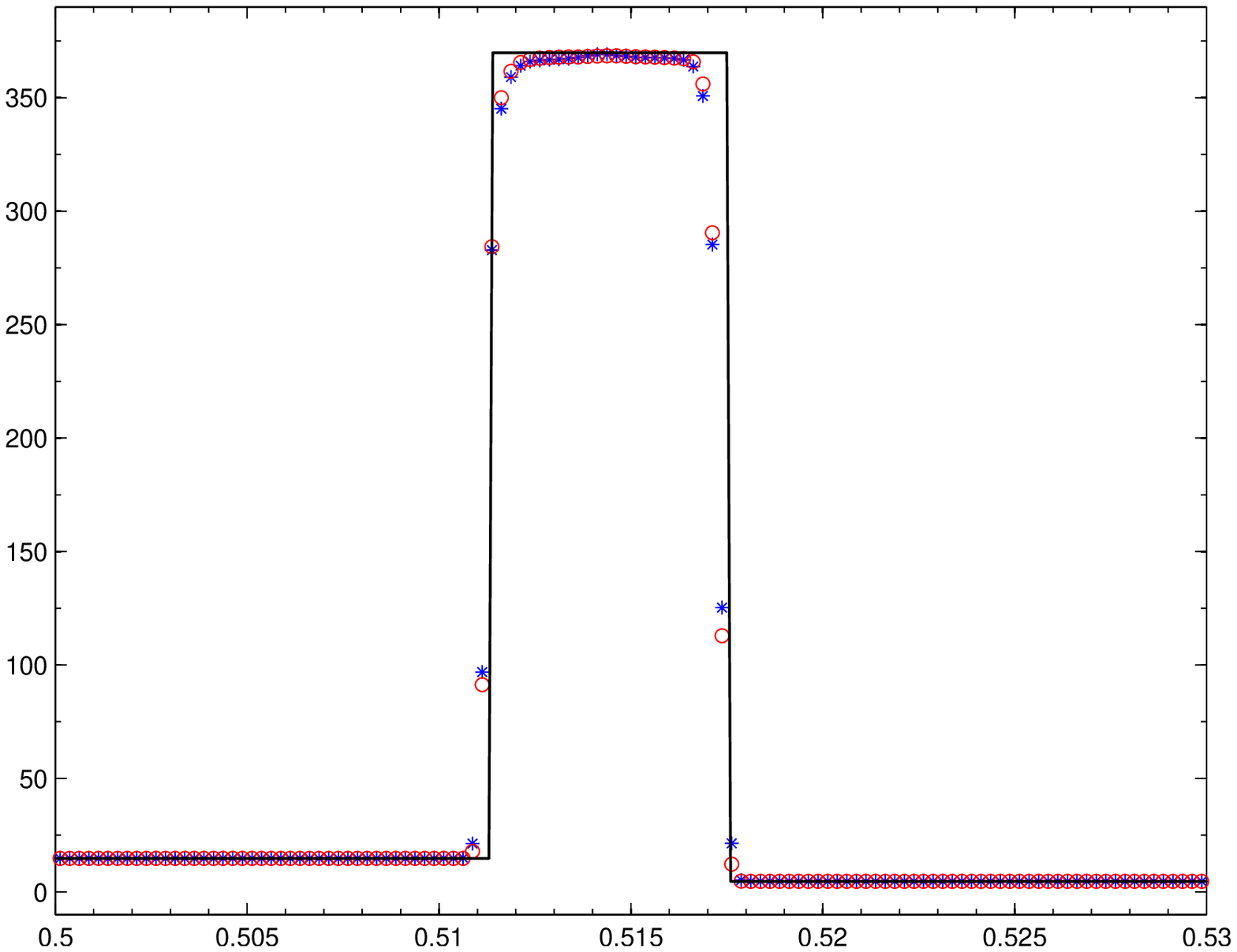}}
  \subfigure[$e$]
  {\includegraphics[width=0.48\textwidth]{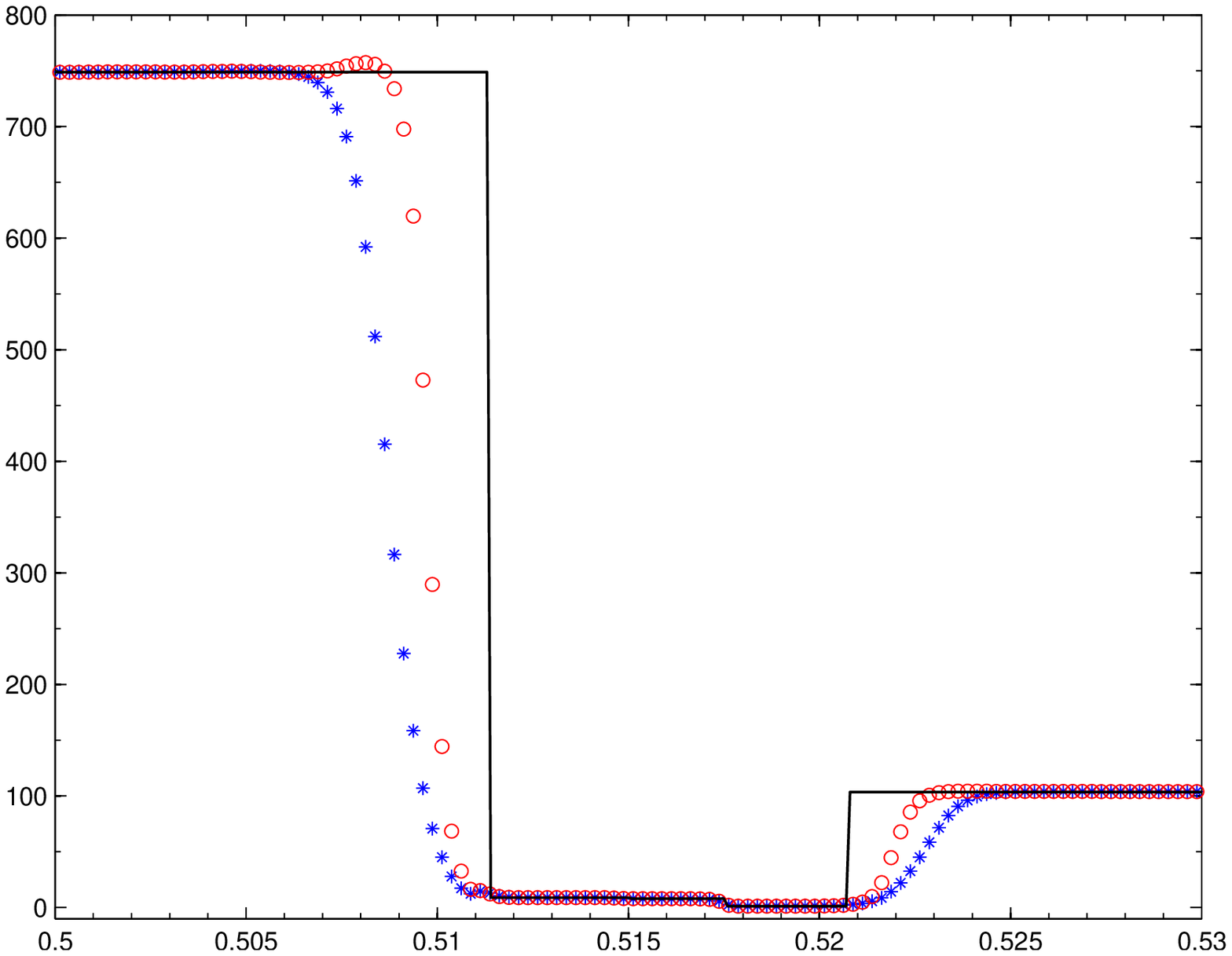}}
  \caption{\small Same as Fig. \ref{fig:BWI}, except for  {\tt PCPFDWENO9} (``{\color{red}$\circ$}'') with a monotonicity-preserving limiter.
 }
  \label{fig:BWI_MP4}
\end{figure}


Fig.~\ref{fig:BWI} gives close-up of the solutions at $t=0.43$ obtained by using {\tt PCPFDWENO5} (``{\color{blue}$\ast$}'') and {\tt PCPFDWENO9} (``{\color{red}$\circ$}'')
with 4000 uniform cells within the domain $[0,1]$.
It is found that the solutions at  $t=0.43$ within the interval $[0.5,0.53]$ consists two shock waves and two
 contact discontinuities since both initial discontinuities evolve and both blast waves collide each other;
and  compared to the GRP scheme in \cite{YangHeTang2011},
both  schemes can well resolve those discontinuities and clearly capture the complex relativistic wave configuration, but {\tt PCPFDWENO9} exhibits better resolution than {\tt PCPFDWENO5} except
 for slight overshoot and undershoot of the rest-mass density  between the left shock and the contact discontinuity.
The  overshoot and undershoot may be suppressed by using
the monotonicity-preserving limiter in \cite{balsara2000}, see Fig.~\ref{fig:BWI_MP4}.

\end{example}

\begin{example}[Shock heating problem] \label{exampleSH}\rm
The last 1D test is to solve the shock heating problem, see \cite{Blandford1976},
by using {\tt PCPFDWENO5} and {\tt PCPFDWENO9}.
The computational domain $[0,1]$ with a  reflecting boundary at the right end
is initially filled with a cold gas (the specific internal energy  is taken as 0.0001).
The gas has  an unit rest-mass density,  the adiabatic index  $\Gamma$ of $4/3$,
 and the velocity $v_0$ of $1-10^{-10}$,
When the initial gas moves toward to the reflecting boundary, the gas is compressed and
  heated as the kinetic energy is converted into the internal energy.
After then, a reflected strong shock wave is formed and propagate to the left  with the speed
$v_s = {(\Gamma -1)W_0 |v_0|}/{(W_0 + 1)}$, where  $W_0=(1-v_0^2)^{-1/2}$ is about 70710.675.
 Behind the reflected shock wave, the gas is at rest and has a specific internal energy of $W_0 - 1$  due to the energy conservation across the shock wave. The compression ratio $\sigma$ across the relativistic shock wave
$$
\sigma= \frac{\Gamma +1}{\Gamma -1} + \frac{\Gamma}{\Gamma -1} (W_0 - 1)
\approx 282845.7,
$$
grows linearly with the Lorentz factor $W_0$ and
towards to infinite as the inflowing gas velocity approaches to speed of light. It is worth noting that  the compression ratio
across the shock wave  in the non-relativistic case is always
bounded by ${(\Gamma+1)}/{(\Gamma -1)}$.



\begin{figure}[htbp]
  \centering
  \subfigure[$\rho$]
  {\includegraphics[width=0.48\textwidth]{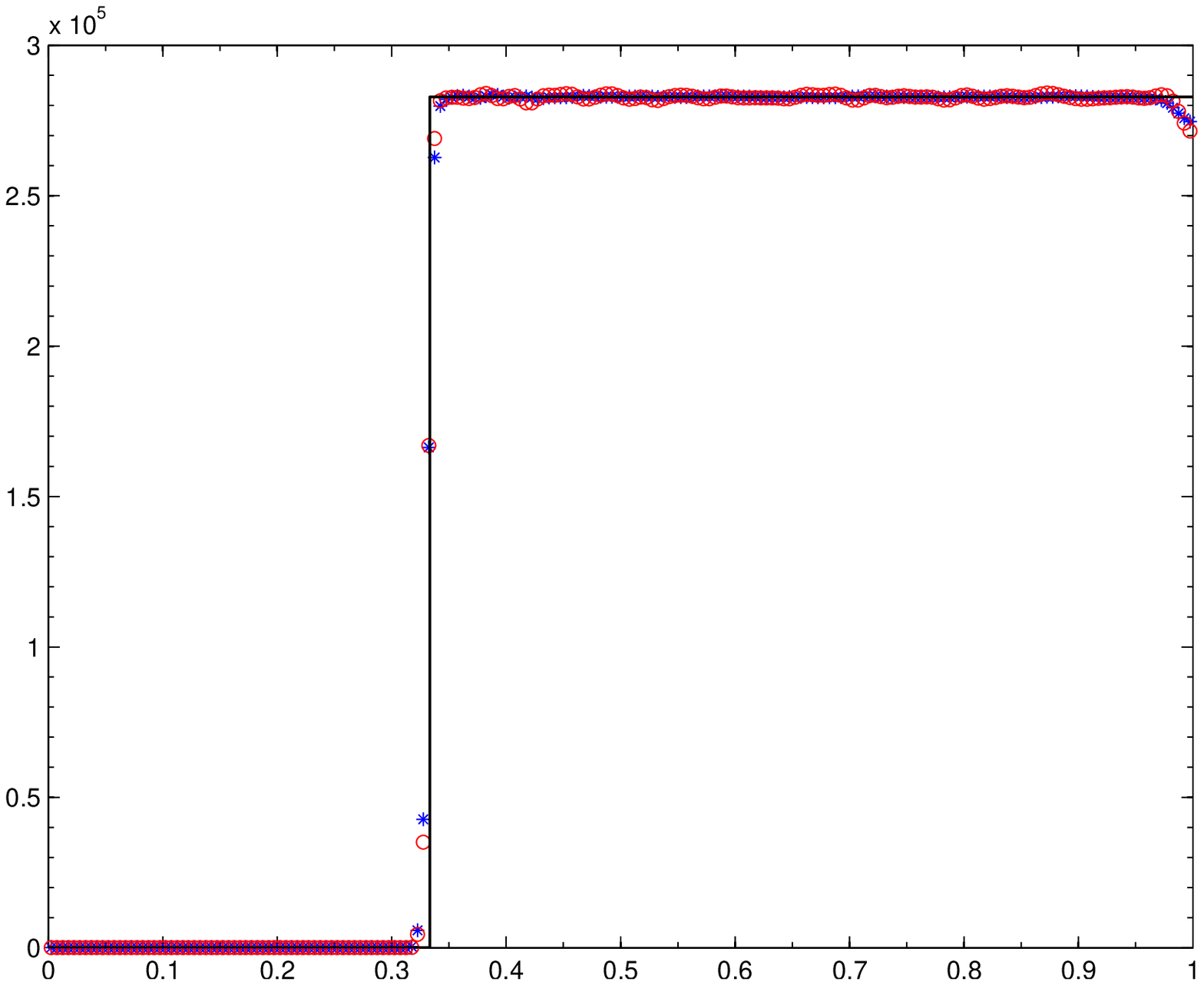}}
  \subfigure[$u$]
  {\includegraphics[width=0.48\textwidth]{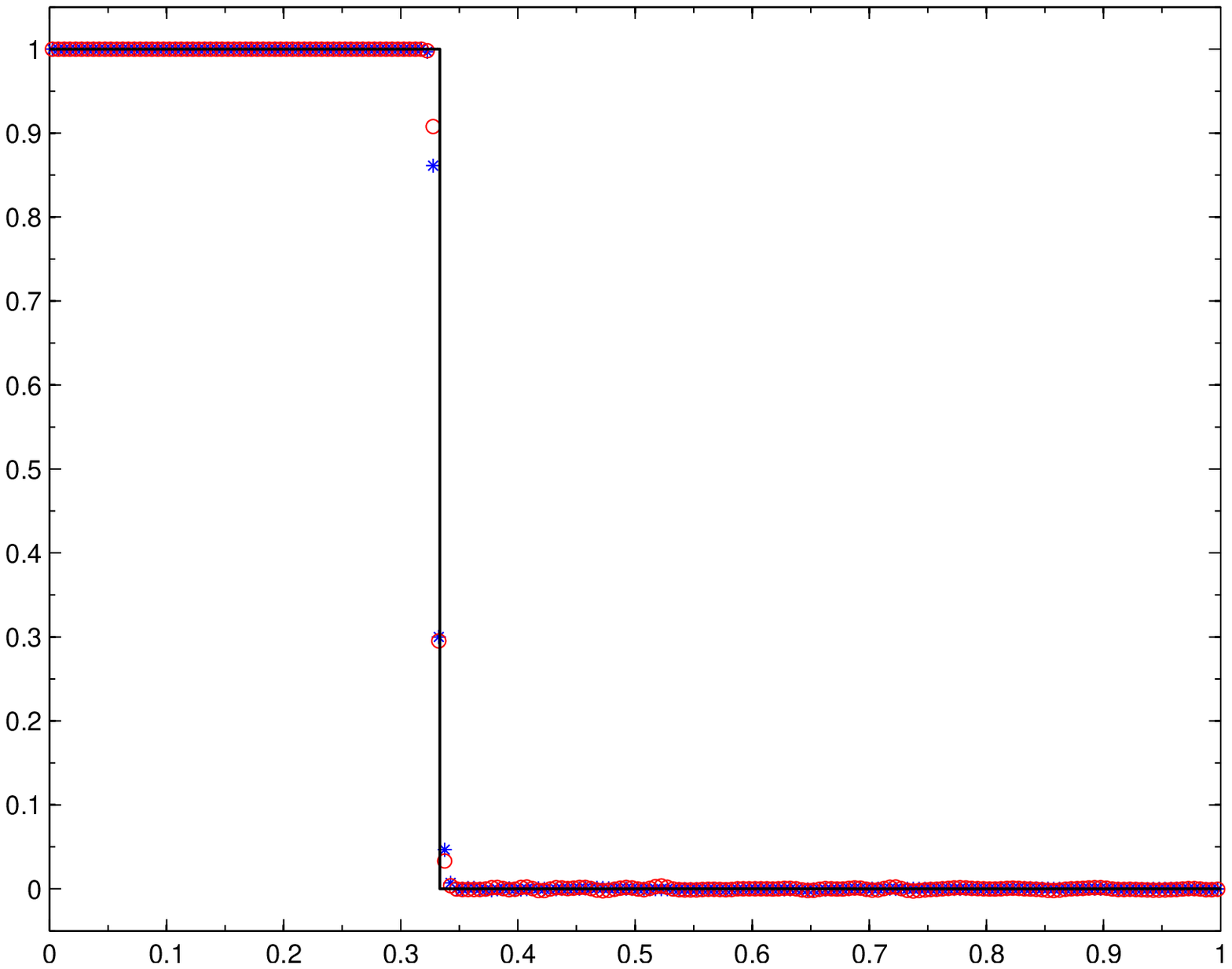}}
  \subfigure[$p$]
  {\includegraphics[width=0.48\textwidth]{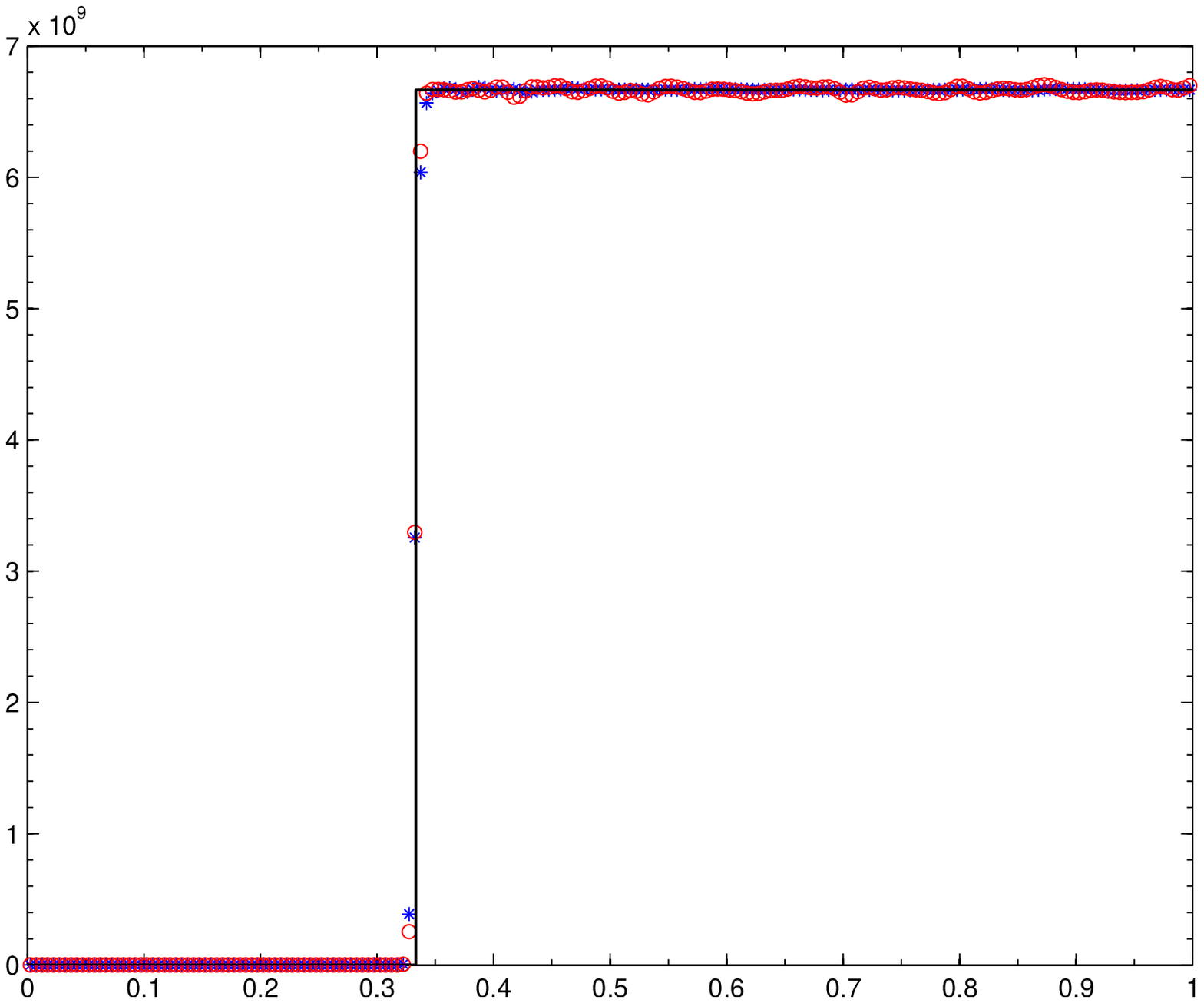}}
  \subfigure[$e$]
  {\includegraphics[width=0.48\textwidth]{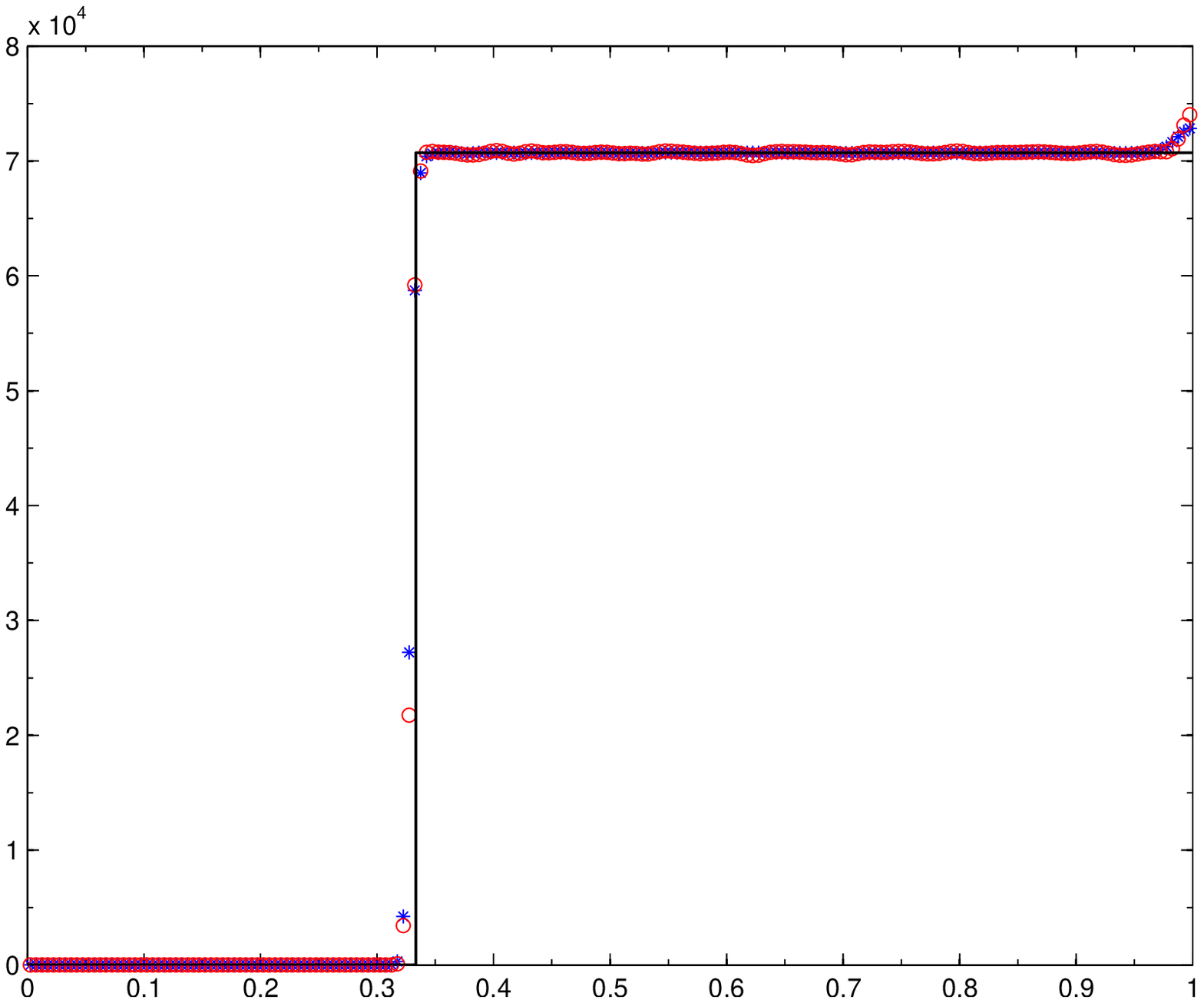}}
  \caption{\small Example \ref{exampleSH}: The density $\rho$, the velocity $v_1$,
  	the pressure $p$, and the internal energy $e$ at $t=2$ obtained by using  {\tt PCPFDWENO5} (``{\color{blue}$\ast$}'') and {\tt PCPFDWENO9} (``{\color{red}$\circ$}'')  with  200 uniform cells. The solid lines denote the exact solutions.
 }
  \label{fig:SH}
\end{figure}

Fig. \ref{fig:SH} displays
the numerical solutions at $t=2$ obtained by using  {\tt PCPFDWENO5} (``{\color{blue}$\ast$}'')
and {\tt PCPFDWENO9} (``{\color{red}$\circ$}'')  with 200 uniform cells.
It shows that both schemes exhibit good robustness for this ultra-relativistic problem
and  good resolution for the strong shock wave, even though there exist slight oscillations
in the rest-mass density and the internal energy behind the shock wave
as well as well-known wall-heating phenomenon near the reflecting boundary $x=1$.
Similar to Example \ref{exampleBW},
those small oscillations can also be efficiently alleviated by using
the monotonicity-preserving limiter \cite{balsara2000}.  To save the paper space, the results are not presented here.

\end{example}

	Besides the flux limiter  in Section \ref{sec:reviweWENO}, 
	we have also tried to  extend  the parametrized flux limiter \cite{Xu_MC2013,Liang2014,XiongQiuXu2014}
	to the 1D RHD equations \eqref{eq:1D},
	the numerical results show that the parametrized flux limiter is less restrictive
	on the CFL number in preserving high order accuracy and has slightly better resolution for Example \ref{example1DRiemann}	than the flux limiter in Section \ref{sec:reviweWENO}.


\begin{example}[2D Riemann problem] \label{example2DRPs}\rm
The non-relativistic 2D RPs are theoretically studied
for the first time in \cite{ZhangZheng1990}. After then, the 2D RPs  become benchmark tests for verifying the accuracy and resolution of numerical schemes, see \cite{SchulzRinne1993,LaxLiu1998,HanLiTang2011,WuYangTang2014b,WuTang2014,YangTang2012}.

Initial data of two RPs of 2D RHD equations  \eqref{eq:2D} considered here
comprise four different constant states in the unit square $\Omega=[0,1]\times[0,1]$, while initial discontinuities parallel to
both coordinate axes respectively.

The initial data of the first RP
 are
$$\vec V({x},{y},0)=
\begin{cases}(0.1,0,0,0.01)^T,& x>0.5,y>0.5,\\
  (0.1,0.99,0,1)^T,&    x<0.5,y>0.5,\\
  (0.5,0,0,1)^T,&      x<0.5,y<0.5,\\
  (0.1,0,0.99,1)^T,&    x>0.5,y<0.5,
  \end{cases}$$
where both the left and bottom discontinuities are contact discontinuities with a jump in the transverse velocity, while both the right and top discontinuities are not  simple waves. Note that this test is different from the case in \cite{YangTang2012}.

Fig.~\ref{fig:2DRP1} gives the contours of the density logarithm $\ln \rho$ at time $t = 0.4$ obtained by using
{\tt PCPFDWENO5} and {\tt PCPFDWENO9}
with several different mesh resolutions.
It is found that four initial discontinuities interact each other
and form two reflected curved shock waves, an elongated jet-like spike,
which is approximately between two points (0.7,0.7) and (0.9,0.9) on the diagonal
$x=y$ when $t = 0.4$,
and a complex mushroom structure starting from the point (0.5,0.5) respectively and  expanding to the bottom-left region;
both {\tt PCPFDWENO5} and {\tt PCPFDWENO9} well capture these complex
wave configuration, and
 it is obvious that with the same mesh of $400 \times 400$ uniform cells, {\tt PCPFDWENO9}  gets better resolution of
 discontinuities than {\tt PCPFDWENO5},  and the solutions obtained by
 {\tt PCPFDWENO9} with the mesh of $800 \times 800$ uniform cells are comparable  to those obtained by {\tt PCPFDWENO5} with a finer mesh
 of $1200 \times 1200$ uniform cells.
For a further comparison,
the rest-mass densities are plotted along the line $y=x$ and $y=1$, see
Fig. \ref{fig:2DRP1_Comparison}. Those plots  validate the above observation.

\begin{figure}[htbp]
  \centering
  \subfigure[${\tt PCPFDWENO5}$ with $400 \times 400$ uniform cells]
  {\includegraphics[width=0.48\textwidth]{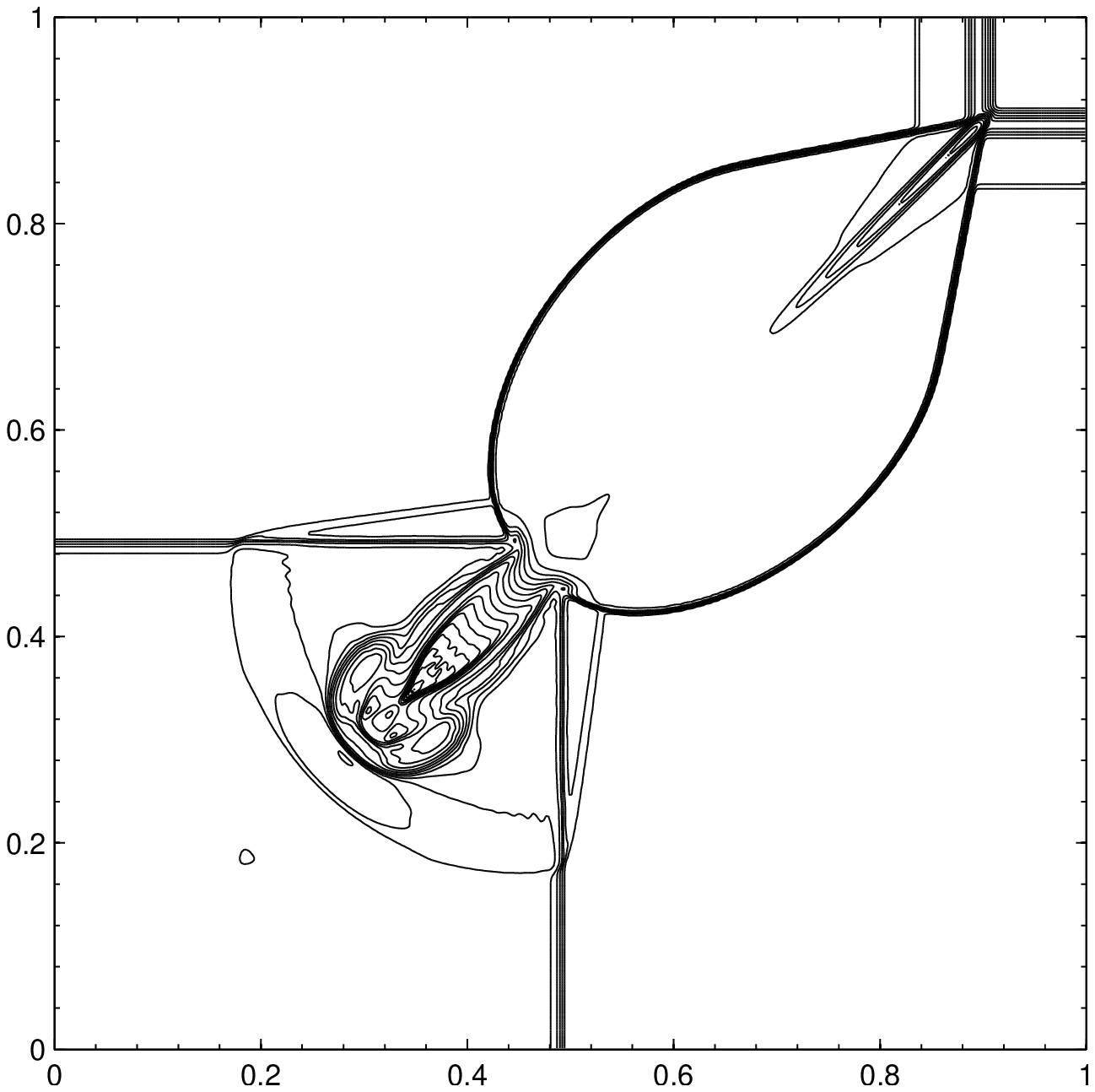}}
  \subfigure[${\tt PCPFDWENO9}$ with $400 \times 400$ uniform cells]
  {\includegraphics[width=0.48\textwidth]{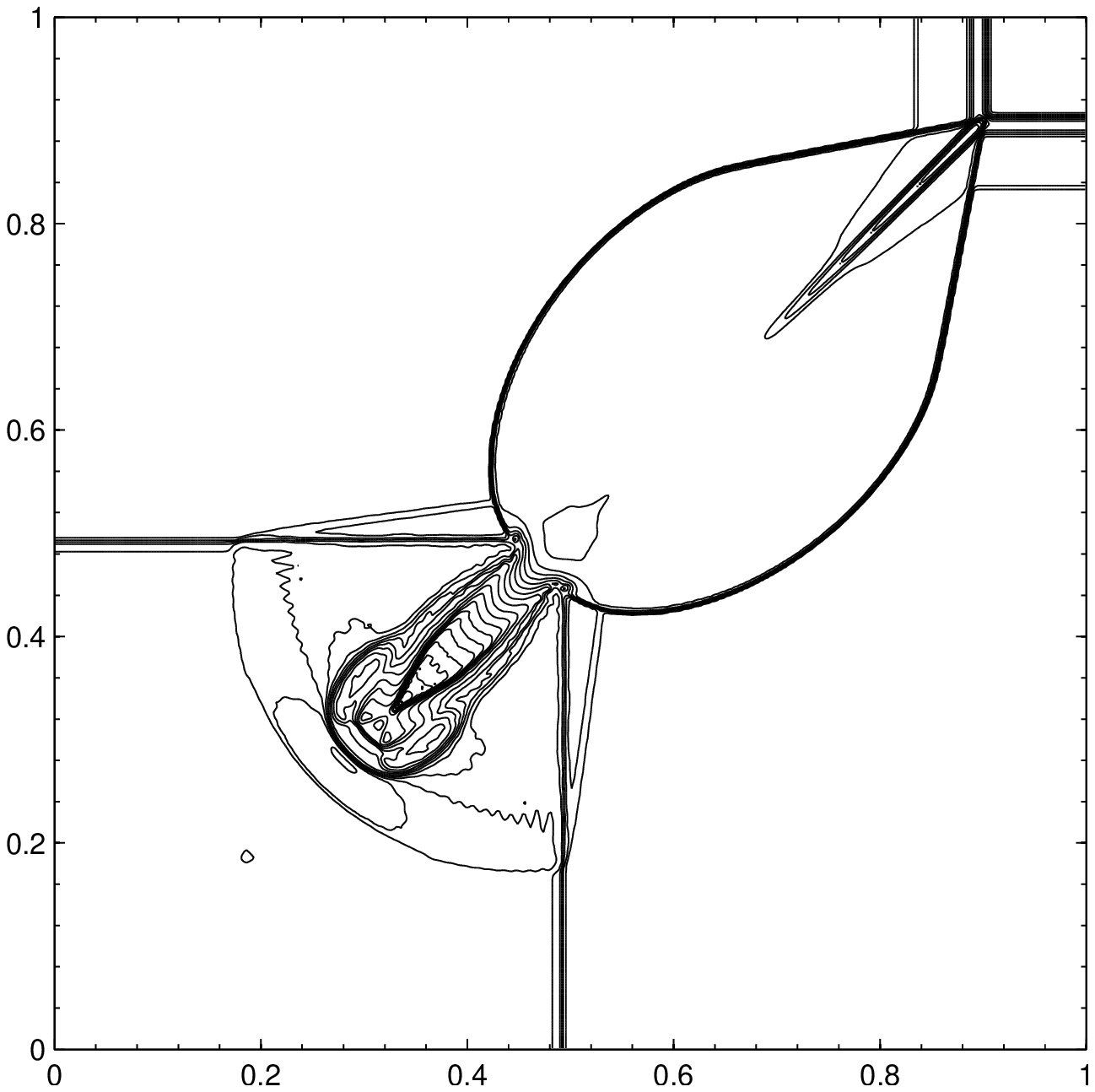}}
  \subfigure[${\tt PCPFDWENO5}$ with $1200 \times 1200$ uniform cells]
  {\includegraphics[width=0.48\textwidth]{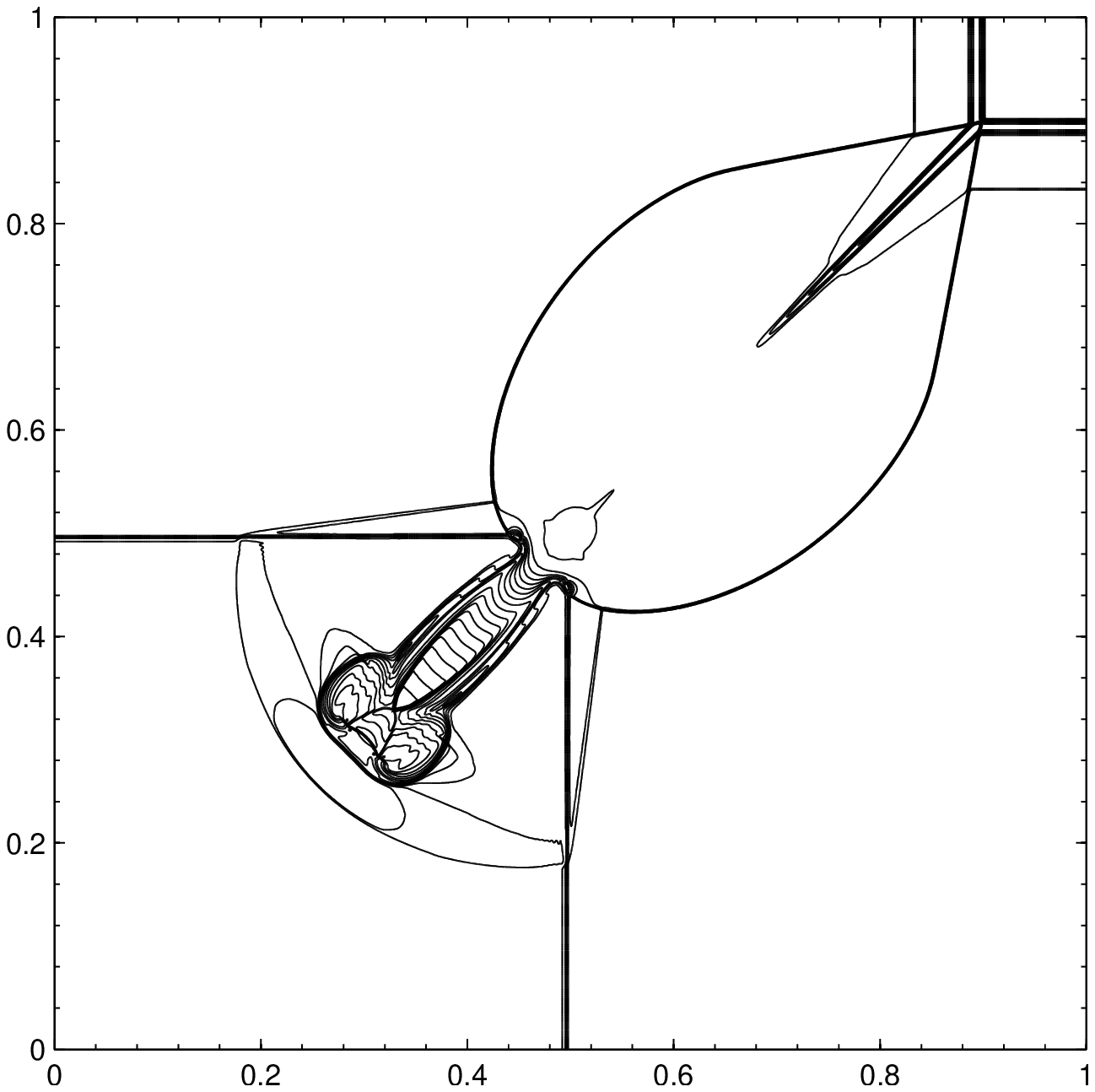}}
  \subfigure[${\tt PCPFDWENO9}$ with $800 \times 800$ uniform cells]
  {\includegraphics[width=0.48\textwidth]{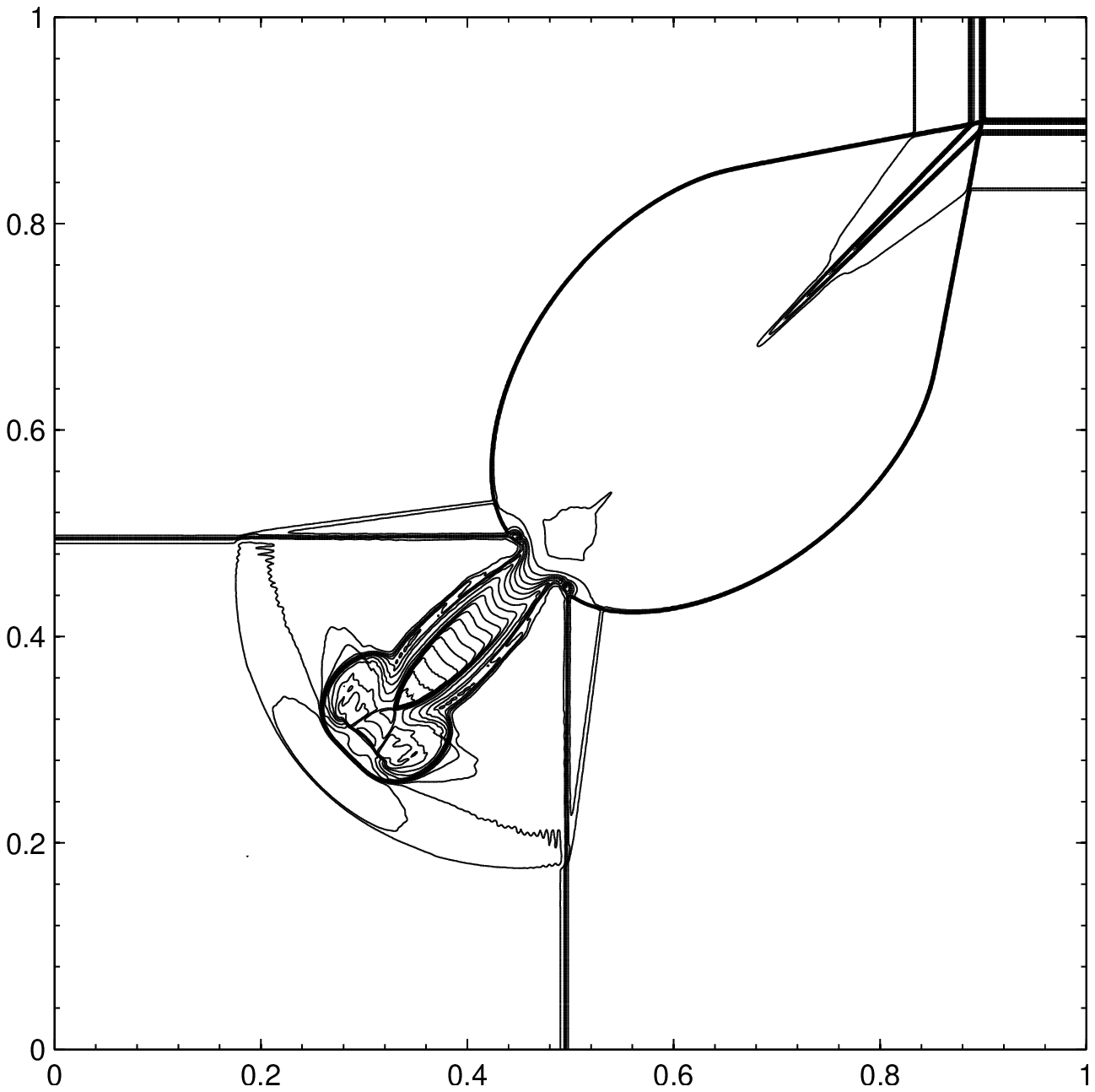}}
  \caption{\small The first 2D RP in Example \ref{example2DRPs}:
  The contours of the density
       logarithm $\ln \rho$ at $t=0.4$ within the domain $[0,1] \times [0,1]$ obtained by using {\tt PCPFDWENO5} and {\tt PCPFDWENO9}
         (25 equally spaced contour lines from $-6$ to $1.9$).
 }
  \label{fig:2DRP1}
\end{figure}

\begin{figure}[htbp]
  \centering
  \subfigure[]
  {\includegraphics[width=0.48\textwidth]{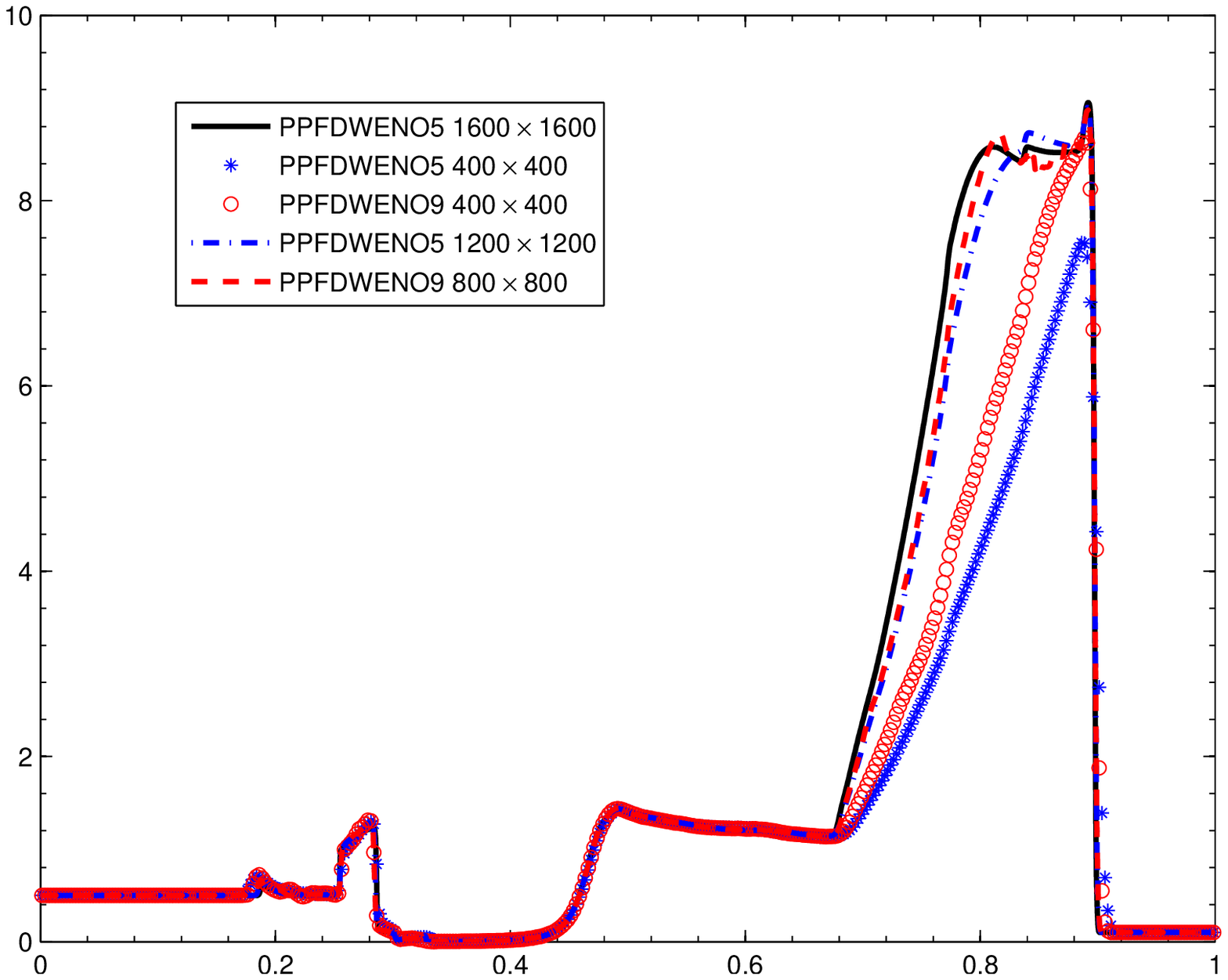}}
  \subfigure[]
  {\includegraphics[width=0.48\textwidth]{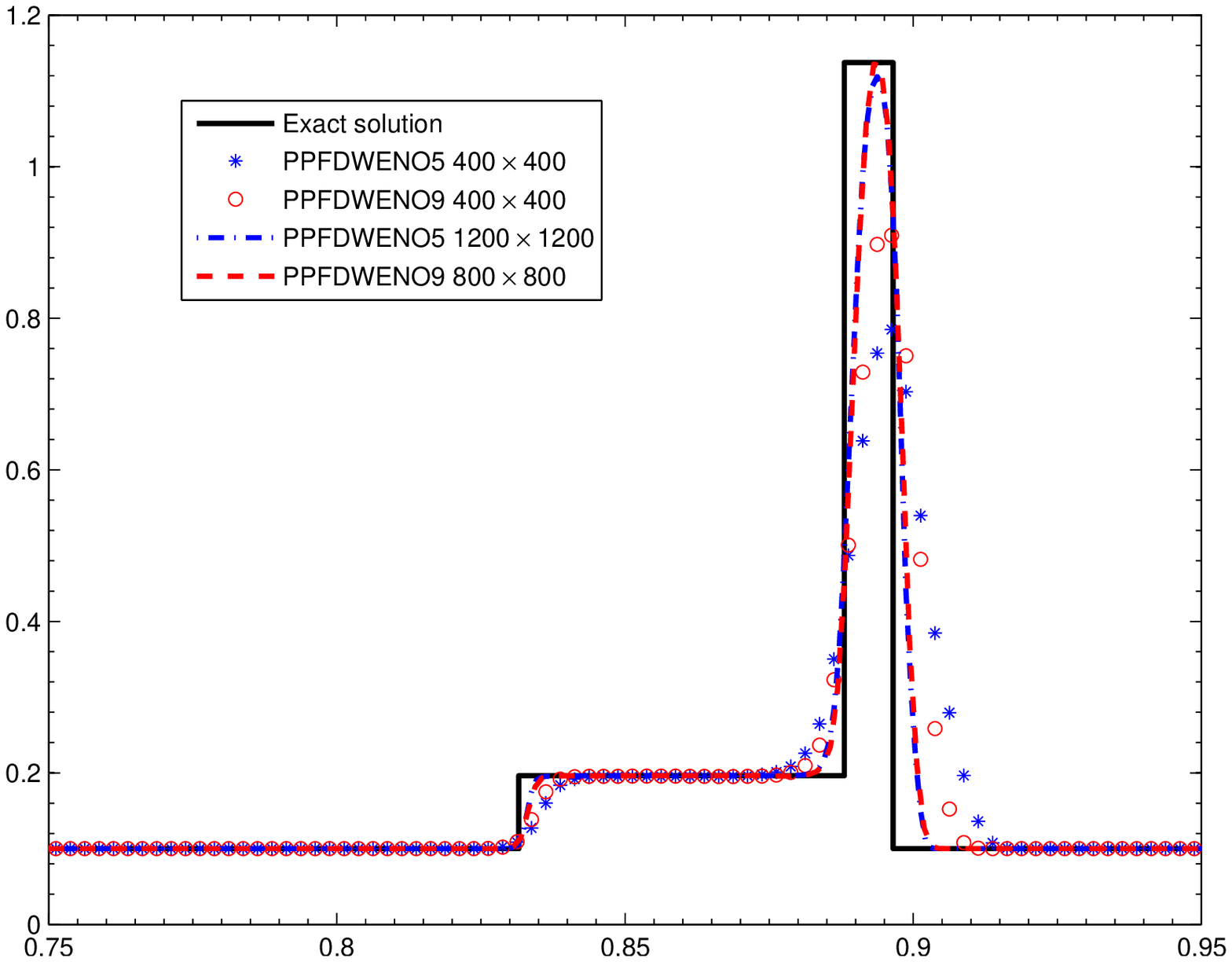}}
  \caption{\small Same as Fig. \ref{fig:2DRP1}, except for
    $\rho$ along the line $y=x$  within  the scaled interval $[0,1]$ (left) and  $\rho$ along the line $y=1$ within the closed interval $[0.75,0.95]$ (right).
 }
  \label{fig:2DRP1_Comparison}
\end{figure}

The initial data of the second 2D RP are
$$\vec V({x},{y},0)=
\begin{cases}(0.1,0,0,20)^T,& x>0.5,y>0.5,\\
  (0.00414329639576,0.9946418833556542,0,0.05)^T,&    x<0.5,y>0.5,\\
  (0.01,0,0,0.05)^T,&      x<0.5,y<0.5,\\
  (0.00414329639576,0,0.9946418833556542, 0.05)^T,&    x>0.5,y<0.5,
  \end{cases}$$
where both the left and bottom discontinuities are  contact discontinuities while both the
top and right are  shock waves with the speed of $-0.66525606186639$.
As the time increases, the maximal value of the fluid velocity may be very close to the speed of light.

\begin{figure}[htbp]
  \centering
  \subfigure[${\tt PCPFDWENO5}$]
  {\includegraphics[width=0.48\textwidth]{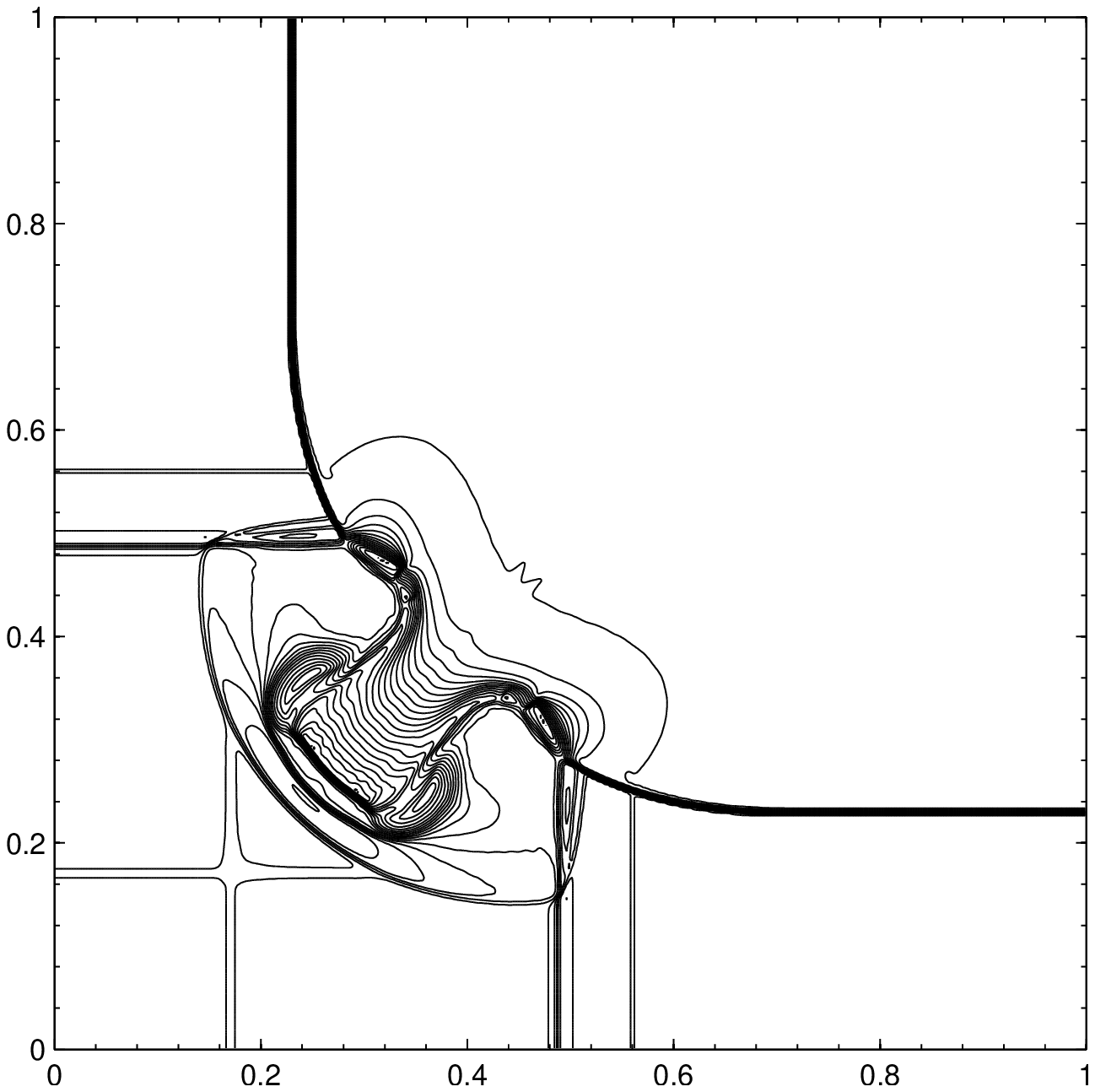}}
  \subfigure[${\tt PCPFDWENO9}$]
  {\includegraphics[width=0.48\textwidth]{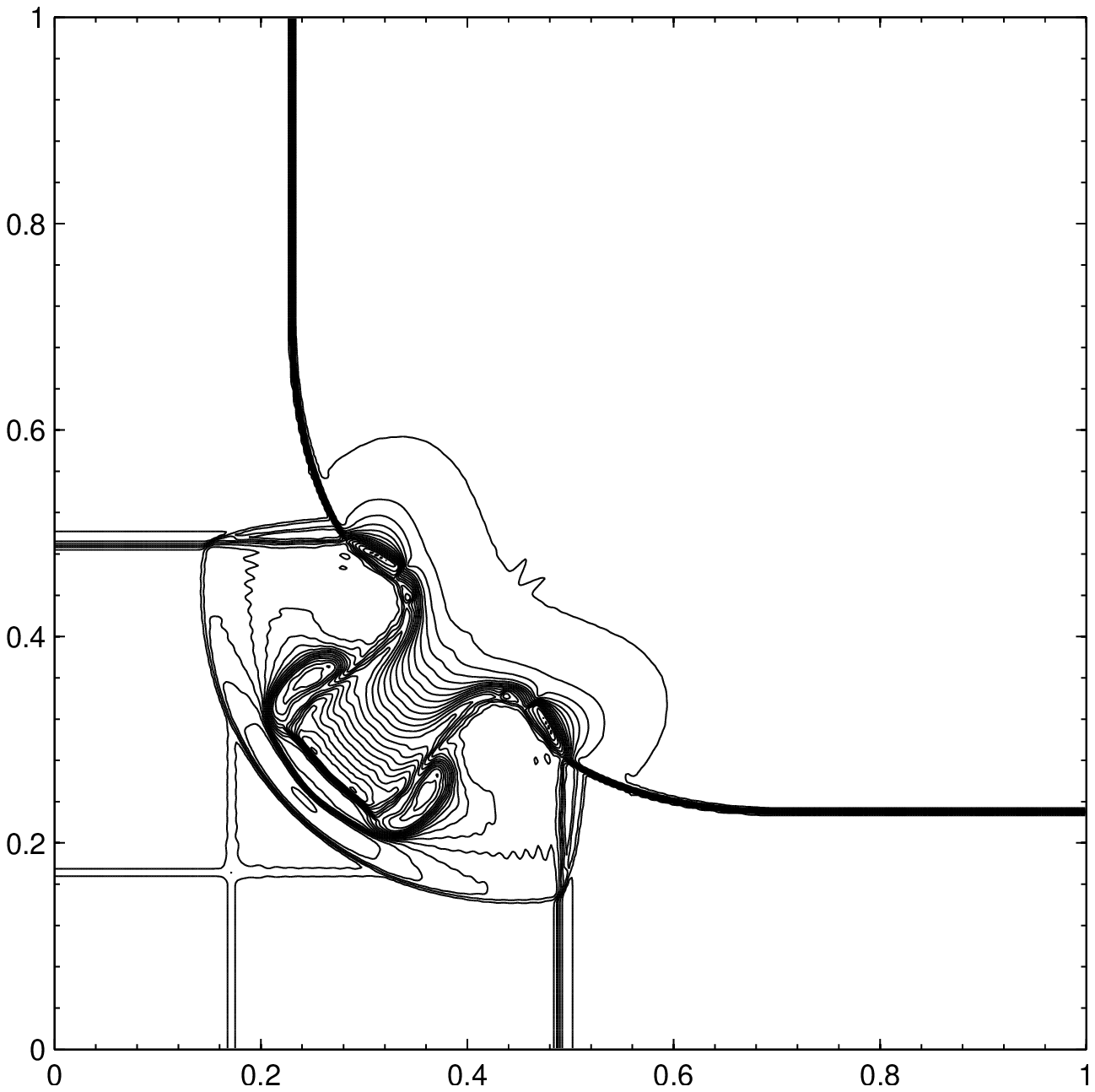}}
  \caption{\small The second 2D RP in Example \ref{example2DRPs}:
  The contours of the density
       logarithm $\ln \rho$ at $t=0.4$ within the domain $[0,1] \times [0,1]$ obtained by using {\tt PCPFDWENO5} and {\tt PCPFDWENO9}
        with $400 \times 400$ uniform cells (25 equally spaced contour lines from $-8$ to $-2.3$).
 }
  \label{fig:2DRP2}
\end{figure}

\begin{figure}[htbp]
  \centering
  {\includegraphics[width=0.48\textwidth]{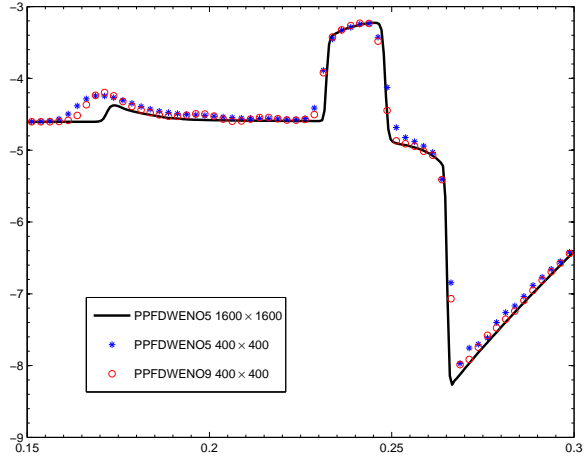}}
  \caption{\small
  Same as Fig. \ref{fig:2DRP2}, except for
      $\ln\rho$ along the line $y=x$ within  the scaled interval $[0,1]$.
 }
  \label{fig:2DRP2_Comparison}
\end{figure}

Fig.~\ref{fig:2DRP2} displays the contours of the density logarithm $\ln \rho$ at time $t = 0.4$ within
the unit square $[0,1] \times [0,1]$ obtained by using {\tt PCPFDWENO5} and {\tt PCPFDWENO9} with the mesh of $400 \times 400$ uniform cells.
It is seen that the
interaction of four initial discontinuities leads to the distortion of
both initial shock wave and the formation of a ``mushroom cloud'' starting from the point $(0.5,0.5)$ and expanding to the left bottom region, and
 the present schemes have good performance and strong robustness in such ultra-relativistic flow simulation, in which the fluid velocity reachs  0.9998458 locally, or the Lorentz factor may be not lesser than 56.95.
Plots of  $\ln \rho$ along the line $y=x$  in Fig. \ref{fig:2DRP2_Comparison}
further show that {\tt PCPFDWENO9} captures the structure better than  {\tt PCPFDWENO5}.

\end{example}

\begin{example}[Forward facing step problem] \label{exampleFFStep}\rm
The forward facing step problem was first introduced by Emery \cite{Emery1968},  has been
widely used to test the non-relativistic hydrodynamic codes, e.g. \cite{ChenTang2008}, and  extended to the ideal relativistic fluid, see \cite{Lucas-Serrano2004,ZhangMacfadyen2006}.

The same setup as in \cite{ZhangMacfadyen2006} is used here.
The wind tunnel is located in the domain $[0,3]\times [0,1]$
and contains a forward facing step with a height of $0.2$, which
starts from $x=0.6$ and continues along the length of the tunnel.
Initially it is filled with an ideal gas with the density of 1.4,
the velocity of 0.999, the Mach number of 3, and
the adiabatic index of $1.4$. 
The reflective boundary conditions are specified on the walls of the tunnel, while the inflow and outflow boundary conditions are specified at two ends of the tunnel ($x=0$ and 3).
Due to the forward facing step, the bow shock wave is formed and
then a Mach reflection happens at the top wall of the tunnel.
Later, the reflected shock wave incidents to the step and then
the regular reflection of the shock wave is caused and generates a second
reflected curved shock wave. Moreover, the step corner (0.6,0.2) is the center of a rarefaction fan and hence  a singular point of the flow.
%
Although the time-evolution of the flow is similar to the non-relativistic case, see e.g. \cite{ChenTang2008},  the present test is much more difficult than
the non-relativistic case because the bow shock wave
has a much fast speed and  there exist the large jumps
in the density and the pressure and the ultra-relativistic regime
where the Lorentz factor $W\gg 1$.

\begin{figure}[htbp]
  \centering
  {\includegraphics[width=0.75\textwidth]{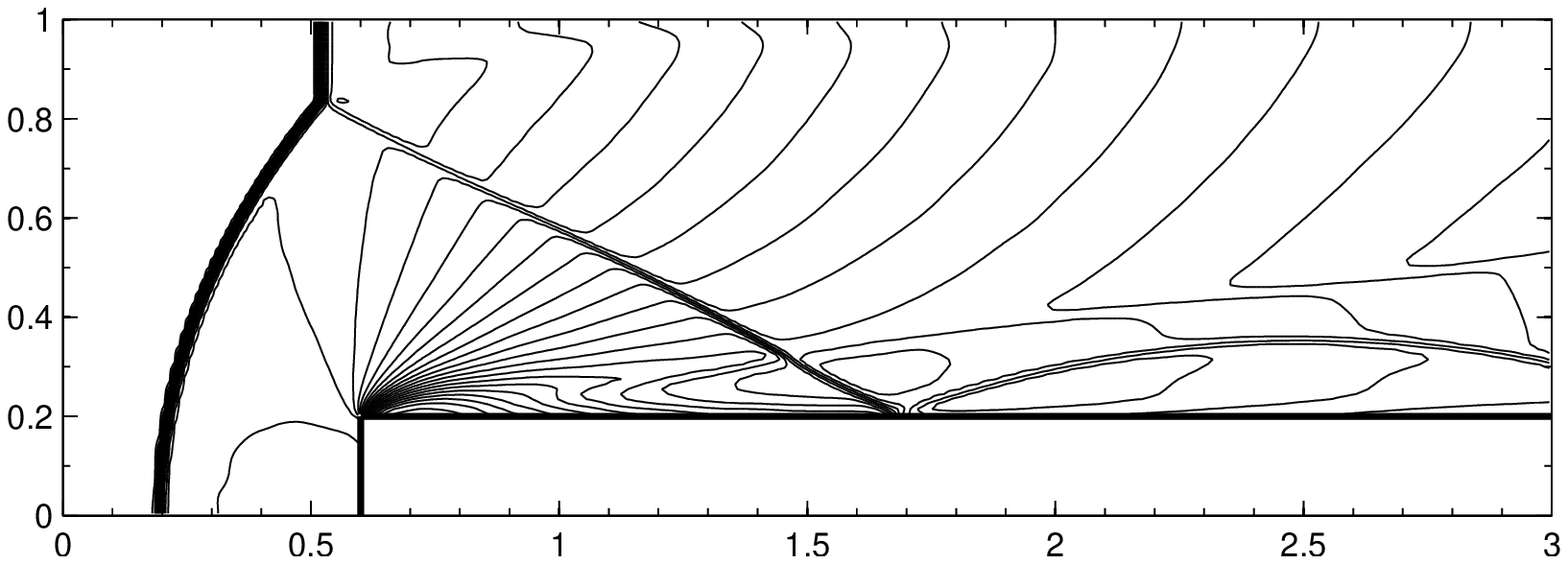}}
  {\includegraphics[width=0.75\textwidth]{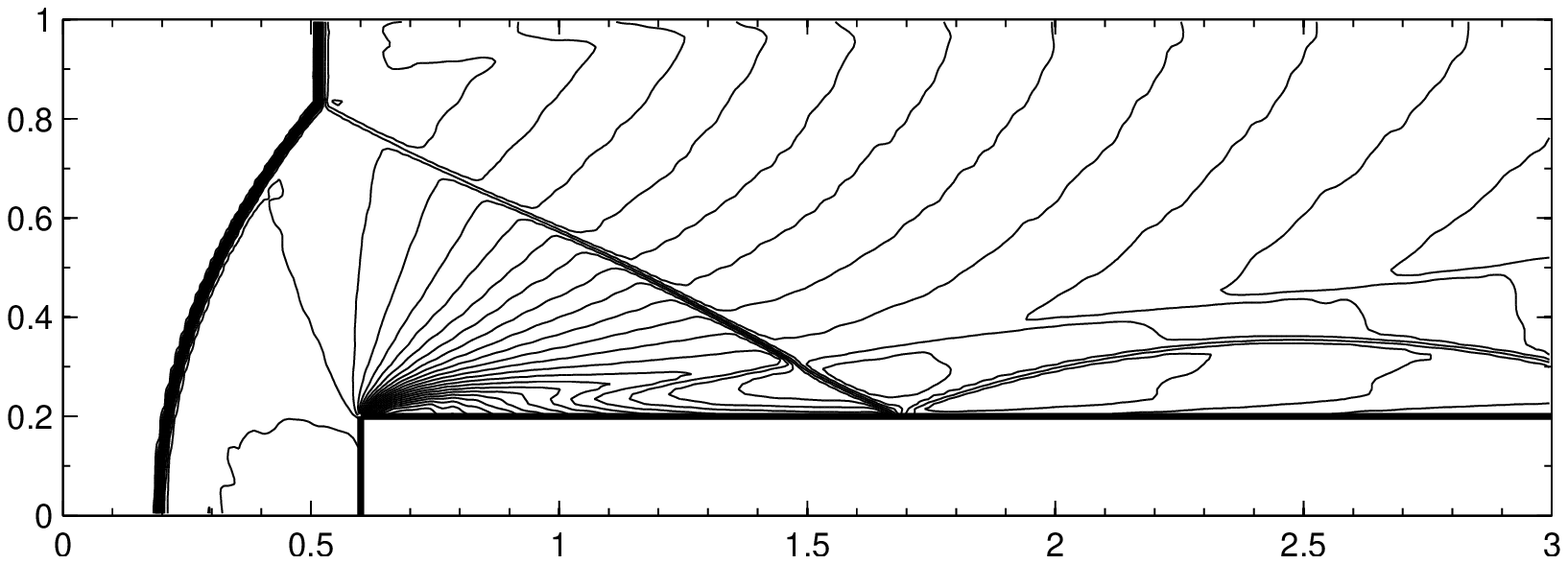}}
  \caption{\small Example \ref{exampleFFStep}:
  The contours of the rest-mass density
       logarithm $\ln \rho$ at $t=4$ obtained by using {\tt PCPFDWENO5} (top) and {\tt PCPFDWENO9} (bottom)
        with  $300 \times 100$ uniform cells for the domain $[0,3]\times [0,1]$
        (25 equally spaced contour lines from -0.86 to 4.64).
 }
  \label{fig:step}
\end{figure}

Fig. \ref{fig:step} gives the contours of the rest-mass density logarithm $\ln \rho$ at $t=4$ obtained by using the proposed {\tt PCPFDWENO5}
and {\tt PCPFDWENO9}  with  $300\times 100$ uniform cells for the domain $[0,3]\times [0,1]$.
Numerical results also show that across the Mach stem, the relative jumps   $\Delta\rho ={|\rho_R/\rho_L - 1|}$ and $\Delta p ={|p_R/p_L - 1|}$ are about 61.33 and  5223.99,
respectively, where the subscripts $L$ and $R$ denote corresponding left and right states.
Comparing our results with those in \cite{ZhangMacfadyen2006}, it can be
seen that the flow structures are well obtained
and the discontinuities, e.g. the bow and reflected shock waves and the Mach stem, are  well captured
with high resolution by using {\tt PCPFDWENO5} and {\tt PCPFDWENO9} without
 any artifical entropy fix around the step corner.

\end{example}

\begin{example}[Axisymmetric relativistic jets] \label{exampleJet}\rm
The last 2D example is to simulate two high-speed relativistic jet flows by solving
the axisymmetric RHD equations \eqref{eq:2Daxis}. The jet flows are ubiquitous in extragalactic radio sources associated with active galactic nuclei and the most compelling case for a special relativistic phenomenon.
Since there are the ultra-relativistic region,
 strong relativistic shock wave and shear flow, and interface instabilities etc.
 in the high-speed jet flow, simulating successfully such jet flow can be a real challenge,
 see e.g. \cite{Marti1994,Duncan:1994,Marti1997,Komissarov1998}.

\begin{figure}[htbp]
  \centering
  {\includegraphics[width=0.96\textwidth]{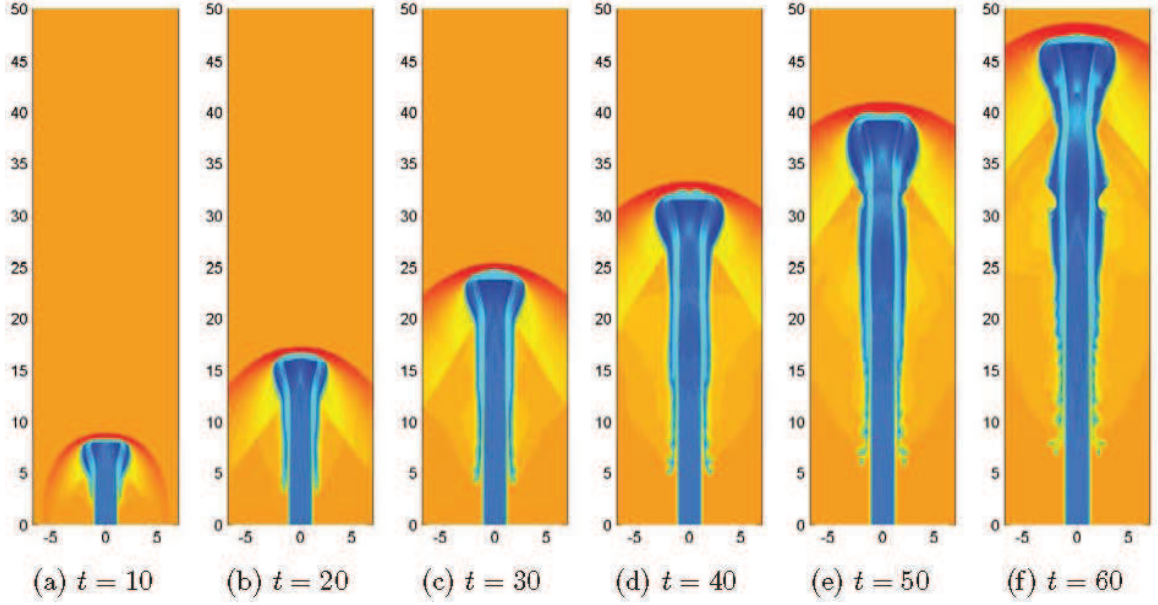}}
  \caption{\small The hot A1 model in Example \ref{exampleJet}:
  Schlieren images of the rest-mass density logarithm $\ln \rho$
  within the symmetrical domain $[-7,7]\times [0,50]$ at several different times obtained
        by using {\tt PCPFDWENO5}  with $280 \times 2000$ uniform cells.
 }
  \label{fig:jetA1}
\end{figure}

\begin{figure}[htbp]
  \centering
  {\includegraphics[width=0.96\textwidth]{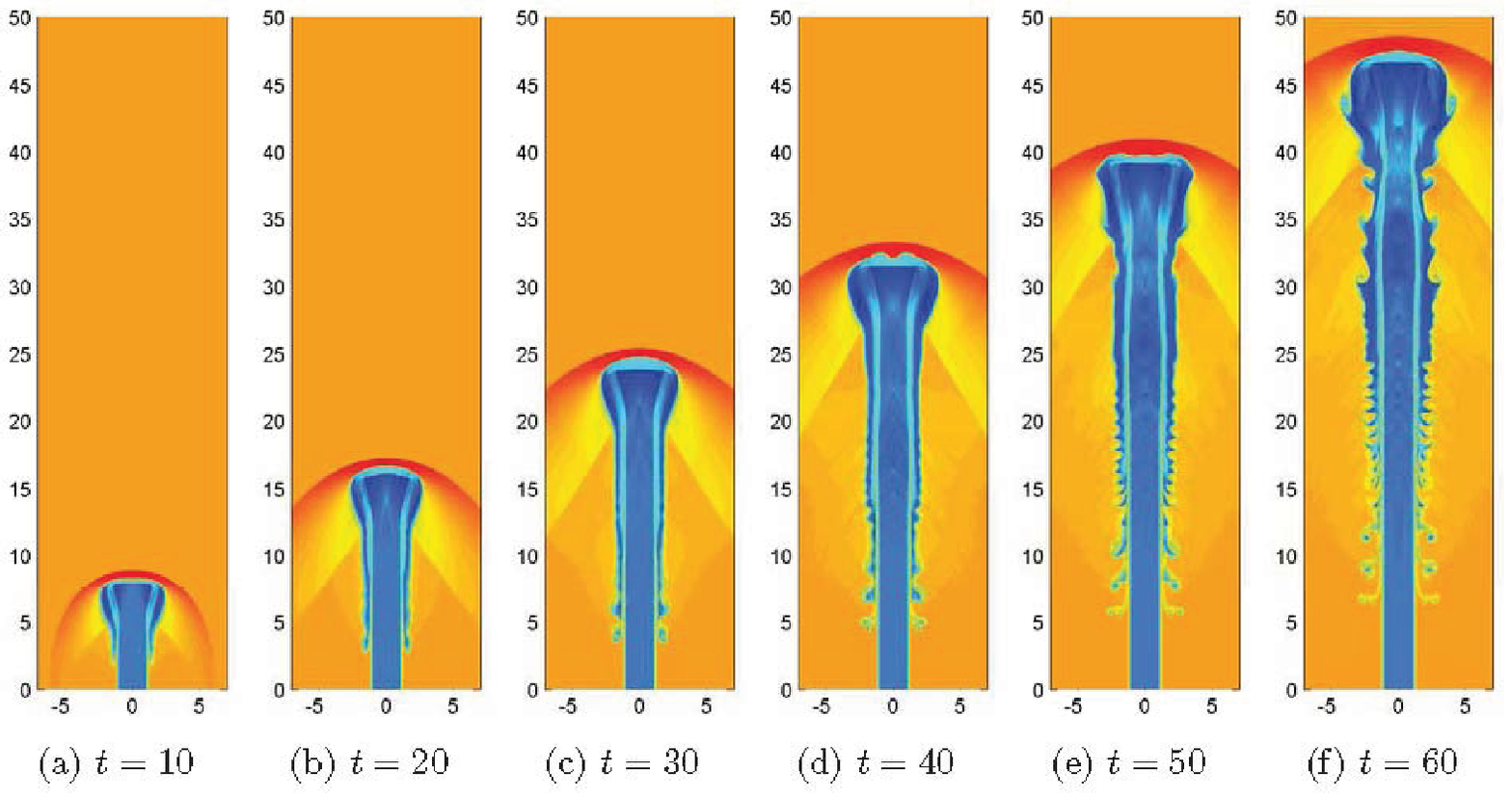}}
  \caption{\small Same as Fig. \ref{fig:jetA1} except for {\tt PCPFDWENO9}.
 }
  \label{fig:jetA1-b}
\end{figure}

The first test is a pressure-matched  hot A1 model,
see \cite{Marti1997}. In this model, the classical beam Mach number $M_b$
is near  the minimum Mach number $M^{\min}=v_b/\sqrt{\Gamma-1}$, the beam is
moving at speed $v_b=0.99$,
and relativistic effects from large beam internal energies
are important and comparable to the effects from fluid velocity
near the speed of light.
Initially, the computational domain $[0,7]\times [0,50]$ in the $(r,z)$ plane is filled with a static uniform medium with unit rest-mass density
with the adiabatic index of $4/3$.
A light relativistic jet is
injected in the $z$--direction through
the inlet part ($r \le 1$) of
the bottom boundary ($z=0$) with a density of 0.01,
a pressure equal to the ambient pressure, and a speed of $v_b$ or
a Lorentz factor of 7.09.
The relativistic Mach number $M_r :=M_b W /W_s$ is
about 9.97, where  $W_s=1/\sqrt{1-c_s^2}$  is the Lorentz factor
associated with the local sound speed and $M_b=v_b/c_s$ is  equal to 1.72.
The symmetrical boundary condition is specified at
 $r=0$,
 the fixed inflow beam condition is specified on the nozzle $\{z=0, r\le 1$\},
while outflow boundary conditions are  on other boundaries.

Figs. \ref{fig:jetA1} and \ref{fig:jetA1-b} display the schlieren images of the rest-mass density logarithm $\ln \rho$
within the domain $[-7,7]\times [0,50]$ at $t=10,20,30,40,50$ and 60 obtained  by using {\tt PCPFDWENO5} and {\tt PCPFDWENO9}  with $280 \times 2000$ uniform cells respectively. Compared them to those in \cite{Marti1997}, it is found that
 the time evolution of a light, relativistic 
 jet with large internal energy is well simulated by our schemes, and
 the Mach shock wave at the jet head is correctly and well captured during the whole simulation.
Although the internal structure is almost completely lacked because of the pressure equilibrium between the beam and its surroundings, and the beam/cocoon interface of the hot jets is very stable against the growth of pinch instabilities that would evolve into internal shock waves, the proposed schemes  still
clearly resolve the beam/cocoon interface and the Kelvin-Helmholtz type instability
at the beam/cocoon interface.





\begin{figure}[htbp]
  \centering
  {\includegraphics[width=0.8\textwidth]{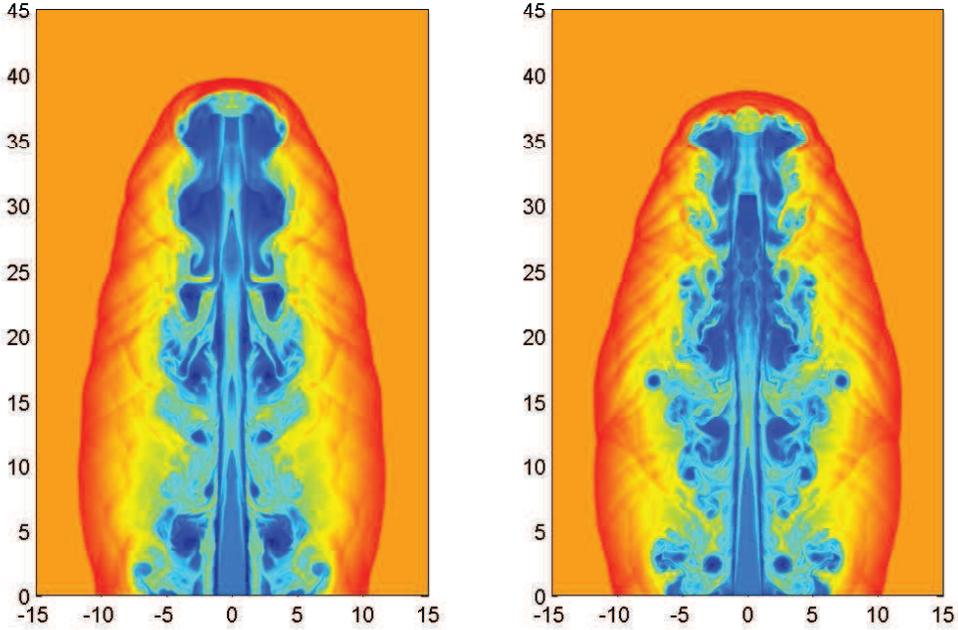}}
  \caption{\small The highly supersonic jet test of Example \ref{exampleJet}:
  Schlieren images of the rest-mass density logarithm $\ln \rho$ at $t=100$ obtained by {\tt PCPFDWENO5} (left) and {\tt PCPFDWENO9} (right)
        on the mesh of $384 \times 1152$ uniform cells.
 }
  \label{fig:jetC2}
\end{figure}


The second test is the  pressure-matched highly supersonic C2 jet model
($M_b \gg M^{\min}$, $\Gamma=5/3$, and $v_b = 0.99$), see  \cite{Marti1997,ZhangMacfadyen2006}.
Highly supersonic jets are also refered to  cold models and
 the relativistic effects from large beam speeds dominate in C2 jet model
so that there exist important differences between hot and cold relativistic jets.

The same setup as in \cite{ZhangMacfadyen2006} is used here
within the computational domain $[0,15]\times [0,45]$ in the $(r,z)$ plane.
The initial relativistic jet has the same rest-mass density and velocity  as those
in the above hot A1 model, but a higher Mach number of 6  (or corresponding relativistic Mach number of about 41.95) and a larger adiabatic index of $5/3$.

Fig. \ref{fig:jetC2} shows the schlieren images of the rest-mass density logarithm $\ln \rho$ within the symmetrical domain $[-15,15]\times [0,45]$ at $t=100$
 obtained by using {\tt PCPFDWENO5} and {\tt PCPFDWENO9}
with $384 \times 1152$ uniform cells. In our simulations, stable cocoons can be found in the early stages of the evolution, but these cocoons eventually evolve into large vortices producing turbulent structures, and the morphology and dynamics of the relativistic jets
in Fig. \ref{fig:jetC2} agree well with those obtained by using an
adaptive mesh refinement RHD code in \cite{ZhangMacfadyen2006}.
It is worth noting that the above-observed gross morphological features already
found in non-relativistic calculations.

\end{example}

\section{Conclusions}
\label{sec:conclude}

The paper developed  ($2r-1$)th-order accurate  finite difference WENO schemes for special RHD equations and proved that their solutions satisfied
physical properties: the positivity of the rest-mass density and the pressure and the bounds of the velocity. The key contributions were that the convexity and some mathematical properties of the admissible state set ${\mathcal G}$ were proved and  the concave function $q(\vec U)$ with respect to the conservative
vector $\vec U$ was  discovered.
The schemes were
built on the local Lax-Friedrich splitting,
the WENO reconstruction, the physical-constraints-preserving
flux limiter, and the high-order strong stability preserving time discretization.
They were considered as formal extensions of   positivity-preserving finite difference WENO schemes for the non-relativistic Euler equations  \cite{Hu2013}.
However,
developing physical-constraints-preserving methods  for the RHD system became much more difficult than the non-relativistic case
because of  the strongly coupling between the RHD equations,
no explicit expressions of
 the primitive variables $\vec V$ and the flux vector $\vec F_i$
in terms  of the conservative vector $\vec U$,
and one more physical constraint for the fluid velocity
in addition to the positivity of the rest-mass density and the pressure,
the non concavity of $p(\vec U)$, which was a key ingredient to enforce the positivity-preserving property for
the non-relativistic Euler equations. 


Several   numerical examples  demonstrated   accuracy, robustness, and effectiveness of the proposed physical-constraints-preserving schemes in solving
relativistic  problems with large Lorentz factor, or strong discontinuities, or low
rest-mass density or pressure etc., which involved
a 1D smooth problem, a 1D Riemann problem, a 1D blast wave interaction
problem, and a 1D shock heating problem, two
2D Riemann problems, and a forward facing step problem and two
axisymmetric jet flows in two dimensions.

Since the present physical-constraints-preserving limiting procedure  was independent on the reconstruction in space,
other high-order accurate reconstruction or interpolation could also be used to
replace the WENO reconstruction.
Moreover, it was possible that
the analysis results and the limiting procedure etc. could be extended
to develop high-order accurate physical-constraints-preserving finite volume schemes
for the RHD equations.


\section*{Acknowledgements}

This work was partially supported by
the National Natural Science Foundation
of China (Nos.  91330205  \& 11421101).
The authors would also like to thank
the referees for many useful suggestions.

\if false

\begin{appendices}


\section{Derivation of  }
\label{sec:AppendixA}

\end{appendices}

\fi


\begin{thebibliography}{99}





\bibitem{Balsara:1994}
D.S. Balsara,
Riemann solver for relativistic hydrodynamics,
{\em J. Comput. Phys.}, 114 (1994), 284-297.


\bibitem{balsara2000}
D.S. Balsara and C.-W. Shu,
Monotonicity preserving weighted essentially non-oscillatory schemes with increasingly
high-order of accuracy,  {\em J. Comput. Phys.}, 160 (2000), 405-452.

\bibitem{Blandford1976}
R.D. Blandford and C.F. McKee, Fluid dynamics of relativistic blast waves,
{\em Phys. Fluids}, 19 (1976), 1130-1138.


\bibitem{ChenTang2008}
G.X. Chen, H.Z. Tang and P.W. Zhang,
Second-order accurate Godunov scheme for multicomponent flows on moving triangular meshes,
{\em J. Sci. Comput.}, 34 (2008), 64-86.


\bibitem{Cheng2014}
J. Cheng and C.-W. Shu, Positivity-preserving Lagrangian scheme for multi-material compressible flow,
{\em J. Comput. Phys.}, 257 (2014), 143-168.




\bibitem{Christlieb}
A.J. Christlieb, Y. Liu, Q. Tang, and Z.F. Xu,
Positivity-preserving finite difference WENO schemes with constrained transport for ideal magnetohydrodynamic equations, arXiv:1406.5098v2,
2014.


\bibitem{Cissoko1992}
M. Cissoko,
Detonation waves in relativistic hydrotlynamics,
{\em Phys. Rev. D}, 45(1992), 1045-1052.

\bibitem{DaiWood:1997}
W.L.~Dai and P.R. Woodward,
An iterative {Riemann} solver for relativistic hydrodynamics,
{\em SIAM J. Sci.  Stat. Comput.}, 18 (1997), 982-995.

%





\bibitem{DolezalWong:1995}
A.~Dolezal and S.S.M. Wong,
Relativistic hydrodynamics and essentially non-oscillatory shock
  capturing schemes,
{\em J. Comput. Phys.}, 120 (1995), 266-277.

\bibitem{DonatFont:1998}
R.~Donat, J.A. Font, J.M. Ib{\'a}{\~n}ez, and A.~Marquina,
A flux-split algorithm applied to relativistic flows,
{\em J. Comput. Phys.}, 146 (1998), 58-81.



\bibitem{Duncan:1994}
G.C. Duncan and P.A. Hughes,
Simulations of relativistic extragalactic jets,
{\em Astrophys. J.}, 436 (1994), L119--L122.

\bibitem{Emery1968}
A.F. Emery,
An evaluation of several differencing methods for inviscid fluid flow problem,
{\em J. Comput. Phys.}, 2 (1968), 306-331.


\bibitem{EulderinkMel:1995}
F.~Eulderink and G.~Mellema,
 General relativistic hydrodynamics with a {Roe} solver,
{\em Astron. Astrophys. Suppl. S.}, 110 (1995), 587-623.


\bibitem{Font2008}
J.A. Font,
Numerical hydrodynamics and magnetohydrodynamics in general relativity,
{\em Living Rev. Relativity}, 11 (2008), 7.



\bibitem{Gerolymos2009}
G.A. Gerolymos, D. S\'en\'echal and I. Vallet,
Very-high-order WENO schemes,
{\em J. Comput. Phys.}, 228 (2009), 8481-8524.



\bibitem{Gottlieb2009}
S. Gottlieb, D.J. Ketcheson and C.-W. Shu, High order strong stability
preserving time discretizations,  {\em J. Sci. Comput.}, 38 (2009), 251-289.

\bibitem{HanLiTang2011}
E. Han, J.Q. Li and H.Z. Tang, Accuracy of the adaptive {GRP} scheme and the simulation of {2-D} Riemann problems for compressible {Euler} equations,
{\em Commun. Comput. Phys.}, 10 (3) (2011), 577-606.

  \bibitem{HePeng2011}
 P. He,  {\em Numerical Simulations of Relativistic} {\em Hydrodynamics
and Relativistic Magneto- hydrodynamics},
 Ph.D. thesis, School of Mathematical Sciences, Peking University, 2011.


  \bibitem{HeTang2011}
 P. He and H.Z. Tang,
 An adaptive moving mesh method for two-dimensional relativistic hydrodynamics,
 {\em Commun. Comput. Phys.}, 11 (2012), 114-146.

  \bibitem{HeTang2012}
P. He and H.Z. Tang, An adaptive moving mesh
 method for two-dimensional relativistic magnetohydrodynamics,
{\em Comput. \& Fluids}, 60 (2012), 1-20.


\bibitem{Hu2013}
X.Y. Hu, N.A. Adams and C.-W. Shu,
Positivity-preserving method for high-order conservative schemes solving compressible Euler equations,
{\em J. Comput. Phys.}, 242 (2013), 169-180.


\bibitem{Hughes2002}
P.A. Hughes, M.A. Miller and G.C. Duncan, Three-dimensional hydrodynamic
simulations of relativistic extragalactic jets, {\em Astrophys. J.}, 572 (2002), 713-728.

\bibitem{Ibanez-Marti1999}
J.M$^{\rm a}$. Ib\'{a}\"{n}ez and J.M$^{\rm a}$. Mart\'{i},
Riemann solvers in relativistic astrophysics,
{\em J. Comput. Appl. Math.},
109 (1999), 173-211.

\bibitem{jiang1996}
G.-S. Jiang and C.-W. Shu, Efficient implementation of weighted ENO
schemes,  {\em J. Comput. Phys.}, 126 (1996), 202-228.


\bibitem{JiangXu2013}
Y. Jiang and Z.F. Xu,
Parametrized maximum principle preserving limiter for finite
difference WENO schemes solving convection-dominated diffusion equations,
{\em SIAM J. Sci. Comput.}, 35 (2013), A2524-A2553.






\bibitem{Komissarov1998}
S.S. Komissarov and S.A.E.G. Falle,
The large-scale structure of FR-II radio sources,
{\em  Mon. Not. R. Astron. Soc.}, 297 (1998), 1087-1108.



\bibitem{Kunik2004}
M. Kunik, S. Qamar, and G. Warnecke,
Kinetic schemes for the relativistic gas dynamics,
{\em Numer. Math.}, 97 (2004), 159-191.








\bibitem{LaxLiu1998}
P.D. Lax and X.D. Liu, Solution of two-dimensional Riemann problems of gas dynamics by positive schemes,
{\em SIAM J. Sci. Comput.}, 19 (1998), 319-340.


\bibitem{Liang2014}
C. Liang and Z.F. Xu,
Parametrized maximum principle preserving flux limiters for high-order schemes
solving multi-dimensional scalar hyperbolic conservation laws,
{\em J. Sci. Comput.}, 58 (2014), 41-60.











\bibitem{Lucas-Serrano2004}
A. Lucas-Serrano, J. A. Font, J. M. Ib\'a\~nez and and J.M. Mart\'i,
Assessment of a high-resolution central scheme for the solution of the relativistic hydrodynamics equations,
{\em  Astron. Astrophys.}, 428 (2004), 703-715.

\bibitem{May-White1966}
M.M. May and R.H. White,
Hydrodynamic calculations of general-relativistic collapse,
{\em Phys. Rev.}, 141 (1966), 1232-1241.

\bibitem{May-White1967}
M.M. May and R.H. White,
Stellar dynamics and gravitational collapse,
{\em Methods Comput.
Phys.}, 7 (1967), 219-258.

\bibitem{Marti3}
J.M. Mart\'i and E. M\"uller,
Extension of the
  piecewise parabolic method to one-dimensional relativistic
  hydrodynamics, {\em J. Comput. Phys.}, 123 (1996), 1-14.

\bibitem{MME:2003}
J.M. Mart{\'\i} and E.~M{\"u}ller,
Numerical hydrodynamics in special relativity,
{\em Living Rev. Relativity}, 6 (2003), 7.

\bibitem{Marti1997}
J.M. Mart\'i, E. M\"uller, J. A. Font, J. M. Ib\'a\~nez and A. Marquina,
Morphology and dynamics of relativistic jets,
{\em Astrophys. J.}, 479 (1997), 151-163.



\bibitem{Marti1994}
J.M. Mart\'i, E. M\"uller and J. M. Ib\'a\~nez,
Hydrodynamical simulations of relativistic jets,
{\em Astron. Astrophys.}, 281 (1994), L9-L12.



%



\bibitem{MignoneBodo:2005}
A.~Mignone and G.~Bodo,
An {HLLC} {Riemann} solver for relativistic flows, {I}:
  hydrodynamics,
{\em Mon. Not. R. Astron. Soc.}, 364 (2005), 126-136.


\bibitem{Mignoneetal:2005}
A.~Mignone, T.~Plewa, and G.~Bodo,
 The piecewise parabolic method for multidimensional relativistic  fluid dynamics,
 {\em Astrophys. J. Suppl. S.}, 160 (2005), 199-219.





  \bibitem{Qamar2004}
S. Qamar and G. Warnecke,
A high-order kinetic flux-splitting method for the relativistic magnetohydrodynamics,
{\em J. Comput. Phys.}, 205 (2005), 182-204.


\bibitem{Qamar2012}
S. Qamar and M. Yousaf,
The space-time {CESE} method for solving special relativistic hydrodynamic equations,
{\em J. Comput. Phys.}, 231 (2012), 3928-3945.


\bibitem{Qiug2011}
J.M. Qiu and C.-W. Shu, Positivity preserving semi-Lagrangian discontinuous Galerkin formulation: theoretical analysis and application to the Vlasov-Poisson system,
{\em J. Comput. Phys.}, 230 (2011), 8386-8409.




\bibitem{Radice2011}
D. Radice and L. Rezzolla,
Discontinuous Galerkin methods for general-relativistic hydrodynamics: Formulation
and application to spherically symmetric spacetimes, {\em Phys. Rev. D},  84 (2011), 024010.






\bibitem{Schneider:1993}
V.~Schneider, U.~Katscher, D.H. Rischke, B.~Waldhauser, J.A. Maruhn  and C.D.
  Munz, New algorithms for ultra-relativistic numerical hydrodynamics,
 {\em J. Comput. Phys.}, 105 (1993), 92-107.

\bibitem{ShuSIReV2009}
C.-W. Shu,
High Order Weighted essentially nonoscillatory schemes for convection dominated problems,
{\em SIAM Rev.}, 51 (2009), 82-126.

\bibitem{SchulzRinne1993}
C.W. Schulz-Rinne, J.P. Collins and H.M. Glaz, Numerical solution of the Riemann problem for two-dimensional gas dynamics,
{\em SIAM J. Sci. Comput.}, 14 (6)(1993), 1394-1414.

\bibitem{Tchekhovskoy2007}
A.~Tchekhovskoy, J.C. McKinney and R.~Narayan,
{WHAM}: a {WENO}-based general relativistic numerical scheme, {I}:
  hydrodynamics,
 {\em Mon. Not. R. Astron. Soc.}, 379 (2007), 469-497.


\bibitem{wang2012}
C. Wang, X.X. Zhang, C.-W. Shu and J.G. Ning, Robust high-order discontinuous Galerkin schemes for two-dimensional gaseous detonations,
{\em J. Comput. Phys.}, 231 (2012), 653-665.


\bibitem{Wilson:1972}
J.R. Wilson, Numerical study of fluid flow in a Kerr
  space, {\em Astrophys. J.}, 173 (1972), 431-438.


\bibitem{Wilson:1979}
J.R. Wilson and G.J. Mathews, {\em Relativistic Numerical Hydrodynamics},
Cambridge University Press, 2007.

\bibitem{WuTang2014}
K.L. Wu and H.Z. Tang, Finite volume local evolution Galerkin method for two-dimensional relativistic hydrodynamics,
{\em J. Comput. Phys.}, 256 (2014), 277-307.

\bibitem{WuYangTang2014}
K.L. Wu, Z.C. Yang and H.Z. Tang, A third-order accurate direct Eulerian GRP scheme for one-dimensional relativistic hydrodynamics,
{\em East Asian J. Appl. Math.}, 4 (2014), 95-131.

\bibitem{WuYangTang2014b}
K.L. Wu, Z.C. Yang, and H.Z. Tang, A third-order accurate direct Eulerian GRP scheme for the Euler equations in gas dynamics, {\em J. Comput. Phys.}, 264 (2014), 177-208.


\bibitem{Xing2010}
Y.L. Xing, X.X Zhang and C.-W. Shu, Positivity-preserving high-order well-balanced discontinuous Galerkin methods for the shallow water equations,
{\em  Adv. Water Resour.}, 
33 (2010), 1476-1493.








\bibitem{XiongQiuXu2014}
T. Xiong, J.-M. Qiu, and Z.F. Xu, Parametrized positivity preserving flux limiters
for the high order finite difference WENO scheme solving compressible Euler equations, arXiv:1403.0594.



\bibitem{XiongQiuXu2013}
T. Xiong, J.-M. Qiu, and Z.F. Xu, A parametrized maximum principle preserving flux limiter for finite
difference RK-WENO schemes with applications in incompressible flows,
{\em J. Comput. Phys.}, 252 (2013), 310-331.



\bibitem{Xu_MC2013}
Z.F. Xu,
Parametrized maximum principle preserving flux limiters
for high order schemes solving hyperbolic conservation laws: one-dimensional scalar problem,
{\em Math. Comput.}, 83 (2014), 2213-2238.






\bibitem{YangBeam1997}
J.Y. Yang, M.H. Chen, I.N. Tsai, and J.W. Chang,
A kinetic beam scheme for relativistic gas dynamics,
{\em J. Comput. Phys.}, 136 (1997), 19-40.

\bibitem{YangHeTang2011}
Z.C. Yang, P. He, and H.Z. Tang,
A direct Eulerian GRP scheme for relativistic hydrodynamics: one-dimensional case,
{\em J. Comput. Phys.}, 230 (2011),7964-7987.

\bibitem{YangTang2012}
Z.C. Yang and H.Z. Tang, A direct Eulerian GRP scheme for relativistic hydrodynamics: two-dimensional case,
{\em J. Comput. Phys.}, 231 (2012), 2116-2139.


\bibitem{ZannaBucciantini:2002}
L.D. Zanna and N.~Bucciantini,
 An efficient shock-capturing central-type scheme for multidimensional
  relativistic flows, {I}: hydrodynamics,
 {\em Astron. Astrophys.}, 390 (2002), 1177-1186.


\bibitem{ZhangMacfadyen2006}
W.Q. Zhang and A.I. Macfadyen, RAM: a relativistic
adpative mesh refinement hydrodynamics code, {\em Astrophys. J. Suppl. S.}, 164 (2006), 255-279.







\bibitem{zhang2010}
X.X. Zhang and C.-W. Shu, On maximum-principle-satisfying high-order schemes for scalar conservation laws,
{\em J. Comput. Phys.}, 229 (2010), 3091-3120.

\bibitem{zhang2010b}
X.X. Zhang and C.-W. Shu, On positivity-preserving high-order discontinuous {Galerkin} schemes for compressible Euler equations on rectangular meshes,
{\em J. Comput. Phys.}, 229 (2010), 8918-8934.




\bibitem{zhang2011}
X.X. Zhang and C.-W. Shu, Positivity-preserving high-order discontinuous Galerkin schemes for compressible Euler equations with source terms,
{\em J. Comput. Phys.}, 230 (2011), 1238-1248.

\bibitem{zhang2011b}
X.X. Zhang and C.-W. Shu, Maximum-principle-satisfying and positivity-preserving high-order schemes for conservation laws: survey and new developments,
{\em Proc. R. Soc. A},    
 467 (2011), 2752-2776.


\bibitem{zhang2012}
X.X. Zhang and C.-W. Shu, Positivity-preserving high-order finite difference WENO schemes for compressible Euler equations,
{\em J. Comput. Phys.}, 231 (2012), 2245-2258.


\bibitem{zhang2012a}
X.X. Zhang, Y.H. Xia and C.-W. Shu, Maximum-principle-satisfying and positivity-preserving high-order discontinuous galerkin schemes
for conservation laws on triangular meshes,
{\em J. Sci. Comput.}, 50 (2012), 29-62.




\bibitem{ZhangZheng1990}
T. Zhang and Y.X. Zheng,
Conjecture on the structure of solutions of the Riemann problem for two-dimensional gas dynamics systems,
{\em SIAM J. Math. Anal.}, 21 (1990), 593-630.








\bibitem{ZhaoHeTang2014}
J. Zhao, P. He and H.Z. Tang,
Steger-Warming flux vector splitting method for special relativistic hydrodynamics,
{\em Math. Meth. Appl. Sci.}, 37 (2014), 1003-1018.




\bibitem{ZhaoTang2013}
J. Zhao and H.Z. Tang,
Runge-Kutta discontinuous Galerkin methods with WENO limiter for the special relativistic
hydrodynamics,   {\em J. Comput. Phys.}, 242 (2013), 138-168. 















\end{thebibliography}
\end{document}